\documentclass[3p,preprint,review,authoryear]{elsarticle}
\usepackage{setspace}
\usepackage{graphicx} 
\usepackage{svg}
\usepackage{multirow}
\usepackage{makecell}
\usepackage{xcolor}
\usepackage{amsmath}
\usepackage{amssymb}
\usepackage{adjustbox}
\usepackage{amsthm}
    \newtheorem{proposition}{Proposition}
    \newtheorem{corollary}{Corollary}
\usepackage{todonotes}
    \presetkeys{todonotes}{inline}{}
\usepackage{xcolor}
\usepackage{xspace}
\usepackage[hidelinks]{hyperref}
\usepackage[capitalise, noabbrev]{cleveref}
\usepackage{tabularx}
\usepackage{enumitem}
\usepackage{algorithm}
\usepackage{algpseudocode}
    \hypersetup{unicode}
    \hypersetup{breaklinks=true}
\usepackage{csquotes}
\usepackage{soul}
\usepackage{booktabs}
\usepackage{subcaption}
\usepackage{siunitx}
\usepackage{comment}
\usepackage{tikz}
\usetikzlibrary{shapes,arrows,positioning,graphs,patterns,patterns.meta}
\usetikzlibrary{backgrounds, arrows.meta}
\usetikzlibrary{calc}
\usepackage{etoolbox} 

\numberwithin{equation}{section}

\usepackage{pgfplotstable}

\usepackage{pgfplots}
\usepgfplotslibrary{statistics}
\usepgfplotslibrary{groupplots}
\pgfplotsset{compat=1.18}

\usepgfplotslibrary{fillbetween}
\usetikzlibrary{positioning,fadings}

\newcommand{\Interval}[1]{#1}

\newcommand{\IdxInterval}{i}
\newcommand{\IdxState}{s}
\newcommand{\StateOff}{\textup{\texttt{off}}}
\newcommand{\StateIdle}{\textup{\texttt{idle}}}
\newcommand{\StateProc}{\textup{\texttt{proc}}}

\newcommand{\LBBDG}{\texttt{LBBD-opt}\xspace}
\newcommand{\LBBD}{\texttt{LBBD-heur}\xspace}
\newcommand{\Freecp}{\texttt{freecp}\xspace}
\newcommand{\ILPF}{\texttt{ILP-fsws}\xspace}
\newcommand{\ILPC}{\texttt{ILP-cpws}\xspace}
\newcommand{\ILP}{\texttt{ILP}\xspace}

\newcommand{\MD}{\operatorname{MD}}

\newcommand{\TransTimeSymbol}{T}
\DeclareDocumentCommand \TransTime {o} {
  \IfNoValueTF{#1} {
        \TransTimeSymbol
  }{
        \TransTimeSymbol(#1)
  }
}

\newcommand{\TransPowerSymbol}{P}
\DeclareDocumentCommand \TransPower {o} {
  \IfNoValueTF{#1} {
        \TransPowerSymbol
  }{
        \TransPowerSymbol(#1)
  }
}

\newcommand{\IdxAnother}[1]{#1^{\prime}} %
\newcommand{\m}[1]{\boldsymbol{#1}}

\journal{European Journal of Operational Research}

\makeatletter
\def\ps@pprintTitle{%
  \let\@oddhead\@empty
  \let\@evenhead\@empty
  \def\@oddfoot{\reset@font\hfil\thepage\hfil}%
  \let\@evenfoot\@oddfoot
}
\makeatother

\begin{document}

\begin{frontmatter}

\title{Resource-Constrained Project Scheduling with Time-of-Use Energy Tariffs and Machine States: A Logic-Based Benders Decomposition approach}
\author[aff1]{Corentin Juvigny}
\ead{corentin.juvigny@cvut.cz}
\author[aff1]{Antonín Novák}
\ead{antonin.novak@cvut.cz}
\author[aff1]{Jan Mandík}
\author[aff1]{Zdeněk Hanzálek}
\ead{zdenek.hanzalek@cvut.cz}

\affiliation[aff1]{organization={Czech Institute of Informatics, Robotics and Cybernetics, Czech Technical University in Prague}, addressline={Jugoslávských partyzánů 1580/3}, postcode={160 00}, city={Prague 6}, country={Czech Republic}}

\begin{abstract}

In this paper, we investigate the \emph{Resource-Constrained Project Scheduling Problem} (RCPSP) with \emph{Time-of-Use}  (TOU) \emph{energy tariffs} and \emph{machine states}, a variant of RCPSP for production scheduling, where energy price is part of the criteria and one highly energy-demanding machine can be in one of the following three states: \StateProc, \StateIdle, or \StateOff. The problem involves scheduling all tasks, respecting precedence constraints and resource limitations, while minimizing the combination of the overall makespan and the Total Energy Cost (TEC), which varies according to the TOU tariffs, which can take negative values. We propose two novel approaches to solve it: a \emph{monolithic Constraint Programming} (CP) approach and a \emph{Logic-Based Benders Decomposition} (LBBD) approach. The latter combines a master problem handling the energy cost solved using \emph{Integer Linear Programming} (ILP) with a subproblem handling the RCPSP, resolved using CP. Both approaches outperform the \emph{monolithic compact} ILP counterpart, but the LBBD significantly outperforms the monolithic CP in most cases, especially when the makespan criterion is not included in the objective function, solving to optimality instances with up to 480 tasks.
Finally, we propose a way to generalize our LBBD approach to other problems sharing similar characteristics, and applied it to various problems, such as an RCPSP with blocking times \& total weighted tardiness criterion, or a flexible job shop.

\end{abstract}


\end{frontmatter}
\section{Introduction}


The industrial sector is a very energy-demanding area, accounting for around half of the total energy consumption of the world, a share that has never stopped growing in recent years \citep{fang2011new}.
The rising costs induced by energy production, combined with increasing concerns about the viability of intensified production as well as new public efforts fostering more durable production processes, have led an increasing number of studies to concentrate on energy-aware project scheduling \citep{du2021energy}.

\emph{Time-of-use} (TOU) tariffs have emerged as a widely adopted strategy among energy suppliers of industrialized nations. This dynamic pricing mechanism serves two primary purposes: it facilitates the balancing of daily energy consumption patterns and mitigates peak load concentrations during specific hours \citep{hung2018modeling,chen2019scheduling}. It consists of adjusting the electricity tariffs according to the fluctuation of demand, resulting in variable tariffs that can change quickly. This pricing structure provides consumers with significant opportunities to reduce the costs of their energy use by shifting their major consumption activities to off-peak periods \citep{geng2020bi}.

In industry, some energy-intensive operations, such as furnaces or industrial ovens, may benefit from a more efficient management strategy that involves maintaining a machine in an intermediate state during periods of inactivity, rather than completely shutting it down. Indeed, turning on a machine may consume a significant amount of energy and time.
However, the addition of these machine states leads to complex optimization problems~\citep{BENEDIKT2025}.
In real industrial processes, not all operations are energy-intensive. They may be part of a more global project scheduling, with precedence constraints between the processes and limited capacities on resources involved, a category of scheduling problems called \emph{Resource-Constrained Project Scheduling Problems} (RCPSPs) \citep{herroelen1998resource}.

To address these challenges, we study an RCPSP problem with TOU tariffs and machine states. The states of energy intensive resource are modeled in a transition diagram, indicating the duration of the transitions between states and their associated energy consumptions.
We assume that this energy-intensive resource has a unitary capacity. The other resources have integer capacities. Moreover, a task requiring this energy-intensive resource does not use other resources.
We consider two objective functions. The first minimizes the total energy cost (TEC) induced by energy-intensive tasks with respect to TOU tariffs, and can be denoted as $PS, 1\text{TOU}|\text{prec,states}|\text{TEC}$ in the extended Graham's three-field notation~\citep{GRAHAM1979287}. The second combines the previous TEC with a minimization of the makespan over all tasks, i.e., $PS, 1\text{TOU}|\text{prec,states}|\text{TEC},C_{\max}$.
The reader can refer to Table~A.1 of the supplementary materials for a summary of the acronyms used throughout this paper.

\subsection{Paper contributions and outline}

The main contributions of this article are the following.
\begin{itemize}[itemsep=0pt, parsep=0pt]
    \item We introduce a new problem combining the RCPSP with machine states and TOU tariffs. To model it, we detail two compact models, which we henceforth refer to as \emph{monolithic} models, one based on \emph{Constraint Programming} (CP) and one based on \emph{Integer Linear Programming} (ILP). We say monolithic because they model the problem as a whole and are not decomposed.
    \item We put forth a \emph{Logic-based Benders decomposition} (LBBD) approach to solve the problem more efficiently. Experiments outline to what extent it outperforms monolithic formulations, especially in larger instances or when the makespan criterion is not considered, i.e., for $PS, 1\text{TOU}|\text{prec,states}|\text{TEC}$.
    \item 
    We demonstrate the \emph{generality} of our LBBD approach, which segregates energy-related parameters within the master problem and employs a classical scheduling model as its subproblem.
    In particular, we apply the LBBD approach to problems based on an RCPSP with blocking times \& total weighted tardiness criterion ($PSm, 1\text{TOU}|\text{intree,states}|\text{TEC},\sum_jw_jT_j$), and a flexible job shop ($FJm,1\text{TOU}|\text{prec,states}|\text{TEC},C_{\max}$).
\end{itemize}

The article is organized as follows.
In Section~\ref{section:literature-review}, we review the current state of the art literature and delineate the position of our article among them.
In Section~\ref{section:problem_statement}, we present our hybrid energy-aware / RCPSP variant and propose an MILP formulation to solve it.
Then, we present a CP approach in Section~\ref{sec:constraint_programming} and a new LBBD-based approach in Section~\ref{section:LBBD}.
We assess and compare the proposed approaches in Section~\ref{section:experiments}.
Afterwards, Section~\ref{section:generalization-of-LBBD} investigates the use of our LBBD approach to other problems.
Finally, Section~\ref{section:conclusion} concludes the article and exposes future avenues of research.

\section{Literature Review}
\label{section:literature-review}

This section surveys the three research streams that converge in the present work: resource-constrained project scheduling, energy-aware scheduling under time-of-use tariffs with machine-states, and
logic-based Benders decomposition for scheduling.
For each stream, we identify the modeling choices and methodological limits that motivate our contribution, before making the relationship explicit at the end of the section.

\begin{table}[tb]
\centering
\renewcommand{\arraystretch}{1.25}
\begin{small}
\resizebox{\textwidth}{!}{%
\begin{tabular}{lcccll}
\toprule
\textbf{Reference} & \textbf{Machine env.} & \textbf{TOU} & \textbf{States} & \textbf{Objective} & \textbf{Method} \\
\midrule
\cite{maghsoudlou2021framework}    & RCPSP (multi-skill) & $\checkmark$ & $-$          & TEC                                                             & Metaheuristic        \\
\cite{pouramin2024multi}           & RCPSP (multi-mode)  & $\checkmark$ & $-$          & $C_{\max}$, cost               & ILP/Metaheuristic \\
\cite{gaggero2023exact}            & Parallel machines   & $\checkmark$ & $-$          & $C_{\max}$, TEC                                                 & MILP/heuristic     \\
\cite{XIECHEN2026107567}           & Unrel. parallel mach.\ (stoch.) & $\checkmark$ & $-$          & $C_{\max}$, TEC                                                 & MILP/MISOCP, approx. \\
\cite{ZHAO2025111754}              & Job shop            & $\checkmark$ & $-$          & TEC                                                             & Metaheuristic        \\
\cite{ho2022exact}                 & Flow shop           & $\checkmark$ & $-$          & TEC                                                             & MILP                 \\
\cite{shrouf2014optimizing}        & $1$ machine         & $\checkmark$ & $\checkmark$ & TEC                                                             & Genetic algorithm          \\
\cite{aghelinejad2018production}   & $1$ machine         & $\checkmark$ & $\checkmark$ & TEC                                                             &  ILP           \\
\cite{aghelinejad2019complexity}   & $1$ machine         & $\checkmark$ & $\checkmark$ & TEC                                                             & Complexity results           \\
\cite{benedikt2020power}           & $1$ machine         & $\checkmark$ & $\checkmark$ & TEC                                                 & ILP/CP + SPACES       \\
\cite{BENEDIKT2025}                & $1$ machine         & $\checkmark$ & $\checkmark$ & TEC                                                 & Branch and Bound      \\
\cite{hall2026optimal}             & $1$ machine (var.\ speeds) & $\checkmark$ & $\checkmark$ & TEC                                                      & Strongly poly.\ alg.  \\
\cite{zuccato2025energy}           & Flexible Job Shop           & $-$          & $\checkmark$ & $C_{\max}$, TEC, $\sum T_j$                                     & CP                    \\
\cite{naderi2022critical}          & Flexible Job Shop          & $-$          & $-$          & $C_{\max}$                                                      & LBBD                  \\
\midrule
\textbf{This work}                & \textbf{RCPSP} & $\checkmark$ & $\checkmark$ & $\boldsymbol{C_{\max}}$\textbf{, TEC} & \textbf{LBBD} \\
\bottomrule
\end{tabular}%
}
\caption{Comparison of the most closely related works. \emph{States} refers to
explicit power-state transition graphs on the energy-intense resource; \emph{TOU}
to time-of-use energy tariffs;
\emph{MISOCP} to mixed-integer second-order cone program.
}
\label{tab:related-work}
\end{small}
\end{table}



\subsection{Resource-Constrained Project Scheduling}

The Resource-Constrained Project Scheduling Problems (RCPSPs), introduced by~\cite{wiest1967heuristic}, are widely studied strongly \textsf{NP-Hard} problems~\citep{blazewicz1983scheduling}, in which solving even moderate instances to optimality remains a computational challenge.
The literature has explored the problem along two methodological axes. \emph{Exact methods} include branch-and-bound \citep{mohring2003solving}, constraint programming \citep{berthold2010constraint,laborie2018ibm}, and hybrid MIP/CP approaches \citep{chakrabortty2015resource}; constraint programming has proved especially competitive \citep{HEINZ2025} for makespan criterion. 
\emph{Heuristic methods} encompass priority-rule dispatching and metaheuristics~\citep{munlin2018solving,pellerin2020survey}. A comprehensive survey on this topic has recently been published by~\cite{artigues2025fifty}.

However, none of the RCPSP variants surveyed above model the internal power dynamics of the constrained resources. Machine dynamics are often disregarded, and energy consumption, when accounted for, is uniformly attributed to processing modes \citep{zheng2015reduction,maghsoudlou2021framework,pouramin2024multi}. This is precisely the modeling gap we aim to fill by coupling RCPSP with explicit machine-state transitions on the energy-intensive resource.

\subsection{Energy-Aware Scheduling Under TOU Tariffs and Machine States}

\paragraph{TOU scheduling without machine states}
A substantial body of work minimizes electricity costs under time-of-use (TOU) tariffs on classical machine environments: parallel machines~\citep{ding2015parallel,gaggero2023exact}, job shop~\citep{kurniawan2021distributed,ZHAO2025111754}, and flow shop and hybrid flow shop~\citep{fazli2018energy,ho2022exact}. Multi-objective extensions trade off energy cost against makespan~\citep{schulz2020multi,park2022energy,gaggero2023exact} or tardiness~\citep{rocholl2020bi,kurniawan2021distributed}. More recently, \citet{XIECHEN2026107567} extended this setting to unrelated parallel machines with uncertain processing times, proposing a chance-constrained bi-objective formulation that they convert into deterministic MILP and MISOCP models, together with a $(1+\epsilon)$-approximation scheme for the makespan–TEC Pareto front. A comprehensive review of this stream is provided by~\citet{gahm2016energy}. Although special TOU structures can yield polynomial algorithms~\citep{fang2016scheduling}, the general case remains strongly \textsf{NP-Hard}, and all these formulations share the assumption that, in any given time period, a machine is either actively processing or completely shut down.

\paragraph{Single-machine TOU scheduling with power-state transitions}
A tighter and more realistic model allows a machine to transition between a set of power states whose switch costs and durations interact non-trivially with the TOU tariff.
\cite{shrouf2014optimizing} introduced this setting for the single-machine problem $1,\text{TOU}|\text{states}|\text{TEC}$ and solved it with genetic algorithms.
\cite{aghelinejad2018production} proposed a compact ILP formulation, and~\cite{aghelinejad2019complexity} established that while fixing the job sequence yields a polynomial subproblem, the problem with free sequencing is strongly \textsf{NP-Hard}.
The decisive methodological advance came with \cite{benedikt2020power}, who observed that, given a fixed job sequence, the optimal state trajectory can be computed in polynomial time via a pre-processing step called SPACES which is based on shortest-path calculation in time-state graph. This reduces the ILP/CP model size dramatically. Building on this, \cite{BENEDIKT2025} designed a dedicated branching scheme with tailored heuristics and achieved speedups of two orders of magnitude over \cite{benedikt2020power}.
In a parallel development, \cite{hall2026optimal} extended the \cite{shrouf2014optimizing} setting with $q$ discrete processing speeds on top of \emph{sleep} and \emph{powered-down} idle states, and obtained the first \emph{strongly polynomial-time} exact algorithm for this class via a block decomposition combined with a shortest-path argument, outperforming CPLEX by several orders of magnitude. They further showed that two seemingly minor changes---restricting speed changes to task boundaries, or allowing power-down at any task completion---render the problem strongly \textsf{NP-complete}, locating the single-machine variant sharply at the tractability frontier.
These works establish a clear progression on the \emph{single-machine} variant. What they leave open is whether the structural insight behind SPACES---the decoupling of scheduling decisions from state-trajectory optimization---survives when the machine is one of \emph{several resources} competing for activities in an RCPSP, and how a decomposition should be organized to exploit it in that richer context.

\paragraph{RCPSP with energy and machine states}
To the best of our knowledge, no prior work combines RCPSP with power-state machine decisions on the energy-intensive resources. \cite{zuccato2025energy} recently incorporated machine modes, setup operations, and transport times into a flexible job-shop variant via constraint programming, but their model does not include explicit power-state transition graphs or TOU tariffs. The RCPSP papers that do address energy consumption \citep{zheng2015reduction, maghsoudlou2021framework,pouramin2024multi} model energy through multi-modal activity durations rather than machine-state dynamics, and therefore do not capture the idle-period optimization that is central to the present problem.

\subsection{Logic-Based Benders Decomposition for Scheduling}

Classical Benders decomposition~\citep{bnnobrs1962partitioning} reformulates a problem as a mixed-integer linear programming (MILP) master problem coupled with a continuous linear programming (LP) subproblem. The algorithm iteratively augments the master problem with \emph{Benders cuts} derived from each integer-feasible solution of the master, and terminates once an optimal solution is obtained. However, this approach is inherently restricted to cases in which the subproblem admits a linear formulation, thereby limiting its applicability to a relatively narrow class of combinatorial optimization problems. To address this limitation, \citet{hooker2003logic} introduced \emph{Logic-Based Benders Decomposition} (LBBD), which extends the framework to allow integer or more general combinatorial subproblems, with problem-specific inference procedures used to derive cutting planes from the subproblem solutions. The effectiveness of integrating MILP master problems with constraint programming (CP) subproblems was empirically demonstrated by \citet{jain2001algorithms} and \citet{hooker2007planning} in the context of scheduling, where substantial computational speedups were reported relative to purely MILP-based or CP-based methods.

More recently, LBBD has gained traction in energy- and sustainability-aware scheduling: \citet{lei2021decomposition} applies it to parallel machine scheduling of step-deteriorating jobs minimizing a combination of tardiness and extra energy consumption, \citet{hu2024logic} develops a bi-objective $\varepsilon$-LBBD for parallel machine selection with release dates and resource consumption, \citet{yang2024logic} tackles order acceptance and scheduling on heterogeneous factories under carbon caps through a Branch-and-Check LBBD scheme, and \citet{seman2024benders} adapts the framework to energy-aware task scheduling of nanosatellite constellations. These works confirm the suitability of LBBD whenever the energetic component naturally decouples into a global cost/capacity master problem and a per-resource sequencing subproblem.

\subsection{Positioning the Present Contribution}

The foregoing analysis delineates a specific research gap at the intersection of the three literature streams reviewed above: no existing study jointly addresses the RCPSP with TOU tariffs while simultaneously incorporating power-state machine decision modeling, and no available logic-based Benders decomposition (LBBD) framework proposes specialized cutting planes for this combined setting. Table~\ref{tab:related-work} articulates this gap explicitly by contrasting the core characteristics of the most relevant prior contributions.

The present paper extends previous work through three novel contributions, elaborated in the following sections. First, the new variant of the problem is a \emph{strict generalization} of the single-machine problems studied by~\cite{benedikt2020power} and~\cite{BENEDIKT2025} including the RCPSP components in the energy problem. Second, the proposed LBBD exploits the SPACES shortest-path structure~\citep{benedikt2020power} as a natural decomposition boundary; its novelty lies in the tailored Benders cuts derived from the state-transition subproblem. Third, we show that the same architecture can be applied to many scheduling problem coupling combinatorial assignment with TOU-aware machine-state decisions, something that has not been done yet in the literature. In particular, we expose two examples, in which the assignment part is an RCPSP with blocking time periods and a flexible job-shop scheduling, respectively. This approach is interesting in real applications as it permits smooth integrations of the energy component to scheduling problems for which efficient solving algorithms exist.

\section{Problem statement}
\label{section:problem_statement}

In this section, we describe the RCPSP with machine states and time-of-use (TOU) tariffs, along with its modeling by a compact monolithic ILP.
Consider a \emph{time horizon} $h \in \mathbb{N}$ divided into non-overlapping \emph{unitary intervals} $\mathcal{I} = \{1, \dots,h\}$. 
Each interval is associated with an \emph{energy cost} given by the vector $c = (c_1, c_2, \dots, c_h) \in \mathbb{R}^h$. 
An example of a real energy cost profile is displayed in Figure~\ref{fig:example-profile}.

 \begin{figure}[t]
    \centering
    \pgfplotstableread[col sep=comma]{images_sources/new_real_energy_costs_eur.csv}\datatable
    \begin{tikzpicture}
        \begin{axis}[
            width=0.75\textwidth,
            height=0.25\textwidth,
            ybar, 
            label style={font=\small},
            tick label style={font=\scriptsize},
            xtick={0,11,23,35,47},
            xticklabels={13th Fri 00:00, 13th Fri 12:00, 14th Sat 00:00, 14th Sat 12:00, 15th Sun 00:00},
            ylabel={cost [EUR/MWh]},
            xlabel={interval $\IdxInterval$ [-]},
            ymajorgrids,
            bar width=3pt,
            enlarge x limits={abs=0.75cm}, 
        ]
        \addplot table[x=idx, y=cost] {\datatable};
        \end{axis}
    \end{tikzpicture}
    \caption{Example of a 48-hour part of the energy profile with real electricity costs in the day-ahead market in the Czech Republic from June 2025 provided by Czech OTE.}
    \label{fig:example-profile}
\end{figure}

Let $\mathcal{R} = \{R_0, R_1, \dots, R_m\}$ be a set of \emph{resources}, with $m\in\mathbb{N}$. We note $\rho_k$ the \emph{resource capacity} of a resource $k \in \mathcal{R}$. Let $\mathcal{T} = \{j_1, j_2, \dots, j_n\}$ be a set of \emph{non-preemptive tasks} to be processed. A \emph{processing time} $p_j$ and a \emph{resource requirement} $r_k^j,\ k \in \mathcal{R}$ are associated to each task $j \in \mathcal{T}$.
We want to calculate the \emph{starting time} $S_j$ of the task $j \in \mathcal{T}$, from which we can induce its \emph{completion time} $C_j = S_j + p_j$.
For each resource $k \in \mathcal{R}$, the sum of the resource requirements for the tasks processed simultaneously must not exceed the capacity of that resource.
Moreover, all tasks must be processed before the end of the time horizon.

\begin{figure}[t]
    \centering
    \begin{tabular}{ccc}
    {
\newcommand{\ImNumCols}{3} 
\newcommand{\ImNumRows}{4} 
\newcommand{\ImGridSize}{0.7} 

\begin{tikzpicture}
	\foreach \x in {0,...,\ImNumCols}
	{
		\foreach \y in {0,...,\ImNumRows}
		{
			\pgfmathsetmacro{\crdX}{(\x-1) * \ImGridSize + \ImGridSize/2};
        		\pgfmathsetmacro{\crdY}{(\y-1) * \ImGridSize + \ImGridSize/2};        		
        		\coordinate (\x-\y) at (\crdX,\crdY);        		
		}
	}
	
	\pgfmathsetmacro{\drawRows}{\ImNumRows - 1};
	
	\draw (0,0) -- (0,\drawRows*\ImGridSize);
	\draw (0,\drawRows*\ImGridSize) -- (\ImNumCols*\ImGridSize,\drawRows*\ImGridSize);
	
	\node at (1-3) {4};
	\node at (1-2) {2};
	\node at (1-1) {5};
	
	\node at (2-3) {2};
	\node at (2-2) {2};
	\node at (2-1) {$\infty$};
	
	\node at (3-3) {1};
	\node at (3-2) {$\infty$};
	\node at (3-1) {0};
	
	\node[anchor=east] at ($(0-3) + (\ImGridSize/3,0)$) {\scriptsize $\StateProc$};
	\node[anchor=east] at ($(0-2) + (\ImGridSize/3,0)$) {\scriptsize $\StateIdle$};
	\node[anchor=east] at ($(0-1) + (\ImGridSize/3,0)$) {\scriptsize $\StateOff$};	
	
	\node at ($(1-4) + (0,-\ImGridSize/5)$) {\scriptsize $\strut \StateProc$};
	\node at ($(2-4) + (0,-\ImGridSize/5)$) {\scriptsize $\strut \StateIdle$};
	\node at ($(3-4) + (0,-\ImGridSize/5)$) {\scriptsize $\strut \StateOff$};
	
	\draw ($(1-3) + (-\ImGridSize/2,\ImGridSize/2)$) -- ($(1-3) + (-3*\ImGridSize/2,\ImGridSize)$);
	\node[anchor=north east] at ($(0-4) + (0,-\ImGridSize/6)$) {$\IdxState$};
	\node[anchor=south west] at ($(0-4) + (-\ImGridSize/5,-\ImGridSize/6)$) {$\IdxAnother{\IdxState}$};
	
	\node at (2-0) {$\strut \TransPower[\IdxState, \IdxAnother{\IdxState}]$};
\end{tikzpicture}
} &
    {
\newcommand{\ImNumCols}{3} 
\newcommand{\ImNumRows}{4} 
\newcommand{\ImGridSize}{0.7} 

\begin{tikzpicture}
	\pgfmathsetmacro{\prepareLayersX}{\ImNumCols+1};
	\pgfmathsetmacro{\prepareLayersY}{\ImNumRows+1};
	\foreach \x in {0,...,\prepareLayersX}
	{
		\foreach \y in {0,...,\prepareLayersY}
		{
			\pgfmathsetmacro{\crdX}{(\x-1) * \ImGridSize + \ImGridSize/2};
        		\pgfmathsetmacro{\crdY}{(\y-1) * \ImGridSize + \ImGridSize/2};        		
        		\coordinate (\x-\y) at (\crdX,\crdY);        		
		}
	}
	
	\pgfmathsetmacro{\crdX}{0*\ImGridSize};
	\pgfmathsetmacro{\drawRows}{\ImNumRows - 1};
	
	\draw (0,0) -- (0,\drawRows*\ImGridSize);
	\draw (0,\drawRows*\ImGridSize) -- (\ImNumCols*\ImGridSize,\drawRows*\ImGridSize);
	
	\node at (1-3) {1};
	\node at (1-2) {1};
	\node at (1-1) {2};
	
	\node at (2-3) {1};
	\node at (2-2) {1};
	\node at (2-1) {$\infty$};
	
	\node at (3-3) {1};
	\node at (3-2) {$\infty$};
	\node at (3-1) {1};
	
	\node[anchor=east] at ($(0-3) + (\ImGridSize/3,0)$) {\scriptsize $\StateProc$};
	\node[anchor=east] at ($(0-2) + (\ImGridSize/3,0)$) {\scriptsize $\StateIdle$};
	\node[anchor=east] at ($(0-1) + (\ImGridSize/3,0)$) {\scriptsize $\StateOff$};	
	
	\node at ($(1-4) + (0,-\ImGridSize/5)$) {\scriptsize $\strut \StateProc$};
	\node at ($(2-4) + (0,-\ImGridSize/5)$) {\scriptsize $\strut \StateIdle$};
	\node at ($(3-4) + (0,-\ImGridSize/5)$) {\scriptsize $\strut \StateOff$};
	
	\draw ($(1-3) + (-\ImGridSize/2,\ImGridSize/2)$) -- ($(1-3) + (-3*\ImGridSize/2,\ImGridSize)$);
	\node[anchor=north east] at ($(0-4) + (0,-\ImGridSize/6)$) {$\IdxState$};
	\node[anchor=south west] at ($(0-4) + (-\ImGridSize/5,-\ImGridSize/6)$) {$\IdxAnother{\IdxState}$};

	\node at (2-0) {$\strut \TransTime[\IdxState, \IdxAnother{\IdxState}]$};
\end{tikzpicture}
} &
    \begin{tikzpicture}[auto,-stealth]
\tikzset{
    state/.style={circle,draw=black,inner sep=1pt, minimum size=0.8cm},
    every edge/.append style={-stealth,thick}
}
    
\node[state] at (1,0.6) (nProc) {\scriptsize $\strut \StateProc$};
\node[state] at (2.5,-0.7) (nIdle) {\scriptsize $\strut \StateIdle$};
\node[state] at (-0.5,-0.7) (nOff)  {\scriptsize $\strut \StateOff$};

\path (nProc) edge[bend left=15] node[above right, pos=0.8] {\scriptsize 1/2} (nIdle);
\path (nIdle) edge[bend left=15] node[below left, pos=0.2] {\scriptsize 1/2} (nProc);

\path (nProc) edge[bend left=15] node[below right, pos=0.8] {\scriptsize 1/1} (nOff);
\path (nOff) edge[bend left=15] node[above left, pos=0.2] {\scriptsize 2/5} (nProc);

\path (nProc) edge[loop above,looseness=4, out=105, in=75] node[left,pos=0.1] {\scriptsize 1/4} (nProc);

\path (nOff) edge[loop below,looseness=4, in=255, out=285] node[right,pos=0.2] {\scriptsize 1/0} (nOff);

\path (nIdle) edge[loop below,looseness=4, in=255, out=285] node[right,pos=0.2] {\scriptsize 1/2} (nIdle);

\node at (0,-2) {};
\end{tikzpicture}  \\
    \end{tabular}
    \caption{Parameters of the transition power function \( \TransPower[\IdxState, \IdxAnother{\IdxState}] \) and transition time function \( \TransTime[\IdxState, \IdxAnother{\IdxState}] \), and the corresponding transition graph, where every edge from $\IdxState$ to $\IdxAnother{\IdxState}$ is labeled by $\TransTime[\IdxState, \IdxAnother{\IdxState}]$/$\TransPower[\IdxState, \IdxAnother{\IdxState}]$.}
    \label{fig:example-func-power-time}
\end{figure}

Additionally, let $G = (V, A)$ be the directed graph that represents the \emph{precedence constraints} between the tasks. Each node of $G$ is a task of $\mathcal{T}$, and each arc $(v_j, v_{j'}) \in A$ means that task $j$ precedes task $j'$, i.e.:
\begin{small}
\begin{align}
\label{eq:precedence_constraints}
    S_{j'} \geq C_j \iff S_{j'} - S_j \geq p_j, &\quad \forall(j,j') \in A
\end{align}
\end{small}

We assume that the resource $R_0$ is the only \emph{energy-intensive} resource, with a capacity equal to 1.
We note $\mathcal{J} \subset \mathcal{T}$, the subset of \emph{energy-intensive tasks}, i.e., such that $\forall j \in \mathcal{J}, r_0^j > 0$. We consider this subset to be non-empty. We note $\overline{\mathcal{J}}$ its \emph{complementary}, i.e., $\overline{\mathcal{J}} = \mathcal{T} \setminus \mathcal{J}$.

Until then, the formulation resembles a classical RCPSP problem. Now, let us introduce the variables and constraints that are specific to our problem.
We consider that the resource $R_0$ has a \emph{three-state set} $\Sigma = \{\StateProc, \StateIdle, \StateOff\}$. During each time interval, the machine operates in one of these states. When an energy-intensive task is processed, the machine must be in $\StateProc$ state.
We define a \emph{transition time} function $T: \Sigma \times \Sigma \xrightarrow{} \mathbb{Z}_{>0} \cup \{\infty\}$ and a \emph{transition energy} function  $P: \Sigma \times \Sigma \xrightarrow{} \mathbb{Z}_{\geq 0} \cup \{\infty\}$. We assume that the transition time between two states cannot be null. A transition time of value $\infty$ between two states denotes that this transition is infeasible. 
\cite{benedikt2020power} have demonstrated that it is possible to precompute the optimal transition with respect to the energy cost between any two $\StateProc$ intervals, as well as from the initial $\StateOff$ state to any $\StateProc$ state, and from any $\StateProc$ state to the final $\StateOff$ state; thanks to a graph-based method called \emph{SPACES}. It consists of calculating the optimal transition between two (state, interval) as the shortest path between these states in the transition graph.
An example of SPACES graph can be found in the supplementary materials.
Let us denote by $\Omega = (\Omega_1, \Omega_2, \dots, \Omega_{h})$ the vector that indicates for each time interval the current transition of the machine. For example, $\Omega_t = (\StateOff, \StateProc)$ means that the machine is in transition from the $\StateOff$ state at time $t$ to the $\StateProc$ state at time $t+1$, while $\Omega_{t'} = (\StateIdle, \StateIdle)$ signifies the machine stays in the $\StateIdle$ state between times $t'$ and $t'+1$. An optimal transition from one state in an interval $t$ to another state in $t'$ (assuming $t' > t$) in \emph{SPACES} is therefore a vector containing successive transitions with the lowest transition costs.
In addition, we assume that this machine starts and ends in the $\StateOff$ state. Hence, we can define $\text{StartTime} = \left\{ 1, \dots, T(\StateOff,\StateProc) + 1\right\}$, $\text{EndTime}_j = \left\{ h-T(\StateProc,\StateOff) -p_j + 1, \dots,h \right\}$ the sets during which no energy-intensive tasks can be processed. The sets and parameters are listed in Tables~\ref{tab:sets} and \ref{tab:parameters}, respectively.

A \emph{solution} of our problem is a pair $(\sigma, \Omega)$ where $\sigma = (\sigma_1, \sigma_2, \dots, \sigma_n) \in \mathbb{Z}_{\geq 0}^n$ is the vector denoting the start time of each task and $\Omega = \{\Omega_1,\dots,\Omega_h\} \in (\Sigma \times \Sigma)^h$ the vector of machine-state transitions.
In summary, a solution is \emph{feasible} if the following conditions are met:
\begin{itemize}[itemsep=0pt, parsep=0pt]
    \item the resource capacities are not exceeded;
    \item the precedence constraints of each task are satisfied;
    \item the energy-demanding $R_0$ tasks are processed only when the $R_0$ machine is in $\StateProc$ state;
    \item the machine is in $\StateOff$ state during the first and last intervals;
    \item the transition between states $s$ and $s'$ takes $T(s,s')$ time intervals.
\end{itemize}
The Total Energy Cost (TEC) of a solution $(\sigma, \Omega)$ can be computed as:
\begin{small}
\begin{align}
    TEC = \sum_{t \in \mathcal{I}} c_t \cdot P(\Omega_t)
    \label{eq:energy_obj}
\end{align}
\end{small}
where for all $(s, s') = \Omega_t \in \Sigma^2,\ P(\Omega_t) = P(s,s')$.

\noindent We can compute the makespan $C_{\max}$ of a solution as:
\begin{small}
\begin{align}
    C_{\max} = \max_{j \in \mathcal{T}}\ C_j
    \label{eq:makespan_obj}
\end{align}
\end{small}

\begin{table}[t]
    \centering
    \begin{small}
    \makebox[\textwidth][c]{
    \begin{tabular}{ll}
    \toprule
        {\textbf{Sets}} \\
    \midrule
        $\mathcal{I} = \{1, \dots, h\}$ & Set of unitary time intervals over the horizon $h$ \\
        $\mathcal{R} = \{R_0, R_1, \dots, R_m\}$ & Set of resources, where $R_0$ is the energy-intensive resource \\
        $\mathcal{T} = \{j_1, j_2, \dots, j_n\}$ & Set of non-preemptive tasks to be processed \\
        $\mathcal{J} \subset \mathcal{T}$ & Subset of energy-intensive tasks (i.e., $\forall j \in \mathcal{J}, r_0^j > 0$) \\
        $\overline{\mathcal{J}} = \mathcal{T} \setminus \mathcal{J}$ & Complementary subset of non-energy-intensive tasks \\
        $\Sigma = \{\StateProc, \StateIdle, \StateOff\}$ & Set of states of the energy-intensive machine $R_0$ \\
        $G = (V, A)$ & Directed precedence graph: nodes $V$ are tasks of $\mathcal{T}$, arcs $A$ encode precedences \\
        $\text{StartTime} =$ & Set of intervals during which no energy-intensive task \\
        \quad $\{1, \dots, 1 + T(\StateOff,\StateProc)\}$ & can start (machine warm-up from $\StateOff$ to $\StateProc$) \\
        $\text{EndTime}_j =$ & Set of intervals during which energy-intensive task $j \in \mathcal{J}$ \\
        \quad $\{h - T(\StateProc,\StateOff) - p_j + 1, \dots, h\}$ & cannot be processed (machine shutdown from $\StateProc$ to $\StateOff$) \\
    \bottomrule
    \end{tabular}}
    \end{small}
    \caption{Sets defined in the problem statement.}
    \label{tab:sets}
\end{table}
\begin{table}[t]
    \centering
    \begin{small}
    \begin{tabular}{ll}
    \toprule
        \textbf{Parameters} \\
    \midrule
        $\rho_k$       & Capacity of resource $k \in \mathcal{R}$ \\
        $r_{k}^j$      & Requirement on resource $k \in \mathcal{R}$ of task $j \in \mathcal{T}$ \\
        $p_j$          & Processing time of task $j \in \mathcal{T}$ \\
        $c_t$          & Energy cost at time $t \in \mathcal{I}$ \\
        $c^{job}_{jt}$ & Cost of starting the job $j \in \mathcal{J}$ at interval $t \in \mathcal{I}$\\
        \multirow[c]{2}{*}{$c^*_{lm}$}& Cost of the optimal transition from interval $l \in \mathcal{I}$ to interval $m \in \mathcal{I}$ (\StateProc\,to\,\StateProc), \\
                       & except for intervals $l = 0$ and $m = h$ for which \StateOff\, state is considered instead of \StateProc \\
    \bottomrule
    \end{tabular}
    \end{small}
    \caption{Parameters of the problem statement.}
    \label{tab:parameters}
\end{table}

\begin{table}
    \centering
    \begin{small}
    \begin{tabular}{ll}
    \toprule
        \multicolumn{2}{l}{\textbf{Decision Variables}} \\
    \midrule
        $x_{jt}= 1$  & if the task $j \in \mathcal{T}$ starts at interval $t \in \mathcal{I}$, $0$ otherwise \\
        $z_{lm} = 1$ & if the machine is performing an optimal SPACES transition between the intervals $l$ and $m$ (\StateProc\,to\,\StateProc), \\
                     & except for intervals $l = 0$ and $m = h$ for which \StateOff\, state is considered instead of \StateProc, $0$ otherwise \\
    \bottomrule
    \end{tabular}
    \end{small}
    \caption{Decision variables of the compact ILP.}
    \label{tab:decision_variables}
\end{table}

We aim to minimize the convex combination of both objectives. For that, we introduce a real parameter $\alpha \in [0,1]$ that represents the balance between the two. In order to have better control over the balance between the two objective values, we normalize the individual objectives by their optimal solutions they would have if treated independently, i.e., the optimal makespan ($lb_{RCPSP}$) of the RCPSP problem for \eqref{eq:makespan_obj} and the optimal solution value for TEC ($lb_{TEC}$) when removing the non-energy-intensive tasks from the problem. However, their minimal impact on energy-intensive tasks is preserved, that is, we update their precedence constraints and their computation times so that, if the removed tasks were introduced in an optimal solution, but with an infinite resource capacity on all resources but $R_0$, they would fit between the energy-intensive tasks according to the precedence constraints. In case $lb_{TEC}$ has a negative value (which in extreme cases may exist), we take the absolute value of the objective so that $lb_{TEC}$ is always a strictly positive value (if the optimal cost $TEC$ is zero, we fix $lb_{TEC}$ to an small $\varepsilon>0$).
Note that these can be computed in practice efficiently --- e.g., $lb_{TEC}$ via an ILP of \cite{benedikt2020power} or by a branch and bound~\citep{BENEDIKT2025}, whereas $lb_{RCPSP}$ resembles a classical RCPSP problem that can be solved efficiently, e.g., by constraint programming~\citep{HEINZ2025}.

The objective function can therefore be written as:
\begin{small}
\begin{align}
    \min\ \left\{ \alpha \cdot \frac{1}{lb_{TEC}} \cdot \sum_{t \in \mathcal{I}} c_t P(\Omega_t) + (1-\alpha) \cdot \frac{1}{lb_{RCPSP}} \cdot C_{\max} \right\}
    \label{eq:full_obj}
\end{align}
\end{small}

One can observe that when $\alpha = 0$, the objective function reduces to calculating the makespan, whereas when $\alpha = 1$, it involves finding the best TEC within the given time horizon. Thanks to the normalization, when $\alpha = 0.5$, both objectives weigh the same in the optimal solution.

The problem can be modeled by the following ILP. Since this ILP models the entirety of our problem, we refer to it as the \emph{monolithic} ILP.
Let $x_{jt}$ be a binary variable equal to 1 if the task $j \in \mathcal{T}$ starts at interval $t \in \mathcal{I}$, and $z_{lm}$ be a binary variable indicating that the machine is performing an optimal transition between two processing states from time interval $l$ to time interval $m$, except for intervals $l = 0$ and $m = h$ for which \StateOff state is considered instead of \StateProc.

\begin{small}
\begin{align}
    \min\ &\frac{\alpha}{lb_{TEC}} \cdot  \left( \sum_{j\in\mathcal{J}} \sum_{t = 1}^{h} c_{jt}^{job} x_{jt} + \sum_{l=1}^h\sum_{m=1}^{h+1}c^*_{lm}z_{lm}\right)+ \frac{(1-\alpha)}{lb_{RCPSP}} \cdot \left( \sum_{t=1}^h tx_{n+1,t} + 1\right) \label{eq:ilp_obj} \\
    s.t.\ &\sum_{t\in\mathcal{I}}x_{jt} = 1 &\quad \forall j \in \mathcal{T} \label{eq:ilp_all_jobs_done}\\
    &\sum_{j \in \mathcal{T}}\sum_{t' = \max\{1, t-p_j+1\}}^tr^j_k x_{jt'} \leq \rho_k &\quad \forall k \in \mathcal{R}, t \in \mathcal{I}\label{eq:ilp_resource_capacities}\\
    &\sum_{t \in \mathcal{I}}t x_{j't} - \sum_{t \in \mathcal{I}}t x_{jt} \geq p_j &\quad \forall(j,j') \in A \label{eq:ilp_precedences_1}\\
    &\sum_{t \in \mathcal{I}}t x_{(n+1)t} - \sum_{t \in \mathcal{I}}t x_{jt} \geq p_j - 1 &\quad \forall j \in \mathcal{T}\label{eq:ilp_precedences_2} \\
    &\sum_{j\in\mathcal{J}}\sum_{t \in \text{StartTime} \cup \text{EndTime}_j} x_{jt} = 0 & \label{eq:ilp_start_end_time} \\
    &\sum_{j \in \mathcal{J}} \sum_{t' = \max\{1, t - p_j + 1\}}^ix_{jt'} + \sum_{l=1}^t \sum_{m=i+1}^{h+1} z_{lm} = 1 &\quad \forall t \in \mathcal{I} \label{eq:ilp_machine_states}
\end{align}
\end{small}
where $c^{job}_{jt} = \sum_{k = t}^{t+p_j-1} c_k \cdot P(\StateProc,\StateProc)$ is the cost of starting the job $j \in \mathcal{J}$ at interval $t \in \mathcal{I}$, and $c^*_{lm}$ is the cost of the optimal transition from interval $l$ to interval $m$ (see \cite{benedikt2020power} for more details). 
As in the classical RCPSP, we introduce a dummy task $n+1$ as the last activity, representing the project's makespan.

Constraints~\eqref{eq:ilp_all_jobs_done} force all tasks to be performed once. Constraints~\eqref{eq:ilp_resource_capacities} ensure that no resource capacity is exceeded.
Constraints~\eqref{eq:ilp_precedences_1} enforce the precedence relations~\eqref{eq:precedence_constraints} between tasks, while constraints~\eqref{eq:ilp_precedences_2} calculate the makespan.
Constraints~\eqref{eq:ilp_start_end_time} prevent energy-intensive tasks from being processed at an infeasible moment regarding machine transition.
Finally, constraints~\eqref{eq:ilp_machine_states} signify that at each time interval, the machine is either processing a task (that is in \StateProc, indicating by the fact that the processing of the task started in the $\max\{1, i - p_j + 1\}$ previous intervals) or undergoing an optimal transition.

\paragraph{Example}

\begin{figure}[t]
\centering

\begin{subfigure}[b]{0.48\textwidth}
  \centering
  \begin{tikzpicture}
    \graph {
      "$j_1$" -> { 
        "$j_7$"[blue] -> "$j_4$",
        "$j_2$"[blue] -> {
          "$j_3$" -> "$j_4$",
          "$j_5$"
        },
      } -> "$j_6$"[blue] -> "$j_8$",
      "$j_7$"[blue] ->[bend left] "$j_6$";
    };
  \end{tikzpicture}
  \caption{Task precedences (energy-intensive tasks in blue).}
  \label{fig:example-requirements}
\end{subfigure}
\hfill
\begin{subfigure}[b]{0.48\textwidth}
  \centering
    \begin{tabular}{c|cccccccc}
      \toprule
      Params & $j_1$ & $j_2$ & $j_3$ & $j_4$ & $j_5$ & $j_6$ & $j_7$ & $j_8$ \\
      \midrule
      $p_j$ & 2 & 1 & 2 & 1 & 3 & 2 & 2 & 2 \\
      $r^j_0$ & 0 & 1 & 0 & 0 & 0 & 1 & 1 & 0 \\
      $r^j_1$ & 5 & 0 & 2 & 3 & 2 & 0 & 0 & 3 \\
      $r^j_2$ & 2 & 0 & 1 & 1 & 2 & 0 & 0 & 2 \\
      \bottomrule
    \end{tabular}%
  \caption{Task parameters.}
  \label{tab:example-table}
\end{subfigure}

\caption{Parameters for the example problem instance.}
\label{fig:example-data}
\end{figure}

\begin{figure}[ht]
    \centering
    \resizebox{0.85\textwidth}{!}{
    \begin{tikzpicture}
    \definecolor{mycolor}{HTML}{D4E7F6}

    \draw[-] (0,0) -- (16,0) node[right] {$\mathcal{I}$};
    \draw[->] (0,0) -- (0,2.75) node[above] {$R_1$};
    
    \draw[dashed] (-0.1,2.5) node[left] {$R_1$ capacity 5} -- (16,2.5) ;

    \draw[fill=mycolor] (0,0) rectangle (2,2.5) node[midway] {\Large $j_1$};
    \draw[fill=mycolor] (4,1) rectangle (6,2) node[midway] {\Large $j_3$};
    \draw[fill=mycolor] (6,1) rectangle (7,2.5) node[midway] {\Large $j_4$};
    \draw[fill=mycolor] (4,0) rectangle (7,1) node[midway] {\Large $j_5$};
    \draw[fill=mycolor] (10,0) rectangle (12,1.5) node[midway] {\Large $j_8$};

    \foreach \x in {1,2,...,16}
        \draw (\x,0.1) -- (\x,-0.1) node[below] at (\x-0.5, -0.2) {\x};
    
    \foreach \y/\ylab in {0.5/1,1/2,1.5/3,2/4}
        \draw (0.1,\y) -- (-0.1,\y) node[left] {\ylab};
    \draw (0.1,2.5) -- (-0.1,2.5);

    \draw[dashed] (12,0) -- (12,2.75) node[above] {$C_{\text{max}}$};

\end{tikzpicture}

}

\resizebox{0.85\textwidth}{!}{
    \begin{tikzpicture}
    \definecolor{mycolor}{HTML}{D4E7F6}

    \draw[-] (0,0) -- (16,0) node[right] {$\mathcal{I}$};
    \draw[->] (0,0) -- (0,1.75) node[above] {$R_2$};
    
    \draw[dashed] (-0.1,1.5) node[left] {$R_2$ capacity 3} -- (16,1.5) ;

    \draw[fill=mycolor] (0,0) rectangle (2,1) node[midway] {\Large $j_1$};
    \draw[fill=mycolor] (4,1) rectangle (6,1.5) node[midway] {\Large $j_3$};
    \draw[fill=mycolor] (6,1) rectangle (7,1.5) node[midway] {\Large $j_4$};
    \draw[fill=mycolor] (4,0) rectangle (7,1) node[midway] {\Large $j_5$};
    \draw[fill=mycolor] (10,0) rectangle (12,1) node[midway] {\Large $j_8$};

    \foreach \x in {1,2,...,16}
        \draw (\x,0.1) -- (\x,-0.1) node[below] at (\x-0.5, -0.2) {\x};
    
    \foreach \y/\ylab in {0.5/1,1/2}
        \draw (0.1,\y) -- (-0.1,\y) node[left] {\ylab};
    \draw (0.1,1.5) -- (-0.1,1.5);

    \draw[dashed] (12,0) -- (12,1.75) node[above] {$C_{\text{max}}$};

\end{tikzpicture}

}
\resizebox{0.85\textwidth}{!}{
    \begin{tikzpicture}
    \definecolor{mycolor}{HTML}{D4E7F6}

    \draw[-] (0,0) -- (16,0) node[right] {$\mathcal{I}$};
    \draw[->] (0,0) -- (0,6.5) node[above] {Energy cost};

    \draw[fill=mycolor] (0,3) rectangle (1,3 + 2/5) node[below] {};
    \draw[fill=mycolor] (1,3) rectangle (2,3 + 1/5) node[below] {};
    \draw[fill=mycolor] (2,3) rectangle (3,3 + 2/5) node[below] {};
    \draw[fill=mycolor] (3,3) rectangle (4,3 + 1/5) node[below] {};
    \draw[fill=mycolor] (4,3) rectangle (5,3 + 6/5) node[below] {};
    \draw[fill=mycolor] (5,3) rectangle (6,3 + 16/5) node[below] {};
    \draw[fill=mycolor] (6,3) rectangle (7,3 + 14/5) node[below] {};
    \draw[fill=mycolor] (7,3) rectangle (8,3 + 3/5) node[below] {};
    \draw[fill=mycolor] (8,3) rectangle (9,3 + 2/5) node[below] {};
    \draw[fill=mycolor] (9,3) rectangle (10,3 + 5/5) node[below] {};
    \draw[fill=mycolor] (10,3) rectangle (11,3 + 3/5) node[below] {};
    \draw[fill=mycolor] (11,3) rectangle (12,3 + 15/5) node[below] {};
    \draw[fill=mycolor] (12,3) rectangle (13,3 + 3/5) node[below] {};
    \draw[fill=mycolor] (13,3) rectangle (14,3 + 2/5) node[below] {};
    \draw[fill=mycolor] (14,3) rectangle (15,3 + 1/5) node[below] {};
    \draw[fill=mycolor] (15,3) rectangle (16,3 + 2/5) node[below] {};

    \draw[fill=black!30] (0,2) rectangle (1,3) node[midway] {2};
    \draw[fill=black!30] (1,2) rectangle (2,3) node[midway] {1};
    \draw[fill=black!30] (2,2) rectangle (3,3) node[midway] {2};
    \draw[fill=black!30] (3,2) rectangle (4,3) node[midway] {1};
    \draw[fill=black!30] (4,2) rectangle (5,3) node[midway] {6};
    \draw[fill=black!30] (5,2) rectangle (6,3) node[midway] {16};
    \draw[fill=black!30] (6,2) rectangle (7,3) node[midway] {14};
    \draw[fill=black!30] (7,2) rectangle (8,3) node[midway] {3};
    \draw[fill=black!30] (8,2) rectangle (9,3) node[midway] {2};
    \draw[fill=black!30] (9,2) rectangle (10,3) node[midway] {5};
    \draw[fill=black!30] (10,2) rectangle (11,3) node[midway] {3};
    \draw[fill=black!30] (11,2) rectangle (12,3) node[midway] {15};
    \draw[fill=black!30] (12,2) rectangle (13,3) node[midway] {3};
    \draw[fill=black!30] (13,2) rectangle (14,3) node[midway] {2};
    \draw[fill=black!30] (14,2) rectangle (15,3) node[midway] {1};
    \draw[fill=black!30] (15,2) rectangle (16,3) node[midway] {2};

    \draw[fill=black!30] (0,0) rectangle (1,1) node[midway] {0};
    \draw[fill=black!30] (1,0) rectangle (2,1) node[midway] {5};
    \draw[fill=black!30] (2,0) rectangle (3,1) node[midway] {10};
    \draw[fill=black!30] (3,0) rectangle (4,1) node[midway] {4};
    \draw[fill=black!30] (4,0) rectangle (5,1) node[midway] {24};
    \draw[fill=black!30] (5,0) rectangle (6,1) node[midway] {64};
    \draw[fill=black!30] (6,0) rectangle (7,1) node[midway] {28};
    \draw[fill=black!30] (7,0) rectangle (8,1) node[midway] {6};
    \draw[fill=black!30] (8,0) rectangle (9,1) node[midway] {8};
    \draw[fill=black!30] (9,0) rectangle (10,1) node[midway] {20};
    \draw[fill=black!30] (10,0) rectangle (11,1) node[midway] {3};
    \draw[fill=black!30] (11,0) rectangle (12,1) node[midway] {0};
    \draw[fill=black!30] (12,0) rectangle (13,1) node[midway] {0};
    \draw[fill=black!30] (13,0) rectangle (14,1) node[midway] {0};
    \draw[fill=black!30] (14,0) rectangle (15,1) node[midway] {0};
    \draw[fill=black!30] (15,0) rectangle (16,1) node[midway] {0};

    \draw[fill=blue!15] (0,1) rectangle (1,2) node[midway] { $\StateOff$};
    \draw[fill=blue!70] (1,1) rectangle (3,2) node[midway] { $\nearrow$};
    \draw[fill=mycolor] (3,1) rectangle (4,2) node[midway] {\Large $j_2$};
    \draw[fill=mycolor] (4,1) rectangle (6,2) node[midway] {\Large $j_7$};

    \draw[fill=blue!50] (6,1) rectangle (7,2) node[midway] { $\searrow$};
    \draw[fill=blue!70] (7,1) rectangle (8,2) node[midway] { $\nearrow$};

    \draw[fill=mycolor] (8,1) rectangle (10,2) node[midway] {\Large $j_6$};
    
    \draw[fill=blue!50] (10,1) rectangle (11,2) node[midway] { $\searrow$};
    \draw[fill=blue!15] (11,1) rectangle (16,2) node[midway] { $\StateOff$};

    \foreach \x in {1,2,...,16}
        \draw (\x,0.1) -- (\x,-0.1) node[below] at (\x-0.5, -0.2) {\x};

    \draw (0,0.0) -- (0,-0.0) node[left] at (-0.1, 1.5) {$R_0$ state};
    \draw (0,0.0) -- (0,-0.0) node[left, align=center] at (-0.1, 0.5) {TEC};
    \draw (0,0.0) -- (0,-0.0) node[left, align=center] at (-0.1, 2.5) {Energy cost $c_i$};
    
    \foreach \y in {0,2,..., 16}
        \draw (0.1, 3 + \y / 5) -- (-0.1, 3 + \y / 5) node[left] {\y};
    \draw[dashed] (12,0) -- (12,6.5) node[above] {$C_{\text{max}}$};

\end{tikzpicture}
}
    \caption{Example project schedule with the optimal transition for \(R_0\). Parameter \(\alpha=0.75\) (more towards TEC): $C_{\max}$ $=12$, TEC $=172$.}
    \label{fig:rcpsp-energy-sched1}
\end{figure}

Consider an instance with three resources $\{R_0, R_1, R_2\}$ and eight tasks $\mathcal{T} = \{j_1,\dots,j_8\}$ with a subset $\mathcal{J} = \{j_2, j_6, j_7\}$ of energy-intensive tasks. The precedence constraints are shown in Figure~\ref{fig:example-requirements}, the processing times and the resource capacities of each task are detailed in Figure~\ref{tab:example-table}.
We assume a transition diagram of the resource $R_0$ as the one shown in Figure~\ref{fig:example-func-power-time}.
The optimal solution for the instance considered for $\alpha = 0.75$ is depicted in Figure~\ref{fig:rcpsp-energy-sched1}.

\paragraph{Model motivation}
The motivation for the scheduling model proposed in this work stems from manufacturing processes with energy-intensive production steps.
For example, in glass production~\citep{modos2016robust}, the process is divided into two main phases: (i)~preprocesing, and (ii)~tempering.
In the preprocessing steps, the glass panels are cut and drilled, while in the tempering stage, the preprocessed glass panels are heated in a furnace up to 620 $^\circ$C and then cooled down \citep{modos2016robust}.
After that, other parts of the production process follow, such as polishing, quality control, and expedition.
The furnace is modeled as a unit-capacity resource.
An operation on this resource can represent a batch of multiple glass panels that are collected into a single operation via precedence constraints.
The products in the furnace undergo a common thermal cycle — no products can enter or leave until the current batch is completed.
The glass furnace is an expensive device; therefore, only a single one is available.
Also, we assume that the considered problem represents an operational-level scheduling with much information already available, i.e., TOU prices are fully known for the upcoming near-term horizon~\citep{bokde2021optimal}.

\section{Constraint programming model}
\label{sec:constraint_programming}

In this section, we propose a monolithic CP approach to solve the problem defined in the previous section. In particular, we will take advantage of the fact that in the CP, some variables may or may not be present in the final solution, depending on the value of other variables. These variables are called \textit{``optional''} variables. Moreover, we remind the reader that in CP, an interval variable may have either a \textit{fixed} length or be in a range of valid lengths. We represent the optimal transitions between the energy-intensive tasks with variables of a fixed length that will be allowed to move around the schedule as necessary. This trick allows us to reduce the number of variables needed to represent the optimal transitions, as we can precisely compute the maximum number of optimal transitions of each length that can be used in the schedule. In the remainder of the article, we will use the CP notations presented in \cite{HEINZ2022108586}.

\paragraph{Variables} We denote by $\xi_j$ the interval variable of a length $p_j$ associated with each task $j \in \mathcal{T}$. This variable corresponds to the processing of a task $j$, starting at time $\textsc{StartOf}(\xi_j)$ and ending at time $\textsc{EndOf}(\xi_j)$. By definition, each interval variable must be scheduled. We introduce optional interval variables $\zeta_{lk}$ to represent the optimal transition in the SPACES surrounding task processing on resource $R_0$. Let $l_{\max} = h - \sum_{j \in \mathcal{J}}p_j$ be the upper bound on the optimal transition length. The function $\mathrm{K}: l \mapsto  \left\lfloor \frac{h - \sum_{j \in \mathcal{J}}p_j}{l} \right\rfloor$ then gives the upper bound on the number of optimal transitions of length $l$ that can appear in the schedule. Therefore, the optional variables $\zeta_{lk}$ are created such that:
\begin{small}
\begin{align}
    \textsc{LengthOf}(\zeta_{lk}) = l &\quad \forall l \in \{1,\dots,l_{\max}\}, k\in\{1,\dots,\mathrm{K}(l)\}
\end{align}
\end{small}
{Figure~\ref{fig:spaces_variables} displays interval variables of the optimal solution of an instance comprising three energy-intensive tasks. Notice that the length of an absent interval always evaluates to 0.

\begin{figure}[t]
        \centering
   \begin{tikzpicture}
  \begin{groupplot}[
    group style = {group size=1 by 1, x descriptions at=edge bottom, vertical sep=0.1cm},
    xmin=-0.5, xmax=16.5,      
    ymin=0, ymax=1,
    xtick={0,...,16},          
    yticklabel=\empty,         
    tick label style={font=\small},
    xlabel={intervals $\mathcal{I}$ [-]},
    ylabel={$R_0$},
    ylabel style={rotate=-90}
  ]
    \nextgroupplot[
      height=0.15\textwidth,
      width=0.9\textwidth,
    ]
      \draw[thick,fill=blue!25]
        (axis cs:0,0.08) rectangle (axis cs:3,0.7);
      \draw[thick,fill=green!25]
        (axis cs:3,0.08) rectangle (axis cs:4,0.7);
      \draw[thick,fill=green!25]
        (axis cs:4,0.08) rectangle (axis cs:6,0.7);
      \draw[thick,fill=blue!25]
        (axis cs:6,0.08) rectangle (axis cs:8,0.7);
      \draw[thick,fill=green!25]
        (axis cs:8,0.08) rectangle (axis cs:10,0.7);
      \draw[thick,fill=blue!25]
        (axis cs:10,0.08) rectangle (axis cs:16,0.7);

      \node at (axis cs:1.5,0.39) {$\zeta_{3,1}$};
      \node at (axis cs:3.5,0.39) {$\xi_2$};
      \node at (axis cs:5,0.39)   {$\xi_7$};
      \node at (axis cs:7,0.39)   {$\zeta_{2,1}$};
      \node at (axis cs:9,0.39)   {$\xi_6$};
      \node at (axis cs:13,0.39)  {$\zeta_{6,1}$};
  \end{groupplot}
\end{tikzpicture}
\caption{Example of interval variables present on resource $R_0$ in a valid schedule.}
\label{fig:spaces_variables}
\end{figure}

\paragraph{Objective} We use a vector $c^{job}_j$ to express the TEC for the processing of task $j$ based on the interval $i$ in which it starts. Additionally, since the variables representing the optimal switching are not fixed in their position, they are not directly associated with the total energy cost of the switching. Hence, we define the vector $c^{space}_l$, the i\textsuperscript{th} element of the vector corresponds to the total energy cost of the optimal length switching $l$, if started in the interval $i$. The values of $c^{space}_l$ are given as follows:
\begin{small}
\begin{align*}
    c^{space}_{l}[i] = c^*(i,i+l) &\quad \forall i \in \{1,\dots,h-l+1\}
\end{align*}
\end{small}

The model minimizes the following objective:
\begin{small}
\begin{equation}
\begin{gathered}
    \frac{\alpha}{lb_{TEC}} \cdot \left( \sum_{l=1}^{l_{\max}}\sum_{k=1}^{\mathrm{K}(l)}\textsc{Element}\left(c^{space}_l, \textsc{StartOf}(\zeta_{lk})\right) + \sum_{j\in\mathcal{J}}\textsc{Element}\left(c^{job}_j,\textsc{StartOf}(\xi_j)\right)\right) \\
    + \frac{1-\alpha}{lb_{RCPSP}} \cdot \max_{i \in \mathcal{T}}\left\{\textsc{EndOf}(\xi_i)\right\}
\end{gathered}
\end{equation}
\end{small}
with the first part representing the TEC part of the objective and the second part representing the $C_{\max}$ part of the objective. Note that when computing the TEC, we consider only tasks requiring energy consumption in the set $\mathcal{J}$, while for the objective $C_{\max}$, we consider all the tasks.

\paragraph{Constraints} A solution is feasible if it respects the following constraints:
\begin{small}
\begin{align}
\textsc{StartOf}(\zeta_{lk}) \geq 0 \quad \forall l \in \{1, \ldots, l_{\max}\} &\quad \forall k \in \{1, \ldots, \mathrm{K}(l)\} \label{eq:freecp-startof} \\
\textsc{EndOf}(\zeta_{lk}) \leq h \quad \forall l \in \{1, \ldots, l_{\max}\} &\quad \forall k \in \{1, \ldots, \mathrm{K}(l)\} \label{eq:freecp-endof} \\
\textsc{PresenceOf}(\zeta_{lk}) \geq \textsc{PresenceOf}(\zeta_{lk+1}) &\quad \forall l \in \{1, \ldots, l_{\max}\}, k \in \{1, \ldots, \mathrm{K}(l) - 1\} \label{eq:freecp-precence-ordering-1} \\
\textsc{EndBeforeStart}(\zeta_{lk}, \zeta_{lk+1}) &\quad \forall l \in \{1, \ldots, l_{\max}\}, k \in \{1, \ldots, \mathrm{K}(l) - 1\} \label{eq:freecp-precence-ordering-2}
\end{align}
\end{small}

Constraints~\eqref{eq:freecp-startof} and \eqref{eq:freecp-endof} ensure that all optimal transition variables, if present, are scheduled within the scheduling horizon. Constraints~\eqref{eq:freecp-precence-ordering-1} and \eqref{eq:freecp-precence-ordering-2} give an order to the variables that represent the optimal transitions and ensure that the variables are used in an order given by their indexes for transitions of the same length. This removes redundant solutions that differ only in the order of variable use.

\begin{small}
\begin{align}
\textsc{StartOf}(\xi_j) \geq 1 + T(\StateOff, \StateProc) &\quad \forall j \in \mathcal{J} \label{eq:freecp-minimum-start} \\
\textsc{EndOf}(\xi_j) \leq h - T(\StateProc, \StateOff) - 1 &\quad \forall j \in \mathcal{J} \label{eq:freecp-maximum-end} \\
\textsc{EndOf}(\xi_j) \leq h &\quad \forall j \in \mathcal{T} \label{eq:freecp-horizon} \\
\textsc{NoOverlap}\left(\{\xi_j: \forall j \in \mathcal{J}\}\cup\{\zeta_{lk}: \forall l \in \{1,\dots,l_{\max}\}, k\in \{1,\dots,\mathrm{K}(l)\}\}\right) &\label{eq:freecp-nooverlap}
\end{align}
\end{small}

Constraints~\eqref{eq:freecp-minimum-start} and \eqref{eq:freecp-maximum-end} ensure that all energy-intensive tasks are scheduled so that energy-intensive machines have time to warm up and shut down, respectively.
Constraints~\eqref{eq:freecp-horizon} ensure that all tasks are scheduled within the scheduling horizon. Compared to the constraints~\eqref{eq:freecp-minimum-start} and \eqref{eq:freecp-maximum-end}, we consider all of the tasks in this case. \textsc{NoOverlap} constraints~\eqref{eq:freecp-nooverlap} guarantee that the resource $R_0$ never processes a task and undergoes an optimal transition at the same time.
\begin{small}
\begin{align}
\sum_{l=1}^{l_{\max}}\sum_{k=1}^{\mathrm{K}(l)}\textsc{LengthOf}(\zeta_{lk}) = l_{\max} & \label{eq:freecp-length-lmax}\\
\sum_{j \in \mathcal{T}} \textsc{Pulse}(\xi_j,r^j_\kappa)\leq \rho_\kappa &\quad \forall \kappa \in \mathcal{R} \label{eq:freecp-resource-capacity}\\
\textsc{EndBeforeStart}(\xi_i,\xi_j) &\quad \forall (i,j) \in A \label{eq:freecp-precedence-relations}
\end{align}
\end{small}

Constraint~\eqref{eq:freecp-length-lmax} ensures that the length of all optimal transitions used in the solution is exactly equal to the value of $l_{\max}$; optional variables that are not present in the solution do not count towards this sum. By the definition of $l_{\max}$, these constraints, together with constraints~\eqref{eq:freecp-nooverlap}, also ensure that for the resource $R_0$, the whole scheduling horizon is filled by the processing of tasks or by optimal transitions in the SPACES around. Constraints~\eqref{eq:freecp-resource-capacity} guarantee that resource capacities are not exceeded in any time interval.
Finally, with $A$ being the set of arcs of the precedence graph, Constraints~\eqref{eq:freecp-precedence-relations} rule out any solutions in which the precedence relations are violated.

\section{Logic-Based Benders Decomposition}
\label{section:LBBD}

This section introduces an LBBD approach for the problem considered, starting from its simplified version when $\alpha = 1$, to a general version for any $\alpha < 1$.
The idea behind this approach is to separate the variables and constraints associated with the RCPSP from those related to the machine states problem. For that, first, we will not consider the makespan in the objective. Therefore, the RCPSP part can be seen as a feasibility problem inside the TEC with TOU tariffs and machine states on a single machine. We then divide the problem into two dependent problems. The master one will be the machine state problem, solved by the ILP-SPACES model, while the subproblem will be an RCPSP feasibility problem modeled by a CP formulation. Of course, this would only work if $\alpha = 1$. Therefore, we first present the special case $\alpha = 1$ in Section~\ref{subsection:alphaeq1}. Then, in Section~\ref{subsection:alphaleq1}, we will investigate how to generalize this approach to the case $\alpha < 1$. Note that the case $\alpha = 0$ is not interesting and therefore can be ignored. Indeed, as long as the constraints~\eqref{eq:ilp_start_end_time} are satisfied, the machine state problem will not generate further infeasibility, since the machine can remain in \StateProc~state for the duration of the schedule. That means that such a problem can be reduced to a classical RCPSP with a single time window for energy-intensive tasks whose size should be nearly equal to the time horizon.
The layout of our LBBD approach is detailed in Algorithm~\ref{algo:LBBD}.

\begin{algorithm}[t]
\caption{Logic-Based Benders Decomposition (LBBD)}
\begin{algorithmic}[1]
\While{$Master$ is not optimal}
    \State $\m{\hat{x}} \gets Master()$ \Comment{Get integer solution.}
    \State $\m{\hat{\zeta}} \gets Subproblem(\m{\hat{x}})$    \Comment{Solve the subproblem.}
    \If{$\m{\hat{\zeta}}$ optimal or feasible}
        \State Add optimality cut \eqref{eq:optimality_cuts} to $Master$ \Comment{Give feedback to the $Master$}
    \Else
        \State Compute $\mathcal{I}nf$ \Comment{Retrieve tasks of $\mathcal{J}$ that lead to infeasibility.}
        \State Add feasibility cut \eqref{eq:lazy-constraint} to $Master$ \Comment{Cut all solutions sharing same $\mathcal{I}nf$ at once.}
    \EndIf
\EndWhile
\end{algorithmic}
\label{algo:LBBD}
\end{algorithm}

\subsection{RCPSP seen as a feasibility problem: Case $\alpha = 1$}
\label{subsection:alphaeq1}

In this section, we assume that $\alpha = 1$, i.e., that only the total energy part of the objective function~\eqref{eq:full_obj} matters. Therefore, the RCPSP can be seen as a feasibility problem in which one must verify that all of the non-energy-intensive tasks can be scheduled in parallel to the energy-intensive ones, while not exceeding the time horizon.

\subsubsection{Master problem}

The master problem consists of optimizing the machine state problem alone, i.e., finding a scheduling of the energy-intensive tasks that minimizes the energy costs. Therefore, we keep in mind all the constraints of the ILP related to energy-intensive tasks. The decision variables presented in Table~\ref{tab:decision_variables}, are similar to those of the ILP, except that only the energy-intensive tasks are represented by the variables $x_{jt}$, i.e., $\mathcal{X} = \{x_{jt}\ |\ j\in\mathcal{J}, t\in \mathcal{I}\}$. All of the decision variables are binary. Hence, constraints~\eqref{eq:ilp_all_jobs_done}, \eqref{eq:ilp_precedences_1}, \eqref{eq:ilp_start_end_time}, \eqref{eq:ilp_machine_states} and \eqref{eq:ilp_resource_capacities} (the latter limited to resource $R_0$) remain in the master. Constraints~\eqref{eq:ilp_precedences_2} are discarded as the makespan is no longer considered.

We denote the relation `$\prec$' for any tuple $u,v \in \mathcal{T}$ such that $u \prec v$ implies that $u$ is a direct predecessor of $v$.
We denote $Pred(j) \subset \mathcal{T}$ and $Succ(j) \subset \mathcal{T}$, the sets of predecessors and successors of a task $j \in \mathcal{T}$, respectively.
The precedence constraints~\eqref{eq:ilp_precedences_1} cannot be used directly; indeed, since the binary variables $x_{ji}$ represent only energy-intensive tasks, if we consider $u, v \in \mathcal{J}$ and $j \in \overline{\mathcal{J}}$, then precedence constraints such as $u \prec j \prec v$ would not be taken into account.
Furthermore, we know that in this case $S_v \geq C_u + p_j$.
Thus, for all $(u, v) \in \mathcal{J}^2$, if there exists a path $L_{uv} = \{ j_1, \dots, j_l \}$ such that $u \prec j_1 \prec \dots \prec j_l \prec v$, then:
\begin{small}
\begin{align}
\label{eq:minimal_starting_date}
S_v \geq C_u + \sum_{j \in L_{uv}} p_j
\end{align}
\end{small}

Let $\mathcal{L}_{uv}$ denote the set of all possible precedence paths from task $u$ to task $v$, i.e., if there exist $\mathrm{L}$ possible precedence paths between $u$ and $v$, then $\mathcal{L}_{uv} = \{ L^l_{uv} \}_{l \in \{1, \dots, \lambda\}}$. This set is finite because there is a finite number of tasks. We say that if $\mathcal{L}_{uv} \not = \emptyset$, then $u$ is an ancestor of $v$, and $v$ is a descendant of $u$.
We name $L^*_{uv}$ the longest path in $\mathcal{L}_{uv}$.
Then, we can state that:
\begin{proposition}
\label{prop:minimal_starting_time}
    For any $(u, v) \in \mathcal{J}^2$, if $\mathcal{L}_{uv} \neq \emptyset$, then:
    \begin{small}
    \begin{equation}
        S_v \geq C_u + \sum_{j \in L^*_{uv}} p_j
    \end{equation}
    \end{small}
\end{proposition}
\begin{proof}
    Consider a pair $(u,v) \in \mathcal{J}^2$ such that $\mathcal{L}_{uv} \neq \emptyset$. Then, there exist paths $L_{uv}^1, \dots, L^\lambda_{uv}$ in $\mathcal{L}_{uv}$. Without loss of generality, we can reorder these paths so that:
    \begin{small}
    $$\sum_{j \in L^1_{uv}} p_j \leq \sum_{j \in L^2_{uv}} p_j \leq \dots \leq \sum_{j \in L^\lambda_{uv}} p_j$$
    \end{small}
    Then, according to \eqref{eq:minimal_starting_date}, we have:
    \begin{small}
    $$ C_u + \sum_{j \in L^1_{uv}} p_j \leq \dots \leq C_u + \sum_{j \in L^\lambda_{uv}} p_j \leq S_v $$
    \end{small}
    As $L^*_{uv} = L^\lambda_{uv}$ holds by construction, the statement follows.
\end{proof}

We introduce the function $\MD: u \times v \in \mathcal{J}^2 \xrightarrow{} \mathbb{Z}^+ \cup \{\infty\}$ such that $\MD(u, v)$ is the minimal processing time between tasks $u$ and $v$ if $v$ is a descendant of $u$, otherwise $\infty$, i.e.,
\begin{equation}
\label{eq:md_definition}
\forall u,v \in \mathcal{J},\ \MD(u,v) = \left\{ \begin{array}{ll} p_u + \sum_{j \in L^*_{uv}} p_j & \text{if}\ \mathcal{L}_{uv} \neq \emptyset \\ \infty & \text{otherwise} \end{array} \right.
\end{equation}

The domain of $\MD$ can be extended to non-energy-intensive tasks $\overline{\mathcal{J}}$ , which is useful for handling cases where all predecessors — or all successors — of an energy-intensive task $j$ are themselves non-energy-intensive.
Computing this function is equivalent to calculating the longest paths between all nodes of the precedence graph associated with the RCPSP. Since there is no cyclic relation (i.e., $\nexists j \in \mathcal{T},\ j \prec \dots \prec j$) in the RCPSP, the precedence graph is a directed acyclic graph. Hence, the longest paths can be found by computing the shortest paths in the precedence graph in which each arc leaving a node $j$ has a weight equal to $-p_j$. Furthermore, since the precedence graph is usually sparse, we use Johnson's all-pairs shortest path algorithm \citep{johnson1977efficient}, which has a complexity of $O(|V||A| \log |V|)$ in a sparse graph $G(V,A)$. We say that a task $j$ is the \emph{predecessor} (resp. \emph{successor}) of a task $i$ if $\MD(i,j) < \infty$ (resp. $\MD(j,i) < \infty$). In light of the above, the master ILP can be written as:
\begin{small}
\begin{align}
\small
    (Master):\quad\min\ \sum_{j\in\mathcal{J}} \sum_{t = 1}^{h} c_{jt}^{job} x_{jt} + \sum_{l=1}^h\sum_{m=1}^{h+1}c^*_{lm}z_{lm}
\end{align}
\emph{subject to}
\begin{align}
    \sum_{t\in\mathcal{I}}x_{jt} = 1 &\quad \forall j \in \mathcal{J} \label{eq:benders_all_jobs_done} \\
    \sum_{t \in \mathcal{I}}t x_{vt} - \sum_{t \in \mathcal{I}}t x_{ut} \geq \MD(u,v) &\quad \forall(u,v) \in \mathcal{J}^2, \MD(u,v) < \infty \label{eq:benders_precedences} \\
    \sum_{j \in \mathcal{J}}\sum_{t' = \max\{1, t-p_j+1\}}^t x_{jt'} \leq 1 &\quad \forall t \in \mathcal{I} \label{eq:benders_capacities} \\
    \sum_{j \in \mathcal{J}} \sum_{t \in \text{StartTime} \cup \text{EndTime}_j} x_{jt} = 0 & \label{eq:benders_startAndEndLimitations} \\
    \sum_{j' \in Pred(j)\cup\overline{\mathcal{J}}} \sum_{t = 0}^{\MD(j',j)} x_{jt} + \sum_{j' \in Succ(j)\cup\overline{\mathcal{J}}} \sum_{t = h - \MD(j,j') + 1}^h x_{jt} = 0 &\quad \forall j \in \mathcal{J} \label{eq:benders_startAndEndLimitations2}\\
    \sum_{j \in \mathcal{J}} \sum_{t' = \max\{1, t - p_j + 1\}}^tx_{jt'} + \sum_{l=1}^t \sum_{m=t+1}^{h+1} z_{lm} = 1 &\quad \forall t \in \mathcal{I}
\end{align}
\end{small}

Constraints~\eqref{eq:benders_precedences} are a modified version of constraints~\eqref{eq:ilp_precedences_1} guarantee that each couple $(u,v) \in \mathcal{J}^2$ satisfies Proposition~\ref{prop:minimal_starting_time}. Constraints~\eqref{eq:benders_startAndEndLimitations2} address the scenarios where an energy-intensive task is preceded (resp. succeeded) by a sequence of non-energy-intensive tasks only. In this case, these constraints compel the energy-intensive task to allocate sufficient time before (resp. after) to allow this sequence to be scheduled. These two sets of constraints prune solutions that would have an associated subproblem that is trivially infeasible, without having the burden of verifying them.
The remaining constraints are similar to the ILP constraints (although restricted to the set $\mathcal{J}$ instead of $\mathcal{T}$).

It is worth noting that an integer solution found by the master problem may not be a valid solution to the RCPSP, as the resource constraints of non-energy-intensive tasks may not necessarily be enforced.
Hence, each time a feasible solution of the master problem is found, we must verify the feasibility for the RCPSP by means of an oracle that consists of solving the subproblem; if the oracle finds the master solution infeasible, new constraints are injected into the master to cut off this solution.

\subsubsection{Subproblem}

The subproblem consists of finding a feasible solution to the RCPSP problem associated with the solution provided by the master problem. Since $\mathcal{X}$ are fixed integers in the subproblem, they impose the starting time of the energy-intensive tasks in the RCPSP.
They shall be considered as parameters in the subproblem, denoted as $\hat{\mathcal{X}} = \{\hat{x}_{jt}\}_{j \in \mathcal{J}, t \in \mathcal{I}}$.
The subproblem also needs to schedule all tasks before the time horizon.
Given the findings of \cite{HEINZ2025}, which show that the constraint programming paradigm is highly effective in solving RCPSPs under a constraint time horizon, we adopt this methodology to address the subproblem. The proposed CP model is the following:

\begin{small}
\begin{align}
    (Subproblem)\quad C_{\max}^{sub}:\ & \min \max_{j \in \mathcal{T}} \textsc{EndOf}(\xi_j) \label{eq:cp-obj}\\
    &\sum_{j \in \mathcal{J}} \textsc{Pulse}(\xi_j,r_k^j) \leq \rho_k &\quad \forall k\in\mathcal{R} \label{eq:cp-capacity}\\
    &\textsc{EndBeforeStart}(\xi_i, \xi_j) &\quad \forall (i, j) \in A \label{eq:cp-precedences}\\
    &\textsc{EndOf}(\xi_j) \leq h &\quad \forall j \in \overline{\mathcal{J}} \label{eq:cp-horizon}\\
    &\textsc{StartOf}(\xi_j) = t &\quad \forall j \in \mathcal{J}, t \in \mathcal{I}, \hat{x}_{jt} = 1 \label{eq:cp-fixed-ee}\\
    &\textsc{Interval}\ \xi_j,\, \textsc{LengthOf}(\xi_j) = p_j &\quad \forall j \in \mathcal{T} 
\end{align}
\end{small}

As in the monolithic CP described in Section~\ref{sec:constraint_programming}, constraints~\eqref{eq:cp-capacity} ensure that the capacities of each resource are not exceeded. Constraints~\eqref{eq:cp-precedences} model the precedence constraints. Constraints~\eqref{eq:cp-horizon} prevent tasks from being scheduled after the end of the schedule horizon, while constraints~\eqref{eq:cp-fixed-ee} impose that energy-intensive tasks start at the same time as in the master solution. The Objective~\eqref{eq:cp-obj} minimizes the makespan.
In the particular case $\alpha = 1$, this objective function~\eqref{eq:cp-obj} does not need to be added to the CP, since our subproblem is therefore considered a feasibility problem. However, at the end of the search, after an optimal solution has been found, it can be relevant to recompute its associated subproblem in view of the objective function, such that we could retrieve the optimal makespan of the master optimal solution (this is only for a practical perspective; it does not change anything on the optimal cost of energy).

Let us denote $\mathcal{X}^\dagger = \{x_{jt}\ |\ x_{jt} \in \mathcal{X}, \hat{x}_{jt} = 1\}$.
When the subproblems cannot find a valid solution for a given energy-intensive task schedule given by the master, we compute the minimal infeasibility set $\mathcal{I}nf \subseteq \mathcal{X}^\dagger$ over the constraints~\eqref{eq:cp-fixed-ee}. This set contains the variables appearing in the minimal conflict that causes the infeasibility. Since the conflict is minimal, the removal of any one of these constraints will remove that particular cause of infeasibility. Note that there may be other conflicts in the solution.
This set can be automatically computed by the CP solver \citep{laborie2018ibm} if it proves the infeasibility of the subproblem. If the solver cannot compute this subset in a reasonable amount of time, then, since we know the solution is infeasible
, we set $\mathcal{I}nf = \mathcal{X}^\dagger$.
Once this set is determined, the following cuts are added to the master.

\paragraph{Feasibility cuts}

Given a minimal infeasibility set $\mathcal{I}nf$, we add in \emph{Master} the following feasibility cut:
\begin{small}
\begin{align}
\label{eq:lazy-constraint}
    \sum_{\substack{x_{jt} \in \mathcal{I}nf}} x_{jt} \leq |\mathcal{I}nf| - 1
\end{align}
\end{small}

This cut enforces that all, but one, of the energy-intensive tasks involved in the infeasibility set can keep their current starting time in a feasible solution. This rules out the current infeasible solution while also pruning all solutions that share the same infeasible subset. If $\mathcal{I}nf = \mathcal{X}^\dagger$, only the solution defined by $\mathcal{X}^\dagger$ is cut off.

\subsubsection{Warmstarts}
\label{subsection:warmstart}

An increase in instance size may undermine the efficacy of ILP approaches, thereby hindering their ability to find feasible solutions. In contrast, CP approaches are highly efficient in identifying feasible solutions quickly. However, these approaches often struggle to narrow the optimality gap due to their limited ability to generate strong lower bounds. Therefore, we propose two ways of generating a warmstart solution, both based on CP. The first uses the classical CP model \citep{artigues2013resource} with additional constraints~\eqref{eq:freecp-minimum-start} and \eqref{eq:freecp-endof}. However, we do not provide any objective function to the CP solver; the objective is to find a valid solution. We call \emph{fsws} this warmstart approach.
The second solves the problem through the CP model delineated in Section~\ref{sec:constraint_programming}, within a limited time frame of approximately one minute. We name \emph{cpws} this warmstart approach. Subsequently, the optimal solution identified through these processes is utilized as a starting point for our ILP approaches, encompassing both ILP and LBBD.

\subsection{RCPSP weighting on the master objective: Case $\alpha < 1$} 
\label{subsection:alphaleq1}

Now, let us focus on the case $\alpha < 1$. In this case, the $C_{\max}$ objective function of the RCPSP also contributes to the objective function. Therefore, we need to introduce a way to represent the makespan within the master problem. The issue is that, since most of the tasks of the RCPSP are not present in the master, we cannot induce the makespan directly from it.

Therefore, we introduce a new decision variable $q \in \mathbb{Z}_{\geq0}$ to represent the value of makespan. We rewrite the objective function as:
\begin{small}
\begin{align}
    \min\ \frac{\alpha}{lb_{TEC}} \cdot \left( \sum_{j\in\mathcal{J}} \sum_{t = 1}^{h} c_{jt}^{job} x_{jt} + \sum_{l=1}^h\sum_{m=1}^{h+1}c^*_{lm}z_{lm} \right) + \frac{1-\alpha}{lb_{RCPSP}} \cdot q
\end{align}
\end{small}

As described above, we require new constraints to ensure that the variable $q$ accurately models the makespan, in the case where the subproblem is feasible. Indeed, if the subproblem is infeasible, feasibility cuts~\eqref{eq:lazy-constraint} are added. Since the makespan is only known after computing the subproblem, we need Benders optimality cuts to inject the value of the optimal subproblem makespan $C^{sub}_{\max}$ into the master.

\paragraph{Optimality cut}
\citet{artigues2003insertion} showed that a solution of an RCPSP can be modeled as a resource-flow graph. The critical path of this graph contains all the tasks that contribute directly to the makespan, i.e., any change in the starting time of one of these tasks would modify the makespan.
The general idea of the proposed optimality cut is that, for a given solution of a problem, only a subset of the energy-intensive tasks are on the resource-flow critical path, so that only the tasks in this subset have to be in the optimality cut. In addition, by exploiting the precedence structure along this resource-flow critical path, we can estimate a lower bound on the makespan deviation.
Let $\Pi^\vartheta = \{j_1, j_2, \dots, j_{|\Pi^\vartheta|}\}$ be the critical energy-intensive tasks of an optimal subproblem solution $\vartheta$, indexed in increasing order of start time, and let $\hat{t}_{j_k}^\vartheta$ denote the start time of $j_k$ in $\vartheta$. For each $j_k$ we define:
\begin{small}
\begin{align}
\varphi(j_k,\hat{t}_{j_k}^\vartheta) \;=\; \max \Bigl\{ \hat{t}_{j_l}^\vartheta + p_{j_l} \;\Big|\; l < k,\ \hat{t}_{j_l}^\vartheta + p_{j_l} \leq \hat{t}_{j_k}^\vartheta \Bigr\},
\end{align}
\end{small}
with the convention $\varphi(j_1,\hat{t}_{j_1}^\vartheta) = 0$. The quantity $\varphi(j_k,\hat{t}_{j_k}^\vartheta)$ is the longest completion time, no later than $\hat{t}_{j_k}^\vartheta$, of a critical energy-intensive task of $\vartheta$ scheduled before $j_k$; it acts as the closest energy-intensive ``anchor'' on the chain leading to $j_k$, even when the immediate predecessor on the critical resource-flow path is itself non energy-intensive. Note that, by construction, $\varphi(j_k,\hat{t}_{j_k}^\vartheta) \leq \hat{t}_{j_k}^\vartheta \leq C_{\max}^{sub}$. The intuition for the cut is that, if the master deviates from the current critical assignment at $(j_k,\hat{t}_{j_k}^\vartheta)$, the energy-intensive task of realizing $\varphi(j_k,\hat{t}_{j_k}^\vartheta)$ is preserved in the subproblem and must still complete by that value, so the makespan cannot fall below it. We can therefore deduce the following Benders optimality cut:
\begin{small}
\begin{align}
C_{\max}^{sub} \;-\; \sum_{j_k \in \Pi^\vartheta} \bigl(C_{\max}^{sub} - \varphi(j_k,\hat{t}^\vartheta_{j_k})\bigr)\,(1-x_{j_k,\hat{t}^\vartheta_{j_k}}) \;\leq\; q.
\label{eq:optimality_cuts}
\end{align}
\end{small}
When the set $\Pi^\vartheta$ is empty (no critical energy-intensive task could be identified on the resource-flow graph), the sum is taken over the entire set of energy-intensive tasks scheduled in $\vartheta$, which conservatively recovers a valid cut.
 
\begin{proposition}
    Inequality~\eqref{eq:optimality_cuts} is a valid Benders optimality cut.
\end{proposition}
\begin{proof}
    To be a valid Benders optimality cut, the inequality must (i) cut off the current suboptimal master solution and (ii) preserve any other feasible integer solution. Let $\vartheta$ be the current suboptimal solution and denote
    $$Q = \left\{ x_{j_k,\hat{t}^\vartheta_{j_k}}\ \big|\ j_k \in \Pi^\vartheta \right\}$$
    the set of master assignments fixed by the critical energy-intensive tasks of $\vartheta$. For any feasible master solution $\varrho$, let
    $$R = \left\{ x_{j_k,\hat{t}^\varrho_{j_k}}\ \big|\ j_k \in \Pi^\vartheta \right\}$$
    denote the placements taken by those same tasks in $\varrho$, and define
    $$S^\varrho \;=\; \bigl\{\, j_k \in \Pi^\vartheta\ \big|\ \hat{t}^\varrho_{j_k} \neq \hat{t}^\vartheta_{j_k}\, \bigr\}$$
    so that $S^\varrho = \emptyset$ if and only f $Q \cap R = Q$. We distinguish two cases.
 
    \emph{Case $Q \cap R = Q$ (i.e., $S^\varrho = \emptyset$):} every factor $(1-x_{j,\hat{t}^\vartheta_j})$ vanishes and~\eqref{eq:optimality_cuts} reduces to $C_{\max}^{sub} \leq q$, which is correctly enforced since reproducing the critical chain of $\vartheta$ forces $C_{\max}^\varrho \geq C_{\max}^\vartheta = C_{\max}^{sub}$.
 
   \emph{Case $Q \cap R \neq Q$ (i.e., $S^\varrho \neq \emptyset$):} let $j_{k^\star}$ be the element of $S^\varrho$ of smallest index, that is, the earliest critical energy-intensive task of $\vartheta$ whose start time is changed in $\varrho$. By minimality of $k^\star$, every $j_l \in \Pi^\vartheta$ with $l < k^\star$ satisfies $\hat{t}^\varrho_{j_l} = \hat{t}^\vartheta_{j_l}$. In particular, the energy-intensive task achieving the maximum in the definition of $\varphi(j_{k^\star},\hat{t}^\vartheta_{j_{k^\star}})$ is placed identically in $\varrho$, hence still completes at time $\varphi(j_{k^\star},\hat{t}^\vartheta_{j_{k^\star}})$ and is preserved by the subproblem. Therefore,
    \begin{align}
        C_{\max}^\varrho \;\geq\; \varphi(j_{k^\star},\hat{t}_{j_{k^\star}}^\vartheta).
        \label{eq:lower_bound_makespan}
    \end{align}
    Evaluating the left-hand side (LHS) of~\eqref{eq:optimality_cuts} at $\varrho$ and using that each coefficient $(C_{\max}^\vartheta - \varphi(j_l,\hat{t}^\vartheta_{j_l}))$ is non-negative,
    \begin{small}
    $$C_{\max}^\vartheta \;-\; \sum_{j_l \in S^\varrho} \bigl(C_{\max}^\vartheta - \varphi(j_l,\hat{t}^\vartheta_{j_l})\bigr) \;\leq\; C_{\max}^\vartheta - \bigl(C_{\max}^\vartheta - \varphi(j_{k^\star},\hat{t}^\vartheta_{j_{k^\star}})\bigr) \;=\; \varphi(j_{k^\star},\hat{t}^\vartheta_{j_{k^\star}}),$$
    \end{small}
    so combining with~\eqref{eq:lower_bound_makespan} we obtain $q = C_{\max}^\varrho \geq$ LHS of~\eqref{eq:optimality_cuts}, as required. Hence~\eqref{eq:optimality_cuts} cuts off only suboptimal master solutions and preserves all feasible ones.
\end{proof}

\paragraph{Valid inequalities}
The optimality constraints update the value of the objective of the master problem. However, they are only evaluated for integer solutions. To reduce the master search tree, it is worth noting that it is possible to estimate a lower bound on the makespan based on the release date of a task.

\begin{proposition}
    For every energy-intensive task $j\in \mathcal{J}$, the makespan $q$ must satisfy the following constraints:
    \begin{small}\begin{equation}
    \label{eq:benders-minimal-makespan-1}
        \sum_{t \in \mathcal{I}} t x_{jt} + \MD(j,s) + p_s \leq q\quad \forall s \in Succ(j)
    \end{equation}\end{small}
\end{proposition}
\begin{proof}
    The makespan of a solution of our problem cannot be less than the minimum amount of time required to process the consecutive tasks, i.e., as if the resource constraints of the RCPSP were ignored and only precedence constraints were present. Thus, we can deduce that, for each energy-intensive task $j$, a valid makespan must be at least equal to the largest minimum distances between $j$ and its successors, i.e.:
    \begin{small}
    $$\forall j \in \mathcal{J}, \forall s \in Succ(j),\ C_{\max} \geq S_j +  \MD(j,s) + p_s$$
    \end{small}
Hence, we can infer the constraints~\eqref{eq:benders-minimal-makespan-1} from them.
\end{proof}

Let us define $\MD^*(j) = \max_{s \in Succ(j)} \{ \MD(j, s) \}$ the \emph{maximal minimal distance} over the successors of a task $j \in \mathcal{T}$.


\begin{corollary}
    For every energy-intensive task $j\in \mathcal{J}$, the makespan $q$ must satisfy the following constraint:
    \begin{small}
    \begin{equation}
    \label{eq:benders-minimal-makespan-2}
        q\geq\sum_{t \in \mathcal{I}} t x_{jt} + \MD^*(j)
    \end{equation}
    \end{small}
\end{corollary}

Constraints~\eqref{eq:benders-minimal-makespan-2} allow us to compute a lower bound for the makespan in the master problem, considering the starting times of the energy-intensive tasks. They greatly help the convergence of our LBBD. Indeed, Constraints~\eqref{eq:optimality_cuts} ultimately only come into play to correct an inaccurate estimate of the makespan, in cases where constraints in resource capacity prevent tasks from being completed as quickly as they could.

Also, we know that the makespan of an optimal solution of our problem cannot exceed by construction the $lb_{RCPSP}$ makespan. Hence, we can add the following constraint:
\begin{align}
\label{eq:makespan_lb_global}
    q \geq lb_{RCPSP}
\end{align}

This constraint~\eqref{eq:makespan_lb_global} contributes significantly to the convergence of our solution (see Figure~\ref{fig:impact_lb_q_computation_time}); indeed, it provides a strong assumption on the value of $q$, which helps to shrink the search space. In fairness, we also inject this constraint into the monolithic formulations. 


\paragraph{Additional remarks}
To conclude this section, we discuss an important remark. The proposed LBBD guarantees the optimality of the final solution as long as all the subproblems have been either solved to optimality or proved infeasible. However, depending on the size or the complexity of the subproblem, it may end up being hard to prove either the feasibility or the optimality of $Subproblem$ in a reasonable time. In these cases, we proceed as follows: if $Subproblem$ finds a feasible solution but cannot close the optimality gap, we insert in $Master$ an optimality constraint~\eqref{eq:optimality_cuts} where $C^{sub}_{\max}$ is the best makespan found so far. On the contrary, when $Subproblem$ can neither identify a feasible solution nor assert that no solution exists, we consider that the problem is infeasible, and we rule out the solution with a feasibility cut~\eqref{eq:lazy-constraint} with $\mathcal{I}nf = \mathcal{X}^\dagger$. Alas, both situations may lead to optimal solutions being cut off in the master.
Although CP approaches are known to be efficient in deciding feasibility, they may struggle to close the gap between the lower and upper bounds.
To tackle this issue that could impede the validity of our approach, we divide the experimental sections into two sections. In the first one, we will only consider instances in which the LBBD can certify the optimality of the provided solutions, i.e., all subproblems have either been solved to optimality or declared infeasible. In this section, the LBBD can thus be treated as an optimal method. We denote it \LBBDG.
In the next section, we will impose a time limit on the computation of the subproblem. By doing so, we may end up generating suboptimal cuts. We denote it \LBBD.

\section{Computational experiments}
\label{section:experiments}

In this section, we assess to what extent the proposed LBBD is a practical alternative to the two monolithic baselines, the compact ILP of Section~\ref{section:problem_statement} and the CP model of Section~\ref{sec:constraint_programming}, and we identify the regimes, in terms of instance density and weighting parameter $\alpha$, in which one approach is preferable to another. Throughout the section, we compare four methods.
The monolithic ILP warm-started from a feasible solution of the underlying RCPSP (found via a classical CP formulation of RCPSP), denoted \ILPF, serves as the reference, and the per-instance gain $\delta_\mu$ for any other method $\mu$ is calculated as a percentage according to Equation~\eqref{eq:delta}. The three remaining methods are: the same monolithic ILP warm-started from a 60 seconds run of \Freecp solution (\ILPC), the monolithic constraint-programming model (\Freecp), and the proposed LBBD detailed in Section~\ref{section:LBBD}, denoted \texttt{LBBD}.

The experiments are divided into three sections. After detailing the dataset, we investigate the impact of each \texttt{LBBD}'s components on the solution quality. Then, we compare our novel approach to the others over instances in which \texttt{LBBD} guarantees the optimality of its optimal solution, whereas in the last section, we observe the \texttt{LBBD} over all instances. Finally, we synthesize the results and critically evaluate the relevance and performance of each method with respect to the specific instance configurations.

All experiments were run on 2 x AMD\textsuperscript{\texttrademark} EPYC\textsuperscript{\texttrademark} 9124 CPUs 3.5 GHz equipped with 64 threads in total and 384 GB of RAM, running NixOS 25.05. Gurobi 13.0.1 and IBM CP Optimizer 22.1.1.0 are used to solve the ILP and CP models, respectively. All the code has been written in \verb|C++23|, compiled with Clang~19.1.7 and the \texttt{O2} optimization flag, using the Boost Graph Library~\citep{siek2001boost} to implement graph components.
All experiments are run on a single thread, and the solver's default pre-solve strategies are enabled. A time limit of 1800 seconds is imposed for all experiments. We impose a soft limit on the memory usage of the solvers of 50 GB. The other solver parameters are left at their default values, except for Gurobi's numeric focus set to 2.
In addition, the precomputation time of the paths SPACES graph and of the $\MD$ function~\eqref{eq:md_definition} are not reported in our experiments, because they never exceed one second in every instance of our datasets, which is negligible.

\subsection{Datasets}
\label{section:dataset}
We generate our instances based on the RCPSP instances of the standard \href{https://www.om-db.wi.tum.de/psplib/}{PSPLIB dataset}~\citep{sprecher1996psplib} that contain sets of instances with $\{30,60,90,120\}$ activities.
To increase the number of tasks in our instances, we defined an approach to merge multiple small RCPSP instances from PSPLIB until we obtained an instance of the desired size\footnote{The generator and the instances are available in \href{https://github.com/corentinjuvigny/rcpsp_ms_tou_generator}{github.com/corentinjuvigny/rcpsp\_ms\_tou\_generator}}.
When we merge two smaller instances, the resources' capacities are set to the maximum of the two.
Then, these two merged projects are connected to the dummy start and end tasks, creating a parallel network of the two projects.
The process continues until the requested instance size is met (or exceeded).
In this way, the complexity of the original RCPSP instances is preserved.
See an example of the precedence structure in the instances generated in the supplementary materials.
We consider for each instance five resources ($R_0, \dots, R_4$).
Then, we selected a given percentage of tasks in the newly created instance that we marked as ``energy-intensive'' (i.e., we set all of its other resource requirements to 0, then we added a requirement of 1 for $R_0$).
These tasks are selected uniformly throughout the project. Indeed, a more in-depth analysis revealed no significant effect on the allocation or positioning of energy-intensive tasks within the project workflow.
The energy-intensive resource $R_0$ uses a transition diagram from Figure~\ref{fig:example-func-power-time}, which is based on a diagram proposed by~\cite{shrouf2014optimizing}.
Finally, we generate a cost vector $c$ of energy prices based on real TOU data collected from the Czech OTE\footnote{\href{https://www.ote-cr.cz/en}{https://www.ote-cr.cz/en}}.
We partition our test instances according to their ratio $\rho:=|\mathcal{J}|/|\mathcal{T}|$ into three categories:
\begin{itemize}[itemsep=0pt, parsep=0pt]
    \item Sparse: $\rho \approx 3\text{ - }5\%$
    \item Standard: $\rho \approx 15\text{ - }20\%$
    \item Dense: $\rho \approx 50\%$
\end{itemize}

\noindent For each of those categories, we generated instance sets of increasing sizes. Each set contains 20 instances.
Sparse instances have up to 640 tasks, standard instances up to 480 tasks, and dense instances up to 224 tasks.
Each instance has its own distinct cost vector.
Tables~C.1 and C.2 of the supplementary materials provide additional details on the number of tasks and the time horizon for each set.
The datasets containing the sparse, standard, and dense instance sets are available at \cite{DVN/MQS0NY_2025}, \cite{DVN/EPX98Z_2025}, and \cite{DVN/SH2DSA_2025}, respectively.

\subsection{Analysis of \texttt{LBBD} components}

First, we analyze the components of the proposed \texttt{LBBD} approach to assess their impact on the overall quality of the method. We do all the experiments on the six smallest standard sets. Only instances for which the \LBBD guarantees optimality of its solutions are considered in this paragraph.

\paragraph{Impact of global lower bound on makespan (constraint~\eqref{eq:makespan_lb_global})}

First, let us analyze the impact of introducing a global lower bound on the makespan.
Figure~\ref{fig:impact_lb_q_computation_time} compares the computation time between \LBBDG with (\emph{\texttt{LBBD-lbq}}) and without (\emph{\texttt{LBBD-nolbq}}) the lower bound on the makespan. We see that the closer $\alpha$ is to $0$, the more impactful the lower bound is on the convergence speed of our approach, which is negligible when $\alpha  = 1$. The behavior is reasonable because the smaller $\alpha$, the greater the contribution of the makespan to the objective value. Table~\ref{tab:average_cuts} corroborates the analysis by showing that the approach without constraint~\ref{eq:makespan_lb_global} requires many more Benders cuts than the approaches with it, especially when $\alpha$ is close to $0$. In fairness, this global lower bound is also added in ILP formulations.

\begin{figure}[ht]
    \centering
    \resizebox{0.75\linewidth}{!}{\begin{tikzpicture}[baseline=(current bounding box.south)]

\pgfplotsset{
  m lbbd/.style        ={solid, color=blue!60,   fill opacity=0.45, line width=1pt, mark=none, no marks},
  m lbbdnouc/.style    ={solid, color=red!60,    fill opacity=0.45, line width=1pt, mark=none, no marks},
  m lbbdnolbq/.style   ={solid, color=violet!60, fill opacity=0.45, line width=1pt, mark=none, no marks},
  m lbbdnolbqcp/.style   ={solid, color=green!60, fill opacity=0.45, line width=1pt, mark=none, no marks},
}

\newcommand{\Ntot}{60}

\begin{groupplot}[
  group style={group size=2 by 2, horizontal sep=1.4cm, vertical sep=1.6cm},
  width=0.45\textwidth, height=0.28\textwidth, scale only axis,
  xlabel={Time [s]}, ylabel={\% instances proven optimal},
  xmin=0, xmax=1800,
  xmode=log, log basis x=10,
  grid=major, no marks,
  legend cell align=left,
  legend style={
    at={(0.02,0.98)}, anchor=north west, font=\scriptsize,
    fill=white, fill opacity=0.9, text opacity=1, draw=black!30,
  },
  title style={font=\small},
]

\nextgroupplot[title={$\alpha = 0.25$}, ymin=-1, ymax=100, xlabel={}]
  \addplot[m lbbdnouc,  restrict x to domain=0:1800]
    table [x=t, y expr=100*\thisrow{lbbd-wcp-wnouc}/\Ntot,        col sep=comma]
    {fig/tikz_data/cactus_lbqEffect_alpha25.csv};  \addlegendentry{\texttt{LBBD-lbq}}
  \addplot[m lbbd,      restrict x to domain=0:1800]
    table [x=t, y expr=100*\thisrow{lbbd-wcp-wuc-cpws}/\Ntot,     col sep=comma]
    {fig/tikz_data/cactus_lbqEffect_alpha25.csv};  \addlegendentry{\texttt{LBBD-lbq-cpws}}
  \addplot[m lbbdnolbq, restrict x to domain=0:1800]
    table [x=t, y expr=100*\thisrow{lbbd-wcp-wnouc-nolbq}/\Ntot,  col sep=comma]
    {fig/tikz_data/cactus_lbqEffect_alpha25.csv};  \addlegendentry{\texttt{LBBD-nolbq}}
  \addplot[m lbbdnolbqcp, restrict x to domain=0:1800]
    table [x=t, y expr=100*\thisrow{lbbd-wcp-nolbq-cpws}/\Ntot,  col sep=comma]
    {fig/tikz_data/cactus_lbqEffect_alpha25.csv};  \addlegendentry{\texttt{LBBD-nolbq-cpws}};

\nextgroupplot[title={$\alpha = 0.5$}, ymin=-1, ymax=90, xlabel={}, ylabel={}]
  \addplot[m lbbdnouc,  restrict x to domain=0:1800]
    table [x=t, y expr=100*\thisrow{lbbd-wcp-wnouc}/\Ntot,        col sep=comma]
    {fig/tikz_data/cactus_lbqEffect_alpha50.csv};
  \addplot[m lbbd,      restrict x to domain=0:1800]
    table [x=t, y expr=100*\thisrow{lbbd-wcp-wuc-cpws}/\Ntot,     col sep=comma]
    {fig/tikz_data/cactus_lbqEffect_alpha50.csv};
  \addplot[m lbbdnolbq, restrict x to domain=0:1800]
    table [x=t, y expr=100*\thisrow{lbbd-wcp-wnouc-nolbq}/\Ntot,  col sep=comma]
    {fig/tikz_data/cactus_lbqEffect_alpha50.csv};
  \addplot[m lbbdnolbqcp, restrict x to domain=0:1800]
    table [x=t, y expr=100*\thisrow{lbbd-wcp-nolbq-cpws}/\Ntot,  col sep=comma]
    {fig/tikz_data/cactus_lbqEffect_alpha50.csv};

\nextgroupplot[title={$\alpha = 0.75$}, ymin=-1, ymax=60]
  \addplot[m lbbdnouc,  restrict x to domain=0:1800]
    table [x=t, y expr=100*\thisrow{lbbd-wcp-wnouc}/\Ntot,        col sep=comma]
    {fig/tikz_data/cactus_lbqEffect_alpha75.csv};
  \addplot[m lbbd,      restrict x to domain=0:1800]
    table [x=t, y expr=100*\thisrow{lbbd-wcp-wuc-cpws}/\Ntot,     col sep=comma]
    {fig/tikz_data/cactus_lbqEffect_alpha75.csv};
  \addplot[m lbbdnolbq, restrict x to domain=0:1800]
    table [x=t, y expr=100*\thisrow{lbbd-wcp-wnouc-nolbq}/\Ntot,  col sep=comma]
    {fig/tikz_data/cactus_lbqEffect_alpha75.csv};
  \addplot[m lbbdnolbqcp, restrict x to domain=0:1800]
    table [x=t, y expr=100*\thisrow{lbbd-wcp-nolbq-cpws}/\Ntot,  col sep=comma]
    {fig/tikz_data/cactus_lbqEffect_alpha75.csv};

\nextgroupplot[title={$\alpha = 1.0$}, ymin=-1, ymax=110, ylabel={}]
  \addplot[m lbbdnouc,  restrict x to domain=0:1800]
    table [x=t, y expr=100*\thisrow{lbbd-wcp-wnouc}/\Ntot,        col sep=comma]
    {fig/tikz_data/cactus_lbqEffect_alpha100.csv};
  \addplot[m lbbd,      restrict x to domain=0:1800]
    table [x=t, y expr=100*\thisrow{lbbd-wcp-wuc-cpws}/\Ntot,     col sep=comma]
    {fig/tikz_data/cactus_lbqEffect_alpha100.csv};
  \addplot[m lbbdnolbq, restrict x to domain=0:1800]
    table [x=t, y expr=100*\thisrow{lbbd-wcp-wnouc-nolbq}/\Ntot,  col sep=comma]
    {fig/tikz_data/cactus_lbqEffect_alpha100.csv};
  \addplot[m lbbdnolbqcp, restrict x to domain=0:1800]
    table [x=t, y expr=100*\thisrow{lbbd-wcp-nolbq-cpws}/\Ntot,  col sep=comma]
    {fig/tikz_data/cactus_lbqEffect_alpha100.csv};

\end{groupplot}
\end{tikzpicture}}
    \caption{Comparison of the percentage of solutions solved to optimality in less than $t$ seconds between the \LBBDG~with (\emph{\texttt{LBBD-lbq}}) / without (\emph{\texttt{LBBD-nolbq}}) constraint~\eqref{eq:makespan_lb_global} and the warmstarted with a 60s-timelimit solution of \Freecp (\emph{\texttt{LBBD-lbq-cpws} / \texttt{LBBD-nolbq-cpws}}).}
    \label{fig:impact_lb_q_computation_time}
\end{figure}

\begin{table}[ht]
\centering
\begin{small}
\begin{tabular}{|ll||r|r|}
\toprule
$\alpha$  & method  & Avg optimality cuts & Avg feasibility cuts \\
\midrule
\multirow[c]{3}{*}{0.25}  & \texttt{LBBD-nolbq} & 1306.97 & 26.30 \\
 & \texttt{LBBD-lbq} & 405.12 & 15.50 \\
 & \texttt{LBBD-lbq-cpws} & 106.17 & 10.07 \\
\cline{1-4}
\multirow[c]{3}{*}{0.50} & \texttt{LBBD-nolbq} & 192.65 & 13.03 \\
 & \texttt{LBBD-lbq} & 56.13 & 4.88 \\
 & \texttt{LBBD-lbq-cpws} & 16.98 & 3.20 \\
\cline{1-4}
\multirow[c]{3}{*}{0.75} & \texttt{LBBD-nolbq} & 24.08 & 1.27 \\
 & \texttt{LBBD-lbq} & 21.72 & 1.35 \\
 & \texttt{LBBD-lbq-cpws} & 10.98 & 0.67 \\
\cline{1-4}
\multirow[c]{3}{*}{1.00} & \texttt{LBBD-nolbq} & 8.98 & 1.62 \\
 & \texttt{LBBD-lbq} & 8.98 & 1.62 \\
 & \texttt{LBBD-lbq-cpws} & 5.42 & 1.23 \\
\bottomrule
\end{tabular}
\end{small}
\caption{Average number of optimality and feasibility cuts per \LBBDG variant on the reduced sets of standard instances.}
\label{tab:average_cuts}
\end{table}

\paragraph{Impact of the \Freecp warmstart procedure}

Figure~\ref{fig:impact_lb_q_computation_time} also provides the computation time performance of the \LBBDG warmstarted by a solution found by a 60s run of \Freecp (\emph{\texttt{LBBD-lbq-cpws}}). Clearly, we see that during the first 60 seconds, the latter proves the optimality of only a few instances. Indeed, this time interval corresponds to the computation of the warmstart solution by \Freecp. But once this warmstart solution yielded, the percentage of solutions proven to be optimal skyrockets, surpassing the percentage of the non-CP-warmstarted version, showing that this warmstart solution is a good starting point. In addition, Table~\ref{tab:average_cuts} indicates that the CP-warmstarted version requires fewer Benders cuts, which means that it explores fewer integer solutions. Moreover, Figure~\ref{fig:impact_lb_q_computation_time} shows that the impact of the constraint~\eqref{eq:makespan_lb_global} on the convergence time is greater than the warm starting, since \emph{\texttt{LBBD-nolbq-cpws}} does not noticeably provide better results than \emph{\texttt{LBBD-nolbq}}.

\paragraph{Time limits tuning}
The time limit of the subproblem influences the performance of the \texttt{LBBD}. If set too tight, we may rule out an optimal solution, but set too long, it may impede the overall search by allotting too much time in closing the gap of one subproblem instead of exploring other ones. Table~\ref{tab:evoluation-obj-along-subtl} shows, for two sparse instances (in which the time limit of the subproblem is critical), that by increasing the time limit the objective value decreases until it reaches a plateau from which it goes back up.
We employ \textsc{iRace} \citep{irace} to find adequate values of the time limits of \emph{Subproblem} and \Freecp warmstart that maximize the ratio of the quality of the solution found over time consumption.
The results gave a warmstart computation time limit of 60s and a subproblem computation time of 20s.

\begin{table}[]
    \centering
    \begin{small}
    \begin{tabular}{|l|rrr|rrr|}
    \toprule
        Instance & \multicolumn{3}{c|}{sparse\_5\_5} & \multicolumn{3}{c|}{sparse\_10\_3} \\
        Sub-TL & obj & Time (s) & \#Non-opt & obj & Time (s) & \#Non-opt \\
    \midrule
        1  & 1.12471 & 155  & 93 & 1.07633 & 1739.22 & 1744 \\
        5  & 1.12077 & 400  & 67 & 1.07804 & 1843.72 & 368  \\
        10 & 1.11915 & 731  & 66 & 1.07368 & 1858.45 & 185  \\
        15 & 1.11915 & 1065 & 66 & 1.07938 & 1847.40 & 122  \\
        20 & 1.11915 & 1421 & 67 & 1.07938 & 1915.40 & 95   \\
        30 & 1.11915 & 2100 & 67 & 1.08874 & 1944.12 & 63   \\
    \bottomrule
    \end{tabular}
    \end{small}
    \caption{Evolution of the objective and computation time of two sparse instances when time limit of the subproblem changes ($\alpha={0.5}$).}
    \label{tab:evoluation-obj-along-subtl}
\end{table}

In summary, the additional procedures and constraints proposed significantly improve our LBBD method. Therefore, we will henceforth consider that they are included in our LBBD method.

\subsection{Comparison \LBBDG\,--- \ILP\,--- \Freecp}
\label{section:analysis_lbbdg}

In this subsection, we will only consider instances for which the LBBD approach guarantees the optimality of its optimal solution, i.e., all of its subproblems have been either solved to optimality or proved infeasible.

We propose to compare the following approaches: the monolithic ILP formulation, warm-started from a feasible solution (\emph{\ILPF}), the monolithic ILP warm-started from a 60s-\Freecp solution (\emph{\ILPC}), the monolithic CP approach (\emph{\Freecp}), and the proposed Benders decomposition approach (\emph{\LBBDG}).
The \ILPF is our reference method, i.e., for any other method $\mu$, we compute the delta of this reference method as:
\begin{small}
\begin{align}
\label{eq:delta}
    \delta_{\mu} = \frac{obj_{\ILPF}-obj_\mu}{|obj_{\ILPF}|}\cdot 100
\end{align}
\end{small}

Additional graphics, including box plots and more detailed tables, are available in the supplementary materials.

\paragraph{Standard instances}
Figure~\ref{fig:standard_optimality_pct} shows the percentage of instances of standard density over time for each method. It reveals that the \LBBDG approach outperforms the others by proving the optimality of many more instances, and the advantage is most striking at $\alpha = 1$, where it proves close to 95\% of the instances optimal against roughly 65\% for \ILPC and below 5\% for \Freecp. As is typical for a pure CP model, \Freecp proves the fewest instances optimal of all four methods.
When we observe the quality of the solutions that each method produces (Table~\ref{tab:summary-performance-combined-standard}, measured by the gain~\eqref{eq:delta} over \ILPF), we see that \LBBDG achieves better gains than the other methods when $\alpha \in \{0.25,1.0\}$ and is slightly better than \Freecp in the biggest set of instances when $\alpha \in \{0.5, 0.75\}$ but is beaten by the latter in smaller instances. \ILPC is almost always behind both \Freecp and \LBBDG, except when $\alpha = 1$ where \Freecp noticeably underperforms even \ILPF. However, for this value of $\alpha$ \LBBDG has a clearest advantage over other methods.

\begin{figure}
    \centering
    \resizebox{0.75\linewidth}{!}{\begin{tikzpicture}[baseline=(current bounding box.south)]

\pgfplotsset{
  m freecp/.style    ={solid, color=green!60, fill opacity=0.45, line width=1pt, mark=none, no marks},
  m ilpcpws/.style   ={solid, color=orange!60, fill opacity=0.45, line width=1pt, mark=none, no marks},
  m ilpfsws/.style   ={solid, color=gray!60, fill opacity=0.45,  line width=1pt, mark=none, no marks},
  m lbbd/.style      ={solid, color=blue!60, fill opacity=0.45, line width=1pt, mark=none, no marks},
}

\newcommand{\Ntot}{300}

\begin{groupplot}[
  group style={group size=2 by 2, horizontal sep=1.4cm, vertical sep=1.6cm},
  width=0.45\textwidth, height=0.28\textwidth, scale only axis,
  xlabel={Time [s]}, ylabel={\% instances proven optimal},
  xmin=0, xmax=1800,
  xmode=log, log basis x=10,
  grid=major, no marks,
  legend cell align=left,
  legend style={
    at={(0.02,0.98)}, anchor=north west, font=\scriptsize,
    fill=white, fill opacity=0.9, text opacity=1, draw=black!30,
  },
  title style={font=\small},
]

\nextgroupplot[title={$\alpha = 0.25$}, ymin=-1, ymax=80, xlabel={}]
  \addplot[m freecp,  restrict x to domain=0:1800]
    table [x=t, y expr=\thisrow{freecp},            col sep=comma]
    {fig/tikz_data/cactus_standard_alpha25.csv};  \addlegendentry{\Freecp}
  \addplot[m ilpcpws, restrict x to domain=0:1800]
    table [x=t, y expr=\thisrow{ilp-cpws},          col sep=comma]
    {fig/tikz_data/cactus_standard_alpha25.csv};  \addlegendentry{\ILPC}
  \addplot[m ilpfsws, restrict x to domain=0:1800]
    table [x=t, y expr=\thisrow{ilp-fsws},          col sep=comma]
    {fig/tikz_data/cactus_standard_alpha25.csv};  \addlegendentry{\ILPF}
  \addplot[m lbbd,    restrict x to domain=0:1800]
    table [x=t, y expr=\thisrow{lbbd-wcp-wuc-cpws}, col sep=comma]
    {fig/tikz_data/cactus_standard_alpha25.csv};  \addlegendentry{\LBBDG}

\nextgroupplot[title={$\alpha = 0.5$}, ymin=-1, ymax=50, xlabel={}, ylabel={}]
  \addplot[m freecp,  restrict x to domain=0:1800]
    table [x=t, y expr=\thisrow{freecp},            col sep=comma]
    {fig/tikz_data/cactus_standard_alpha50.csv};
  \addplot[m ilpcpws, restrict x to domain=0:1800]
    table [x=t, y expr=\thisrow{ilp-cpws},          col sep=comma]
    {fig/tikz_data/cactus_standard_alpha50.csv};
  \addplot[m ilpfsws, restrict x to domain=0:1800]
    table [x=t, y expr=\thisrow{ilp-fsws},          col sep=comma]
    {fig/tikz_data/cactus_standard_alpha50.csv};
  \addplot[m lbbd,    restrict x to domain=0:1800]
    table [x=t, y expr=\thisrow{lbbd-wcp-wuc-cpws}, col sep=comma]
    {fig/tikz_data/cactus_standard_alpha50.csv};

\nextgroupplot[title={$\alpha = 0.75$}, ymin=-1, ymax=26]
  \addplot[m freecp,  restrict x to domain=0:1800]
    table [x=t, y expr=\thisrow{freecp},            col sep=comma]
    {fig/tikz_data/cactus_standard_alpha75.csv};
  \addplot[m ilpcpws, restrict x to domain=0:1800]
    table [x=t, y expr=\thisrow{ilp-cpws},          col sep=comma]
    {fig/tikz_data/cactus_standard_alpha75.csv};
  \addplot[m ilpfsws, restrict x to domain=0:1800]
    table [x=t, y expr=\thisrow{ilp-fsws},          col sep=comma]
    {fig/tikz_data/cactus_standard_alpha75.csv};
  \addplot[m lbbd,    restrict x to domain=0:1800]
    table [x=t, y expr=\thisrow{lbbd-wcp-wuc-cpws}, col sep=comma]
    {fig/tikz_data/cactus_standard_alpha75.csv};

\nextgroupplot[title={$\alpha = 1.0$}, ymin=-1, ymax=100, ylabel={}]
  \addplot[m freecp,  restrict x to domain=0:1800]
    table [x=t, y expr=\thisrow{freecp},            col sep=comma]
    {fig/tikz_data/cactus_standard_alpha100.csv};
  \addplot[m ilpcpws, restrict x to domain=0:1800]
    table [x=t, y expr=\thisrow{ilp-cpws},          col sep=comma]
    {fig/tikz_data/cactus_standard_alpha100.csv};
  \addplot[m ilpfsws, restrict x to domain=0:1800]
    table [x=t, y expr=\thisrow{ilp-fsws},          col sep=comma]
    {fig/tikz_data/cactus_standard_alpha100.csv};
  \addplot[m lbbd,    restrict x to domain=0:1800]
    table [x=t, y expr=\thisrow{lbbd-wcp-wuc-cpws}, col sep=comma]
    {fig/tikz_data/cactus_standard_alpha100.csv};

\end{groupplot}
\end{tikzpicture}}
    \caption{Comparison of the percentage of solutions solved to optimality over time for each method in standard sets.}
    \label{fig:standard_optimality_pct}
\end{figure}

\begin{table}[ht]
\centering
\begin{footnotesize}
\begin{tabular}{|l|rrr|rrr|}
\toprule
 & \Freecp & \ILPC & \LBBDG & \Freecp & \ILPC & \LBBDG \\
\cmidrule{1-7}
Set & \multicolumn{3}{c|}{$\alpha = 0.25$} & \multicolumn{3}{c|}{$\alpha = 0.5$} \\
\midrule
1 & \textbf{0.00} & \textbf{0.00} & \textbf{0.00} & \textbf{0.00} & \textbf{0.00} & \textbf{0.00} \\
2 & \textbf{0.00} & \textbf{0.00} & \textbf{0.00} & \textbf{0.00} & \textbf{0.00} & \textbf{0.00} \\
3 & 2.90 & 2.93 & \textbf{2.95} & \textbf{0.95} & 0.93 & 0.89 \\
4 & 6.35 & 5.79 & \textbf{6.55} & \textbf{5.07} & 4.83 & 4.91 \\
5 & 6.55 & 6.44 & \textbf{6.66} & \textbf{4.65} & 4.62 & 4.63 \\
6 & 8.27 & 7.71 & \textbf{8.51} & \textbf{6.44} & 6.09 & 6.35 \\
7 & 6.53 & 6.28 & \textbf{6.54} & \textbf{5.16} & 4.84 & 4.94 \\
8 & \textbf{6.89} & 6.61 & 6.71 & \textbf{5.16} & 4.39 & 5.01 \\
9 & 4.83 & 6.09 & \textbf{6.30} & \textbf{4.22} & 3.98 & 4.16 \\
10 & \textbf{4.77} & 3.46 & 4.76 & \textbf{3.82} & 2.72 & 3.30 \\
11 & \textbf{4.23} & 3.16 & 4.19 & \textbf{2.98} & 1.68 & 2.94 \\
12 & \textbf{4.98} & 4.21 & 4.97 & 3.42 & 2.48 & \textbf{3.45} \\
13 & \textbf{4.64} & 3.45 & 4.62 & \textbf{3.25} & 2.54 & 3.24 \\
14 & 4.78 & 5.08 & \textbf{5.63} & 4.27 & 3.35 & \textbf{4.96} \\
15 & 4.66 & 4.30 & \textbf{5.15} & 4.29 & 3.48 & \textbf{4.59} \\
\textit{Mean} & \textit{4.69} & \textit{4.37} & \textit{\textbf{4.90}} & \textit{\textbf{3.58}} & \textit{3.06} & \textit{3.56} \\
\midrule
 & \multicolumn{3}{c|}{$\alpha = 0.75$} & \multicolumn{3}{c|}{$\alpha = 1.0$} \\
\midrule
1 & \textbf{0.00} & \textbf{0.00} & \textbf{0.00} & \textbf{0.00} & \textbf{0.00} & \textbf{0.00} \\
2 & \textbf{0.00} & \textbf{0.00} & \textbf{0.00} & \textbf{0.00} & \textbf{0.00} & \textbf{0.00} \\
3 & 0.47 & \textbf{0.77} & 0.69 & \textbf{0.00} & \textbf{0.00} & \textbf{0.00} \\
4 & 2.57 & 2.54 & \textbf{2.63} & -0.06 & \textbf{0.00} & \textbf{0.00} \\
5 & 2.22 & 2.28 & \textbf{2.67} & -0.38 & \textbf{0.00} & \textbf{0.00} \\
6 & 3.69 & 3.27 & \textbf{3.70} & -1.93 & -0.00 & \textbf{0.00} \\
7 & \textbf{4.06} & 2.72 & 3.99 & -3.08 & 0.02 & \textbf{0.02} \\
8 & \textbf{15.32} & 6.95 & 10.30 & -4.17 & 0.16 & \textbf{0.17} \\
9 & 3.17 & 2.96 & \textbf{3.59} & -3.38 & 0.74 & \textbf{0.74} \\
10 & \textbf{3.28} & 2.58 & 3.00 & 2.29 & 3.73 & \textbf{5.77} \\
11 & 2.29 & 1.82 & \textbf{2.67} & 1.12 & 3.85 & \textbf{5.28} \\
12 & 2.12 & 1.76 & \textbf{2.37} & 1.88 & 5.11 & \textbf{6.53} \\
13 & 3.05 & 2.26 & \textbf{3.91} & 5.86 & 9.86 & \textbf{11.71} \\
14 & 3.76 & 2.55 & \textbf{4.35} & 11.87 & 12.63 & \textbf{17.85} \\
15 & 3.58 & 2.80 & \textbf{3.69} & 11.58 & 14.71 & \textbf{17.30} \\
\textit{Mean} & \textit{\textbf{3.30}} & \textit{2.35} & \textit{3.17} & \textit{1.44} & \textit{3.39} & \textit{\textbf{4.36}} \\
\bottomrule
\end{tabular}
\end{footnotesize}
\caption{Performance comparison relative to ILP-fsws across different $\alpha$ values on standard instance sets. The results are given as percentages. The best approaches per set is reported in bold.}
\label{tab:summary-performance-combined-standard}
\end{table}

\paragraph{Dense instances}
In dense instances, \LBBDG stays superior compared to the other methods, both in objective value (see Table~\ref{tab:summary-performance-dense-combined}) and in computation time (see Figure~\ref{fig:dense_optimality_pct}). \LBBDG proves the most instances optimal at all $\alpha$ values (Figure~\ref{fig:dense_optimality_pct}) and delivers the best average gain over \ILPF at $\alpha \in \{0.25, 0.75, 1.0\}$ (averaging respectively 3.62\%, 4.90\% and 6.72\% across sets), tied with \Freecp at $\alpha = 0.5$ (averaging 4.48\% vs 4.52\%), as shown in Table~\ref{tab:summary-performance-dense-combined}. As in the standard case, the decomposition's edge becomes more pronounced on the larger instances. In terms of objective value, \Freecp again improves on \ILPC for the smaller $\alpha$ (e.g. 3.43\% versus 2.47\% in $\alpha = 0.25$), but the ordering reverses at $\alpha = 1.00$, where \ILPC (3.76\%) overtakes \Freecp (3.23\%).

\begin{figure}
    \centering
    \resizebox{0.75\linewidth}{!}{\begin{tikzpicture}[baseline=(current bounding box.south)]

\pgfplotsset{
  m freecp/.style    ={solid, color=green!60, fill opacity=0.45, line width=1pt, mark=none, no marks},
  m ilpcpws/.style   ={solid, color=orange!60, fill opacity=0.45, line width=1pt, mark=none, no marks},
  m ilpfsws/.style   ={solid, color=gray!60, fill opacity=0.45,  line width=1pt, mark=none, no marks},
  m lbbd/.style      ={solid, color=blue!60, fill opacity=0.45, line width=1pt, mark=none, no marks},
}

\newcommand{\Ntot}{140}

\begin{groupplot}[
  group style={group size=2 by 2, horizontal sep=1.4cm, vertical sep=1.6cm},
  width=0.45\textwidth, height=0.28\textwidth, scale only axis,
  xlabel={Time [s]}, ylabel={\% instances proven optimal},
  xmin=0, xmax=1800,
  xmode=log, log basis x=10,
  grid=major, no marks,
  legend cell align=left,
  legend style={
    at={(0.02,0.98)}, anchor=north west, font=\scriptsize,
    fill=white, fill opacity=0.9, text opacity=1, draw=black!30,
  },
  title style={font=\small},
]

\nextgroupplot[title={$\alpha = 0.25$}, ymin=-1, ymax=60, xlabel={}]
  \addplot[m freecp,  restrict x to domain=0:1800]
    table [x=t, y expr=\thisrow{freecp},            col sep=comma]
    {fig/tikz_data/cactus_dense_alpha25.csv};  \addlegendentry{\Freecp}
  \addplot[m ilpcpws, restrict x to domain=0:1800]
    table [x=t, y expr=\thisrow{ilp-cpws},          col sep=comma]
    {fig/tikz_data/cactus_dense_alpha25.csv};  \addlegendentry{\ILPC}
  \addplot[m ilpfsws, restrict x to domain=0:1800]
    table [x=t, y expr=\thisrow{ilp-fsws},          col sep=comma]
    {fig/tikz_data/cactus_dense_alpha25.csv};  \addlegendentry{\ILPF}
  \addplot[m lbbd,    restrict x to domain=0:1800]
    table [x=t, y expr=\thisrow{lbbd-wcp-wuc-cpws}, col sep=comma]
    {fig/tikz_data/cactus_dense_alpha25.csv};  \addlegendentry{\LBBDG}

\nextgroupplot[title={$\alpha = 0.5$}, ymin=-1, ymax=40, xlabel={}, ylabel={}]
  \addplot[m freecp,  restrict x to domain=0:1800]
    table [x=t, y expr=\thisrow{freecp},            col sep=comma]
    {fig/tikz_data/cactus_dense_alpha50.csv};
  \addplot[m ilpcpws, restrict x to domain=0:1800]
    table [x=t, y expr=\thisrow{ilp-cpws},          col sep=comma]
    {fig/tikz_data/cactus_dense_alpha50.csv};
  \addplot[m ilpfsws, restrict x to domain=0:1800]
    table [x=t, y expr=\thisrow{ilp-fsws},          col sep=comma]
    {fig/tikz_data/cactus_dense_alpha50.csv};
  \addplot[m lbbd,    restrict x to domain=0:1800]
    table [x=t, y expr=\thisrow{lbbd-wcp-wuc-cpws}, col sep=comma]
    {fig/tikz_data/cactus_dense_alpha50.csv};

\nextgroupplot[title={$\alpha = 0.75$}, ymin=-1, ymax=25]
  \addplot[m freecp,  restrict x to domain=0:1800]
    table [x=t, y expr=\thisrow{freecp},            col sep=comma]
    {fig/tikz_data/cactus_dense_alpha75.csv};
  \addplot[m ilpcpws, restrict x to domain=0:1800]
    table [x=t, y expr=\thisrow{ilp-cpws},          col sep=comma]
    {fig/tikz_data/cactus_dense_alpha75.csv};
  \addplot[m ilpfsws, restrict x to domain=0:1800]
    table [x=t, y expr=\thisrow{ilp-fsws},          col sep=comma]
    {fig/tikz_data/cactus_dense_alpha75.csv};
  \addplot[m lbbd,    restrict x to domain=0:1800]
    table [x=t, y expr=\thisrow{lbbd-wcp-wuc-cpws}, col sep=comma]
    {fig/tikz_data/cactus_dense_alpha75.csv};

\nextgroupplot[title={$\alpha = 1.0$}, ymin=-1, ymax=80, ylabel={}]
  \addplot[m freecp,  restrict x to domain=0:1800]
    table [x=t, y expr=\thisrow{freecp},            col sep=comma]
    {fig/tikz_data/cactus_dense_alpha100.csv};
  \addplot[m ilpcpws, restrict x to domain=0:1800]
    table [x=t, y expr=\thisrow{ilp-cpws},          col sep=comma]
    {fig/tikz_data/cactus_dense_alpha100.csv};
  \addplot[m ilpfsws, restrict x to domain=0:1800]
    table [x=t, y expr=\thisrow{ilp-fsws},          col sep=comma]
    {fig/tikz_data/cactus_dense_alpha100.csv};
  \addplot[m lbbd,    restrict x to domain=0:1800]
    table [x=t, y expr=\thisrow{lbbd-wcp-wuc-cpws}, col sep=comma]
    {fig/tikz_data/cactus_dense_alpha100.csv};

\end{groupplot}
\end{tikzpicture}}
    \caption{Comparison of the percentage of solutions solved to optimality over time for each method in dense sets.}
    \label{fig:dense_optimality_pct}
\end{figure}

\begin{table}[ht]
\centering
\begin{footnotesize}
\begin{tabular}{|r|rrr|rrr|}
\toprule
    & \Freecp & \ILPC & \LBBDG & \Freecp & \ILPC & \LBBDG \\
\cmidrule{2-7}
Set & \multicolumn{3}{c|}{$\alpha = 0.25$} & \multicolumn{3}{c|}{$\alpha = 0.5$} \\
\midrule
1 & \textbf{0.00} & \textbf{0.00} & \textbf{0.00} & \textbf{0.00} & \textbf{0.00} & \textbf{0.00} \\
2 & \textbf{0.00} & \textbf{0.00} & \textbf{0.00} & \textbf{0.00} & \textbf{0.00} & \textbf{0.00} \\
3 & \textbf{0.00} & \textbf{0.00} & \textbf{0.00} & \textbf{0.00} & \textbf{0.00} & \textbf{0.00} \\
4 & \textbf{1.64} & 1.39 & 0.86 & 2.86 & \textbf{3.40} & 2.34 \\
5 & 7.29 & 5.60 & \textbf{10.41} & 12.65 & 6.69 & \textbf{12.82} \\
6 & 7.83 & 5.92 & \textbf{8.56} & 8.32 & 5.11 & \textbf{9.64} \\
7 & \textbf{7.22} & 4.36 & 5.51 & \textbf{7.81} & 4.73 & 6.59 \\
\textit{Mean} & \textit{3.43} & \textit{2.47} & \textit{\textbf{3.62}} & \textit{\textbf{4.52}} & \textit{2.85} & \textit{4.48} \\
\midrule
  & \multicolumn{3}{c|}{$\alpha = 0.75$} & \multicolumn{3}{c|}{$\alpha = 1.0$} \\
\midrule
1 & \textbf{0.00} & \textbf{0.00} & \textbf{0.00} & \textbf{0.00} & \textbf{0.00} & \textbf{0.00} \\
2 & \textbf{0.02} & 0.00 & -0.01 & -1.72 & \textbf{0.00} & \textbf{0.00} \\
3 & \textbf{1.37} & 1.26 & 1.15 & -3.57 & 0.25 & \textbf{0.25} \\
4 & \textbf{3.42} & 3.38 & 3.35 & 4.17 & 8.24 & \textbf{10.09} \\
5 & 9.30 & 7.25 & \textbf{12.63} & 5.95 & 6.99 & \textbf{11.73} \\
6 & 8.87 & 4.85 & \textbf{11.11} & 8.73 & 5.37 & \textbf{15.32} \\
7 & \textbf{8.87} & 4.61 & 6.10 & 9.08 & 5.44 & \textbf{9.63} \\
\textit{Mean} & \textit{4.55} & \textit{3.05} & \textit{\textbf{4.90}} & \textit{3.23} & \textit{3.76} & \textit{\textbf{6.72}} \\
\bottomrule
\end{tabular}%
\end{footnotesize}
\caption{Performance comparison relative to ILP-fsws in dense instances across different $\alpha$ values. The results are given as percentages. The best approach per set is reported in bold.}
\label{tab:summary-performance-dense-combined}
\end{table}

\paragraph{Number of cuts generated}
Table~\ref{tab:nbr_lbbd_cuts} reports the average number of feasibility and optimality cuts added during the search by \LBBDG. Interestingly, two regularities emerge. First, feasibility cuts are consistently rare, with on average between 1.47 and 3.72 per instance. It means that the decomposition only seldom eliminates a master solution because its subproblem is infeasible. Second, optimality cuts are more numerous but decrease steeply with $\alpha$, from 72.51 at $\alpha = 0.25$ to 7.58 at $\alpha = 1.00$, mirroring the diminishing weight of the makespan in the objective and hence the diminishing amount of refinement the master needs. Therefore, few cuts are injected into the master, which is one reason for the strong performance of \LBBDG, especially when $\alpha$ is close to 1.
Finally, we defer the sparse sets to the next subsection: as detailed there, in the large majority of sparse instances the LBBD approach cannot certify the exactness of its solution, so they fall outside the optimality-guaranteed pool considered here.

\subsection{Comparison \LBBD\,--- \ILP\,--- \Freecp}

In this section, let us reintroduce the instances in which the \LBBD approach may not guarantee optimality of the solution.
Table~\ref{tab:percentage_guaranteed_optimal} shows how unevenly this affects the three densities: optimality is certified on all dense instances and on between 96.9\% and 99.7\% of the standard ones, but on only 15\% to 20\% of the sparse instances when $\alpha \leq 0.75$.
This behavior can be explained by the fact that in standard, and even more in dense ones, the constraints~\eqref{eq:cp-fixed-ee} significantly reduce the possibilities that the other tasks can take. Thus, in sparse instances, the \emph{subproblem} cannot close the optimal gap before the time limit. It is hardly due to the \emph{subproblem} struggling with its feasibility decidability. Indeed, the percentage of instances without an optimality guarantee plummets to zero when the \emph{subproblem} is a feasibility problem, i.e., when $\alpha = 1$.
Table~\ref{tab:non-opt-subproblems} displays the average percentage of \emph{subproblems} that do not prove optimality, per set, and $\alpha$. As expected, the larger the instance, the more of them fail in it. 
Because so few standard and dense instances lack a guarantee, the behavior of \LBBD on those densities coincides with the \LBBDG analysis of Section~\ref{section:analysis_lbbdg}, to which we refer the reader.

\begin{table}
\centering
\begin{minipage}[b]{0.48\textwidth}
\centering
\begin{small}
\begin{tabular}{|l|r|r|}
\toprule
 $\alpha$ & Avg optimality cuts & Avg feasibility cuts \\
\midrule
0.25 & 72.51 & 2.75 \\
0.50 & 63.44 & 3.72 \\
0.75 & 16.31 & 1.47 \\
1.00 & $7.58^\ddagger$ & $1.68^\ddagger$ \\
\bottomrule
\end{tabular}
\end{small}
\caption{Average number of \LBBDG optimality and feasibility cuts in the standard instances. $\ddagger$: for $\alpha=1$, no cut is added; instead, the optimality cuts represent the number of subproblems in which a feasible solution was found.}
\label{tab:nbr_lbbd_cuts}
\end{minipage}
\hfill
\begin{minipage}[b]{0.48\textwidth}
\centering
\begin{small}
\begin{tabular}{|c|rrr|}
\toprule
    \multirow{2}{*}{$\alpha$}  & \multicolumn{3}{c|}{\makecell{\% guaranteed optimal solutions}} \\
\cmidrule{2-4}
         & Sparse & Standard & Dense    \\
\midrule
    0.25 & 15.0   &  96.90 & 100.0  \\
    0.50 & 19.0   &  97.25 & 100.0  \\
    0.75 & 20.0   &  98.98 & 100.0  \\
    1.00 & 100.0  &  99.67 & 100.0  \\
\bottomrule
\end{tabular}
\end{small}
\caption{Percentage of solutions guaranteed optimal across instance densities under a 20s \emph{subproblem} time limit. A solution is guaranteed optimal if none of its \emph{subproblems} reached the time limit.}
\label{tab:percentage_guaranteed_optimal}
\end{minipage}
\end{table}

\begin{table}[ht]
  \centering
  \begin{small}
  \setlength{\tabcolsep}{4pt}
  \makebox[\textwidth][c]{
  \begin{tabular}{|c|ccccccccccccccccccc|}
    \toprule
    \multirow{2}{*}{$\alpha$} & \multicolumn{19}{c|}{Set} \\
    \cmidrule(lr){2-20}
         & 2    & 3    & 4    & 5    & 6    & 7    & 8    & 9    & 10   & 11   & 12   & 13   & 14   & 15   & 16   & 17   & 18   & 19   & 20   \\
    \midrule
    0.25 & 4.5  & 58.9 & 59.8 & 80.8 & 90.5 & 71.8 & 66.3 & 93.6 & 78.5 & 95.9 & 92.5 & 95.5 & 88.5 & 95.3 & 95.3 & 94.5 & 88.3 & 88.6 & 86.7 \\
    0.50 & - & 68.9 & 38.3 & 78.0 & 72.3 & 59.9 & 55.0 & 81.4 & 70.9 & 79.9 & 94.7 & 93.3 & 93.6 & 83.5 & 75.9 & 84.2 & 87.2 & 78.6 & 78.3 \\
    0.75 & 12.5 & 64.5 & 36.5 & 63.6 & 69.6 & 30.1 & 39.6 & 49.4 & 62.3 & 58.1 & 42.1 & 84.6 & 74.9 & 72.4 & 61.4 & 69.1 & 78.8 & 66.5 & 80.3 \\
    \bottomrule
  \end{tabular}}
  \end{small}
  \caption{Share of \LBBD subproblems which reached the 20s-subproblem time limit per sparse instance set and per $\alpha$ value, restricted to sparse instances for which at least one subproblem has timed out.}
  \label{tab:non-opt-subproblems}
\end{table}

\paragraph{Sparse instances} In sparse instances, \LBBD tends to be better than the other approaches in terms of objective value only when $\alpha = 1$ (12.46\% gain over \ILPF, versus 9.72\% for \Freecp), as shown in Table~\ref{tab:summary-performance-sparse-combined}. In the other cases, the \Freecp approach is more efficient, reaching gains of 21.13\%, 25.49\%, 12.73\% at $\alpha = 0.25, 0.50\ \text{and}\ 0.75$, respectively. \LBBD is also the method that converges the fastest to a solution it considers ``optimal'' (which may not be the true global optimum), as shown in Figure~\ref{fig:sparse_optimality_pct}.

\begin{figure}
    \centering
    \resizebox{0.75\linewidth}{!}{\begin{tikzpicture}[baseline=(current bounding box.south)]

\pgfplotsset{
  m freecp/.style    ={solid, color=green!60, fill opacity=0.45, line width=1pt, mark=none, no marks},
  m ilpcpws/.style   ={solid, color=orange!60, fill opacity=0.45, line width=1pt, mark=none, no marks},
  m ilpfsws/.style   ={solid, color=gray!60, fill opacity=0.45,  line width=1pt, mark=none, no marks},
  m lbbd/.style      ={solid, color=blue!60, fill opacity=0.45, line width=1pt, mark=none, no marks},
}

\newcommand{\Ntot}{200}

\begin{groupplot}[
  group style={group size=2 by 2, horizontal sep=1.4cm, vertical sep=1.6cm},
  width=0.45\textwidth, height=0.28\textwidth, scale only axis,
  xlabel={Time [s]}, ylabel={\% instances proven optimal},
  xmin=0, xmax=1800,
  xmode=log, log basis x=10,
  grid=major, no marks,
  legend cell align=left,
  legend style={
    at={(0.02,0.98)}, anchor=north west, font=\scriptsize,
    fill=white, fill opacity=0.9, text opacity=1, draw=black!30,
  },
  title style={font=\small},
]

\nextgroupplot[title={$\alpha = 0.25$}, ymin=-1, ymax=40, xlabel={}]
  \addplot[m freecp,  restrict x to domain=0:1800]
    table [x=t, y expr=\thisrow{freecp},            col sep=comma]
    {fig/tikz_data/cactus_sparse_alpha25.csv};  \addlegendentry{\Freecp}
  \addplot[m ilpcpws, restrict x to domain=0:1800]
    table [x=t, y expr=\thisrow{ilp-cpws},          col sep=comma]
    {fig/tikz_data/cactus_sparse_alpha25.csv};  \addlegendentry{\ILPC}
  \addplot[m ilpfsws, restrict x to domain=0:1800]
    table [x=t, y expr=\thisrow{ilp-fsws},          col sep=comma]
    {fig/tikz_data/cactus_sparse_alpha25.csv};  \addlegendentry{\ILPF}
  \addplot[m lbbd,    restrict x to domain=0:1800]
    table [x=t, y expr=\thisrow{lbbd-wcp-wuc-cpws}, col sep=comma]
    {fig/tikz_data/cactus_sparse_alpha25.csv};  \addlegendentry{\LBBD}

\nextgroupplot[title={$\alpha = 0.5$}, ymin=-1, ymax=45, xlabel={}, ylabel={}]
  \addplot[m freecp,  restrict x to domain=0:1800]
    table [x=t, y expr=\thisrow{freecp},            col sep=comma]
    {fig/tikz_data/cactus_sparse_alpha50.csv};
  \addplot[m ilpcpws, restrict x to domain=0:1800]
    table [x=t, y expr=\thisrow{ilp-cpws},          col sep=comma]
    {fig/tikz_data/cactus_sparse_alpha50.csv};
  \addplot[m ilpfsws, restrict x to domain=0:1800]
    table [x=t, y expr=\thisrow{ilp-fsws},          col sep=comma]
    {fig/tikz_data/cactus_sparse_alpha50.csv};
  \addplot[m lbbd,    restrict x to domain=0:1800]
    table [x=t, y expr=\thisrow{lbbd-wcp-wuc-cpws}, col sep=comma]
    {fig/tikz_data/cactus_sparse_alpha50.csv};

\nextgroupplot[title={$\alpha = 0.75$}, ymin=-1, ymax=75]
  \addplot[m freecp,  restrict x to domain=0:1800]
    table [x=t, y expr=\thisrow{freecp},            col sep=comma]
    {fig/tikz_data/cactus_sparse_alpha75.csv};
  \addplot[m ilpcpws, restrict x to domain=0:1800]
    table [x=t, y expr=\thisrow{ilp-cpws},          col sep=comma]
    {fig/tikz_data/cactus_sparse_alpha75.csv};
  \addplot[m ilpfsws, restrict x to domain=0:1800]
    table [x=t, y expr=\thisrow{ilp-fsws},          col sep=comma]
    {fig/tikz_data/cactus_sparse_alpha75.csv};
  \addplot[m lbbd,    restrict x to domain=0:1800]
    table [x=t, y expr=\thisrow{lbbd-wcp-wuc-cpws}, col sep=comma]
    {fig/tikz_data/cactus_sparse_alpha75.csv};

\nextgroupplot[title={$\alpha = 1.0$}, ymin=-1, ymax=100, ylabel={}]
  \addplot[m freecp,  restrict x to domain=0:1800]
    table [x=t, y expr=\thisrow{freecp},            col sep=comma]
    {fig/tikz_data/cactus_sparse_alpha100.csv};
  \addplot[m ilpcpws, restrict x to domain=0:1800]
    table [x=t, y expr=\thisrow{ilp-cpws},          col sep=comma]
    {fig/tikz_data/cactus_sparse_alpha100.csv};
  \addplot[m ilpfsws, restrict x to domain=0:1800]
    table [x=t, y expr=\thisrow{ilp-fsws},          col sep=comma]
    {fig/tikz_data/cactus_sparse_alpha100.csv};
  \addplot[m lbbd,    restrict x to domain=0:1800]
    table [x=t, y expr=\thisrow{lbbd-wcp-wuc-cpws}, col sep=comma]
    {fig/tikz_data/cactus_sparse_alpha100.csv};

\end{groupplot}
\end{tikzpicture}}
    \caption{Comparison of the percentage of solutions solved to optimality\textsuperscript{$\ddagger$} over time for each method in sparse sets. $\ddagger$: Note that, in the case of \LBBD, this implies that the solver certifies the optimality of the incumbent solution; however, due to the inefficacy of the cuts generated in the \emph{subproblem}, this solution may in fact not correspond to the true global optimum.}
    \label{fig:sparse_optimality_pct}
\end{figure}

\begin{table}[ht]
\centering
\begin{footnotesize}
\begin{tabular}{|r|rrr|rrr|}
\toprule
    & \Freecp & \ILPC & \LBBD & \Freecp & \ILPC & \LBBD \\
\cmidrule{2-7}
Set & \multicolumn{3}{c|}{$\alpha = 0.25$} & \multicolumn{3}{c|}{$\alpha = 0.5$} \\
\midrule
1 & \textbf{0.00} & \textbf{0.00} & \textbf{0.00} & \textbf{0.00} & \textbf{0.00} & \textbf{0.00} \\
2 & \textbf{0.32} & 0.12 & \textbf{0.32} & \textbf{0.00} & \textbf{0.00} & \textbf{0.00} \\
3 & 0.97 & 0.55 & \textbf{1.18} & \textbf{0.39} & \textbf{0.00} & 0.35 \\
4 & 4.29 & 3.85 & \textbf{4.55} & \textbf{0.83} & 0.57 & 0.81 \\
5 & \textbf{3.79} & 2.71 & 3.75 & 3.03 & 2.61 & \textbf{3.44} \\
6 & \textbf{8.46} & 5.60 & 8.33 & 5.36 & 3.74 & \textbf{6.03} \\
7 & 10.99 & 10.45 & \textbf{11.67} & 7.60 & 6.34 & \textbf{8.28} \\
8 & 18.94 & 19.27 & \textbf{20.85} & 10.78 & 16.03 & \textbf{17.22} \\
9 & 27.45 & 27.68 & \textbf{29.40} & 29.00 & \textbf{32.73} & 29.59 \\
10 & \textbf{22.87} & 21.11 & 22.64 & \textbf{25.90} & 25.38 & 15.73 \\
11 & \textbf{25.52} & 23.68 & 24.91 & 26.86 & 25.52 & \textbf{26.99} \\
12 & \textbf{23.07} & 12.50 & 20.18 & \textbf{24.08} & 23.85 & 22.62 \\
13 & \textbf{35.58} & 27.93 & 30.58 & \textbf{47.13} & 34.91 & 36.90 \\
14 & \textbf{39.45} & 31.97 & 33.92 & \textbf{56.63} & 41.71 & 44.84 \\
15 & 31.91 & 27.47 & \textbf{34.45} & \textbf{41.04} & 33.97 & 29.76 \\
16 & \textbf{33.15} & 19.83 & 27.61 & \textbf{45.66} & 22.26 & 32.82 \\
17 & \textbf{40.40} & 35.37 & 35.85 & \textbf{53.35} & 44.11 & 36.18 \\
18 & \textbf{29.08} & 15.84 & 17.76 & \textbf{45.48} & 29.48 & 26.45 \\
19 & \textbf{37.63} & 21.08 & 25.38 & \textbf{44.26} & 25.38 & 27.48 \\
20 & \textbf{28.79} & 21.05 & 27.94 & \textbf{42.36} & 32.99 & 41.18 \\
\textit{Mean} & \textit{\textbf{21.13}} & \textit{16.40} & \textit{19.06} & \textit{\textbf{25.49}} & \textit{20.08} & \textit{20.33} \\
\midrule
    & \multicolumn{3}{c|}{$\alpha = 0.75$} & \multicolumn{3}{c|}{$\alpha = 1.0$} \\
\midrule
1 & \textbf{0.00} & \textbf{0.00} & \textbf{0.00} & \textbf{0.00} & \textbf{0.00} & \textbf{0.00} \\
2 & \textbf{0.00} & \textbf{0.00} & \textbf{0.00} & \textbf{0.00} & \textbf{0.00} & \textbf{0.00} \\
3 & \textbf{0.00} & \textbf{0.00} & \textbf{0.00} & \textbf{0.00} & \textbf{0.00} & \textbf{0.00} \\
4 & \textbf{0.58} & 0.45 & 0.53 & \textbf{0.00} & \textbf{0.00} & \textbf{0.00} \\
5 & -0.26 & 0.03 & \textbf{0.69} & \textbf{0.00} & \textbf{0.00} & \textbf{0.00} \\
6 & \textbf{4.17} & 3.22 & 4.15 & \textbf{0.00} & \textbf{0.00} & \textbf{0.00} \\
7 & 0.13 & 0.64 & \textbf{1.14} & \textbf{0.00} & \textbf{0.00} & \textbf{0.00} \\
8 & -0.15 & \textbf{0.64} & 0.04 & \textbf{0.00} & \textbf{0.00} & \textbf{0.00} \\
9 & -29.03 & -30.91 & \textbf{-28.59} & -1.12 & \textbf{0.00} & \textbf{0.00} \\
10 & \textbf{21.88} & 19.41 & -1.68 & -2.83 & \textbf{0.00} & \textbf{0.00} \\
11 & 26.40 & 24.68 & \textbf{26.63} & -0.28 & \textbf{0.00} & \textbf{0.00} \\
12 & \textbf{15.81} & 12.66 & -12.15 & -0.01 & \textbf{0.00} & \textbf{0.00} \\
13 & 22.36 & \textbf{43.44} & 33.71 & 3.07 & 13.42 & \textbf{13.42} \\
14 & \textbf{29.10} & 24.54 & 28.25 & 13.68 & 0.55 & \textbf{19.22} \\
15 & 25.38 & 31.13 & \textbf{46.80} & 33.49 & 28.31 & \textbf{36.96} \\
16 & \textbf{28.07} & 15.64 & 28.07 & 9.83 & -4.80 & \textbf{15.71} \\
17 & \textbf{29.31} & 18.71 & 18.99 & 30.34 & 28.54 & \textbf{34.36} \\
18 & \textbf{15.80} & 13.74 & 14.62 & 34.10 & 33.34 & \textbf{44.05} \\
19 & 37.85 & 31.20 & \textbf{39.30} & 46.91 & 41.49 & \textbf{53.46} \\
20 & 27.16 & 31.14 & \textbf{35.01} & 27.16 & 27.75 & \textbf{32.11} \\
\textit{Mean} & \textit{\textbf{12.73}} & \textit{12.02} & \textit{11.78} & \textit{9.72} & \textit{8.43} & \textit{\textbf{12.46}} \\
\bottomrule
\end{tabular}%
\end{footnotesize}
\caption{Performance comparison relative to ILP-fsws in sparse instances across different $\alpha$ values. The results are given as percentages. The best approach per set is reported in bold.}
\label{tab:summary-performance-combined}
\label{tab:summary-performance-sparse-combined}
\end{table}

\paragraph{Statistical analysis} To put these comparisons on a firmer footing, Table~\ref{tab:wilcoxon} reports Holm-corrected $p$-values from one-sided paired Wilcoxon signed-rank tests \citep{wilcoxon1945individual, holm1979simple} comparing \LBBD to each other method on the per-instance objective. Against \ILPC and \ILPF, \LBBD is significantly better ($p<0.001$) in all density, establishing uniform dominance over the ILP baselines throughout the experimental grid. The comparison against \Freecp is more nuanced. \LBBD is statistically better ($p<0.001$) than \Freecp for all densities when $\alpha = 1$. In sparse and dense instances, the advantage of \LBBD tends to decrease when $\alpha$ lowers, so in those instances \LBBD and \Freecp are statistically comparable in objective value, and in sparse sets when $\alpha \leq 0.5$ \Freecp even achieves the higher mean gains (Table~\ref{tab:summary-performance-sparse-combined}). In standard instances, the Wilcoxon test reveals that in the case $\alpha = 0.5$ \LBBD is not superior to \Freecp, as Table~\ref{tab:summary-performance-combined-standard} also shows, but for the other values of $\alpha$, it appears to be the best approach.

Taken together, these results identify \LBBD as the strongest method whenever a certified optimum is attainable and, at the objective value, wherever the makespan leaves the objective: its advantage over \Freecp is concentrated at $\alpha = 1.0$ in all densities (and also at $\alpha \in \{0.25, 0.75\}$ in standard sets), while for smaller $\alpha$ in the sparse instances the monolithic CP remains competitive and is frequently the better choice.

\begin{table}[ht!]
\centering
\begin{small}
\begin{tabular}{|l|c||rrr|}
\toprule
Density & $\alpha$ & \texttt{freecp} & \texttt{ILP-cpws} & \texttt{ILP-fsws} \\
\midrule
   Sparse & 0.25 & \emph{1.000} & $<\!0.001$ & $<\!0.001$ \\
    & 0.50 & \emph{0.896} & $<\!0.001$ & $<\!0.001$ \\
    & 0.75 & \emph{0.341} & $<\!0.001$ & $<\!0.001$ \\
    & 1.00 & $<\!0.001$ & $<\!0.001$ & $<\!0.001$ \\
\cmidrule(lr){1-5}
   Standard & 0.25 & 0.005 & $<\!0.001$ & $<\!0.001$ \\
    & 0.50 & \emph{0.998} & $<\!0.001$ & $<\!0.001$ \\
    & 0.75 & 0.067 & $<\!0.001$ & $<\!0.001$ \\
    & 1.00 & $<\!0.001$ & $<\!0.001$ & $<\!0.001$ \\
\cmidrule(lr){1-5}
   Dense & 0.25 & \emph{0.770} & $<\!0.001$ & $<\!0.001$ \\
    & 0.50 & \emph{0.130} & $<\!0.001$ & $<\!0.001$ \\
    & 0.75 & \emph{0.131} & $<\!0.001$ & $<\!0.001$ \\
    & 1.00 & $<\!0.001$ & $<\!0.001$ & $<\!0.001$ \\
\bottomrule
\end{tabular}
\end{small}
\caption{Pairwise comparison of \LBBD against each of the three competitors on the per-instance objective. Each cell reports the $p$-value of a one-sided paired Wilcoxon signed-rank test~\citep{wilcoxon1945individual}, the pairing being by instance: instances on which the two methods returned the same objective are dropped, the remaining differences are ranked by magnitude, and the alternative is $H_1$: \LBBD produces a strictly smaller (better) objective than the competitor. Within each density$\,\times\,\alpha$ row, the three resulting $p$-values are then Holm-corrected \citep{holm1979simple} to control the family-wise error rate across the three competitors. A value $p = 0.1$ therefore indicates that, after multiplicity correction, \LBBD is significantly better than the corresponding competitor at the 10\,\% level on the underlying instance pool. The values for which no significant statistical correlation is found are reported in italic.}
\label{tab:wilcoxon}
\end{table}

\subsection{Discussion on results}

In summary, the LBBD approach is the best approach by far, in terms of solution quality and computation time, when $\alpha = 1$, i.e., when \textit{subproblem} collapses to a feasibility oracle. In this case, no matter the density of the instance, the CP modeling the subproblem always decides the feasibility before the time limit, meaning the optimal solution yielded is guaranteed to be the global optimum.
Experiments show that LBBD approach can solve to optimality instances up to 480 tasks whose around 86 of them are energy-intensive, and a time horizon of up to 6207.
It remains a relevant approach in dense and standard instances when $\alpha < 1$, in which it nearly often guarantees the validity of its optimal solution, and beats the other approaches in terms of computation speed and solution quality, with an exception for $\alpha = 0.5$ in standard instances and $\alpha = 0.25$ in dense instances, in which \Freecp matches the LBBD.
Finally, in sparse instances, $\Freecp$ seems to be the best approaches when $\alpha < 1$, because most of the LBBD subproblems timeout, hampering significantly the exploration of search space, potentially pruning optimal solutions while considerably slowing down the master search tree exploration.

\section{Generalization to different resource environments}
\label{section:generalization-of-LBBD}

In this section, we examine the applicability of the proposed decomposition scheme to additional energy-aware optimization problems that arise in heterogeneous machine environments. 
Specifically, we demonstrate how our method can be adapted to two further scheduling problems: the Resource-Constrained Project Scheduling Problem (RCPSP) with blocking times under a total weighted tardiness objective, and the Flexible Job Shop Scheduling Problem. 
In both cases, scheduling problems are integrated with a time-of-use electricity tariff structure and machine state modeling. We then apply the proposed decomposition framework to each of these combined problem settings.
In the following, we detail only the subproblems, since the master problem is technically unchanged. Hence, the notations defined in Section~\ref{section:problem_statement} still hold.
The time limit for the subproblem is still set at $20$ seconds.

\subsection{Subproblem: RCPSP with blocking times \& total weighted tardiness}
\label{subsection:RCPSP-with-blocking-time}

In this section, we investigate another variant of the RCPSP, $PSm|intree|\sum_jw_jT_j$, detailed in \cite{nedbalek2025bottleneck}, which serves as the subproblem. In it, each task $j\in\mathcal{T}$ has a processing duration $p_j$ and a due date $d_j$. We define the tardiness weight $w_j$ as the penalty for each task $j$ that is late in a given period of time. The precedence constraints between tasks are represented by a graph $G=(V,A)$. Task preemption is not permitted.
Tasks are assigned to resources $\mathcal{R} = \{R_0,\dots,R_m\}$ with time-variant renewable capacities, i.e., at any time $t \in \mathcal{I}$, the capacity of a resource $k$ is $\rho_k^{(t)}$. We assume that $\rho_k^{(t)} \in \{0, r_k\}$ with $r_k > 0$, except for the energy-aware resource, which has a constant capacity of $1$ at any time. At any time $t \in \mathcal{I}$, a task $j$ would consume a quantity $r_{jk}^{(t)}$ of resource $k$ if processed. The subproblem can be modeled by the following constraint programming formulation:

\begin{small}
\begin{align}
    (Subproblem):\ \min&\quad \sum_{j \in \mathcal{T}} w_j \max \left\{\textsc{EndOf}(\xi_j) - d_j, 0\right\} &\\
    s.t. &\quad \sum_{j \in \mathcal{T}} \textsc{Pulse}\left(\xi_j, r_{jk}^{(t)}\right) \leq \rho_k^{(t)} &\forall t \in \mathcal{I}, k \in \mathcal{R} \\
    &\quad \textsc{EndBeforeStart}(\xi_i,\xi_j) &\forall (i, j) \in \mathcal{T}^2,\ i \prec j \\
    &\quad \textsc{StartOf}(\xi_j) = t &\forall j \in \mathcal{J}, t\in\mathcal{I},\hat{x}_{jt} = 1 \\
    &\quad \textsc{Interval}\ \xi_j,\ \textsc{LengthOf}(\xi_j) = p_j &\forall j \in \mathcal{T}
\end{align}
\end{small}

We compare our cold-started LBBD approach with a monolithic CP formulation derived from the CP proposed in Section~\ref{sec:constraint_programming} on a small set of instances derived from the dataset\footnote{The dataset can be retrieved from \cite{DVN/XI9JE0_2025}} proposed by \cite{nedbalek2025bottleneck}. For each instance, we randomly select 20\% of the tasks to be energy-intensive.
Table~\ref{tab:performance-blocking-twt} presents, for different values of $\alpha$, the minimal, maximal, and average gaps~\eqref{eq:delta} of the LBBD approach relative to the CP over the data set and the associated standard deviations.
It reveals that our approach exhibits superior performance in terms of identifying optimal solutions compared to the CP for all instances, with an average improvement of 7.43\%, with a maximum reaching nearly 50\%. It is also worth noting that the LBBD never performs worse than the CP.
Also, no subproblems reached the time limit, there were all solved to optimality.
Furthermore, Figure~\ref{fig:time-blocking-twt} depicts how quickly our LBBD approach proves the optimality of solutions, with nearly 90\% of instances solved to optimality in less than half an hour, whereas the CP could only prove the optimality of less than 5\% of them in the same time.

\begin{figure}[]
\centering
\begin{minipage}[b]{0.45\textwidth}
    \centering
    \begin{tabular}{|lrrrr|}
    \toprule
    & \multicolumn{4}{c|}{Relative to CP [\%]}\\
    \cmidrule(lr){2-5}
    $\alpha$ & mean & min & max & std \\
    \midrule
    0.25 & 8.09 & 0.00 & 36.00 & 10.67 \\
    0.50 & 6.64 & 0.00 & 35.93 & 9.42 \\
    0.75 & 7.80 & 0.00 & 48.70 & 11.61 \\
    1.00 & 7.17 & 0.00 & 30.81 & 9.02 \\
    \bottomrule
    \end{tabular}
    \captionof{table}{RCPSP with blocking times \& total weighted tardiness - Performance comparison for various values of $\alpha$ of the LBBD approach relative to the monolithic CP presented in Section~\ref{subsection:RCPSP-with-blocking-time}. A time limit of 1800 seconds is imposed for both methods. std = standard deviation.}
    \label{tab:performance-blocking-twt}
\end{minipage}
\hfill
\begin{minipage}[b]{0.45\textwidth}
    \centering
    \begin{tikzpicture}
\begin{axis}[
  width=1\textwidth, height=0.7\textwidth,
  xlabel={time [s]},
  ylabel={solved to optimality [\%]},
  ymin=0,
  grid=major,
  legend pos=north west,
  legend cell align=left,
  cycle list name=color list,
  xmode=log,
  log basis x=10           
]

\pgfplotsset{
  mybox A/.style={solid,  draw=blue!60,  fill opacity=0.45},
  mybox B/.style={solid,   draw=red!60,   fill opacity=0.45},
  mybox C/.style={solid,  draw=green!60, fill opacity=0.45},
  mybox D/.style={solid,draw=orange!60,fill opacity=0.45},
  mybox E/.style={solid,draw=gray!60,fill opacity=0.45},
}

  \addplot+[mybox A, very thick] table [x=time, y={benders}, col sep=comma] {fig/tikz_data/cactus_percent_combined_problem1.csv};
  \addlegendentry{LBBD}

  \addplot+[mybox D, very thick] table [x=time, y=freecp, col sep=comma] {fig/tikz_data/cactus_percent_combined_problem1.csv};
  \addlegendentry{CP}

\end{axis}
\end{tikzpicture}
    \caption{RCPSP with blocking times \& total weighted tardiness - Proportion of instances solved to optimality by each method over time.  The results for each value of $\alpha\in\{0.25,0.5,0.75,1.0\}$ are concatenated.}
    \label{fig:time-blocking-twt}
\end{minipage}
\end{figure}

\subsection{Subproblem: Flexible job shop}
\label{subsection:fjsp}

In this subsection, we consider a classical flexible job shop problem (FJSP) $FJm||C_{\max}$ \citep{dauzere2024flexible}. Let $\mathcal{M}$ be a set of $m$ machines. Each machine cannot handle more than one job at a time. We have a set of $n$ jobs $\mathfrak{J}$. Each job $i \in \mathfrak{J}$ comprises a chain $\mathcal{O}_i$ of operations $o_{i\lambda}$, so the set of all operations is $\mathcal{T} = \bigcup_{i \in \mathfrak{J}} \mathcal{O}_i$. Each operation $i$ can be processed by any machine in a subset of machines $\mathcal{M}_i \subseteq \mathcal{M}$.
We use the MILP detailed in \cite{dauzere2024flexible}. Let us consider binary variables $\alpha_{i}^k = 1$ when operation $i$ starts on machine $k$ and $\beta_{ij} = 1$ when operation $i$ is sequenced before $j$, and continuous variables $C_{\max}$ and $t_i$, the MILP subproblem is:

\begin{small}
\begin{align}
    (Subproblem):\ \min&\ C_{\max} & \\
    s.t.&\ \sum_{k \in \mathcal{M}_i} \alpha_i^k = 1 & \forall i \in \mathcal{T} \\
        &\ t_i \geq t_{pr(i)} + \sum_{k \in \mathcal{M}_{pr(i)}}p^m_{pr(i)}\alpha^k_{pr(i)} &\forall i \in \mathcal{T} \\
        &\ t_i \geq t_{j} + p_{j}^k - M \left(2 - \alpha_i^k - \alpha_{j}^k + \beta_{ij}\right) & \forall (i,j)\in\mathcal{T}^2,\ \text{if}\ i \not= j, k \in \mathcal{M}_i \cap \mathcal{M}_{j} \\
        &\ t_{j} \geq t_i + p_i^k - M \left(3 - \alpha^k_i - \alpha^k_{j} - \beta_{ij}\right) & \forall (i,j)\in\mathcal{T}^2,\ \text{if}\ i \not= j, k \in \mathcal{M}_i \cap \mathcal{M}_{j} \\
        &\ C_{\max} \geq t_i + \sum_{k \in \mathcal{M}_i}p_i^k \alpha_i^k & \forall i \in \mathcal{T} \\
        &\ t_j = t & \forall j \in \mathcal{J}, t \in \mathcal{I}, \hat{x}_{jt} = 1
\end{align}
\end{small}
The objective function of the master is normalized as in \eqref{eq:full_obj}.
We compare LBBD with a monolithic MILP (see Appendix G of the supplementary materials) based on the time-indexed formulation of FJSP presented by \cite{ku2016mixed}. It is interesting to notice that the SPACES graph forces the monolithic ILP to follow a time-indexed formulation, while the subproblem of our LBBD can employ models that do not have to stick to this kind of formulation; in particular, the one we use in the subproblem does not. No warmstart solution is supplied to models.
We employ the IIS (Irreducible Inconsistent Subsystem)~\citep{gleeson1990Identifying} provided by the solver to determine $\mathcal{I}nf$.
The tests are performed on a dataset\footnote{The dataset can be retrieved from the \cite{DVN/AXMPOS_2025}} created from the \cite{brandimarte1993routing} and \cite{fattahi2010dynamic} datasets, in which around 20\% of the operations are labeled energy-intensive.
Figure~\ref{fig:problem2_feasible_sols} shows the number of instances in which LBBD and ILP have found an optimal solution for each dataset and each value of $\alpha$. We observe that ILP is better to prove optimality when $\alpha$ is low, while LBBD is always better when $\alpha = 1$. Although the methods' performances remain close, ILP seems slightly better at proving the optimality of a solution in the other cases, especially in the Brandimarte set.
However, a general trend tends to emerge when the size of the instances grows, in which the LBBD digs out a feasible solution more often than the monolithic MILP. Indeed, as shown in Figure~\ref{fig:problem2_feasible_sols}, in the 120 overall experiments run, with a time limit of 1800 seconds, the LBBD returned a feasible solution 107 times, while the monolithic ILP was only able to find one 83 times, i.e., around 20\% less.

\begin{figure}[t]
    \centering
\tikzset{
  mybox A/.style={solid, draw=blue!60,  fill=blue!60,  fill opacity=0.8}, 
  mybox Afeas/.style={solid, draw=blue!60,  fill=blue!30,  fill opacity=0.6}, 
  mybox E/.style={solid, draw=gray!60, fill=gray!60, fill opacity=0.8}, 
  mybox Efeas/.style={solid, draw=gray!60, fill=gray!30, fill opacity=0.6}, 
}

\tikzset{
  lbbd-opt/.style   ={mybox A},
  lbbd-feas/.style  ={mybox Afeas},
  lbbd-infeas/.style={mybox Afeas,
                      fill=blue!20, fill opacity=0.5,
                      pattern={Lines[angle=45,distance=3pt,line width=0.35pt]},
                      pattern color=blue!70},
  ilp-opt/.style    ={mybox E},
  ilp-feas/.style   ={mybox Efeas},
  ilp-infeas/.style ={mybox Efeas,
                      fill=gray!20, fill opacity=0.5,
                      pattern={Lines[angle=45,distance=3pt,line width=0.35pt]},
                      pattern color=gray!70},
}

\makeatletter
\newcommand\resetstackedplots{%
  \pgfplots@stacked@isfirstplottrue
  \addplot [forget plot,draw=none] coordinates{(1,0) (2,0) (3,0) (4,0) (5,0) (6,0) (7,0) (8,0)};
}
\makeatother

\begin{tikzpicture}
\begin{axis}[
    width=0.9\textwidth, height=0.25\textwidth,
    ybar stacked, ymin=0,
    enlarge x limits=0.2,
    xtick={1,2,3,4,5,6,7,8},
    xticklabels={
      {Brandimarte $\alpha=0.25$},
      {Brandimarte $\alpha=0.5$},
      {Brandimarte $\alpha=0.75$},
      {Brandimarte $\alpha=1.0$},
      {Fattahi $\alpha=0.25$},
      {Fattahi $\alpha=0.5$},
      {Fattahi $\alpha=0.75$},
      {Fattahi $\alpha=1.0$}
    },
    xticklabel style={rotate=25,anchor=east,yshift=-7pt,xshift=4pt,font=\footnotesize},
    ylabel={instances [-]},
    grid=both, grid style={dotted},
    legend style={at={(0.5,1.05)},anchor=south,legend columns=3,draw=none,font=\footnotesize},
    legend cell align=left,
    legend entries={
      LBBD proven optimal,
      LBBD proven feasible,
      LBBD no solution,
      ILP proven optimal,
      ILP proven feasible,
      ILP no solution
    },
]


\addplot +[bar shift=-.22cm, lbbd-opt] coordinates {
  (1,1) (2,1) (3,1) (4,8) (5,12) (6,13) (7,14) (8,16)
};

\addplot +[bar shift=-.22cm, lbbd-feas] coordinates {
  (1,7) (2,7) (3,8) (4,2) (5,7) (6,5) (7,5) (8,0)
};

\addplot +[bar shift=-.22cm, lbbd-infeas] coordinates {
  (1,2) (2,2) (3,1) (4,0) (5,1) (6,2) (7,1) (8,4)
};

\resetstackedplots


\addplot +[bar shift=.22cm, ilp-opt] coordinates {
  (1,2) (2,3) (3,3) (4,7) (5,13) (6,13) (7,12) (8,14)
};

\addplot +[bar shift=.22cm, ilp-feas] coordinates {
  (1,4) (2,3) (3,3) (4,0) (5,1) (6,2) (7,3) (8,0)
};

\addplot +[bar shift=.22cm, ilp-infeas] coordinates {
  (1,4) (2,4) (3,4) (4,3) (5,6) (6,5) (7,5) (8,6)
};

\end{axis}
\end{tikzpicture}
    \caption{Flexible Job Shop - Number of solutions with proven optimality and feasible solution found grouped by benchmark set and $\alpha$ value.}
    \label{fig:problem2_feasible_sols}
\end{figure}

\subsection{Additional remarks}

In the problems mentioned in Sections~\ref{subsection:RCPSP-with-blocking-time} and \ref{subsection:fjsp}, our LBBD approach preserves the optimality of the final solution, as long as the subproblem is solved to optimality each time it is triggered.
However, if this optimality criterion is not mandatory, approximation methods can be used to compute the subproblem. In particular, a wide variety of heuristics or metaheuristic approaches can be envisioned. 
They seem to be particularly suited for cases where the subproblem would be very complex but only checks the feasibility of the master solution.

\section{Conclusion}
\label{section:conclusion}

In this work, we have investigated the classical resource-constrained project scheduling problem extended with time-of-use energy tariffs and machine states. We proposed three models for this problem: two monolithic models employing \emph{ILP} and \emph{CP} approaches, and \emph{\texttt{LBBD}}. In the latter, the problem was divided into two problems: a master problem, solved by an ILP, solves the energy part of the original problem, while the RCPSP components are put aside in a subproblem. This subproblem is solved efficiently via a CP model. We assessed these models in a series of energy-intensive density instances. Outcomes revealed that in standard and dense instances, \texttt{LBBD} outperforms the other proposed approaches, with increasing margins when the instance size grows, solving to optimality instances with up to 480 tasks. In sparse instances, it is the monolithic CP that tends to yield the best results, as long as the makespan appears in the objective function. Otherwise, \texttt{LBBD} statistically retains its dominance in most cases.
Finally, we have demonstrated that \texttt{LBBD} is applicable to a broader class of problems, provided that they admit a decomposition into an energy management problem in which the start times of energy-intensive activities can be determined within the subproblem. We investigated two distinct subproblem formulations: a resource-constrained project scheduling problem (RCPSP) with blocking times and a total weighted tardiness objective, and a classical flexible job shop scheduling problem. In both cases, \texttt{LBBD} outperforms a traditional monolithic solution approach.

In further work, we can envision having multiple machines that require energy-intensive resources. These machines may operate on the same energy source or require different energy-constrained resources, each with its own TOU. 
Furthermore, we can also consider more than three states, but with multiple idle/processing times, as some tasks require a precise state to be performed. We would also suggest looking into the impact of temporary energy storage capacities, such as batteries, especially when combined with on-site energy production capacities (i.e., independent of time-of-use), such as solar panels or wind turbines.
Finally, another avenue of research would be to explore the combined effect of uncertainty of energy tariffs~\citep {XIECHEN2026107567} and state switching over longer time horizons.


\section*{Acknowledgments}
This work was co-funded by the European Union under the project ROBOPROX (reg. no. CZ.02.01. 01/00/22\_008/0004590) and by the Grant Agency of the Czech Republic under the Project GACR 25-17904S.



\bibliographystyle{apa}
\bibliography{refs}

\appendix


\section{Acronyms table}
\label{section:acronyms}

\begin{table}[h!]
    \centering
    \begin{small}
    \begin{tabular}{cc}
    \toprule
        RCPSP & Resource-Constrained Project Scheduling Problem \\
        (M)ILP & (Mixed) Integer Linear Programming \\
        CP & Constraint Programming \\
        LBBD & Logic-Based Benders Decomposition \\
        TOU & Time-of-use tariffs \\
        TEC & Total Energy Cost \\
        LB/UB & Lower bounds / Upper bounds \\
    \bottomrule
    \end{tabular}
    \end{small}
    \caption{Acronyms used throughout the paper.}
    \label{tab:summary_acronyms}
\end{table}

\clearpage
\newpage
\section{SPACES graph}
\label{section:SPACES}

The SPACES graph, depicted in Figure~\ref{fig:example-states-graph}, is a time–state graph whose nodes are (state, interval) pairs of the energy-intensive machine over its state set $\Sigma$, and whose arcs weight each admissible state change by its transition time $T$ and transition energy $P$ at the prevailing time-of-use tariffs. The minimum-cost machine-state trajectory between any two proc intervals — and between the initial/final off state and any proc state — is obtained as a shortest path in this graph. SPACES therefore decouples the scheduling decisions from the optimization of the state trajectory: once the task start times are fixed, the optimal states filling the surrounding gaps follow directly from precomputed shortest paths, which sharply reduces the model size.

\begin{figure}[h!]
    \centering
    \input{fig/clipped_spaces}
    \caption{SPACES graph for the transition graph in Figure~2 and energy cost profile from Figure~4.}
    \label{fig:example-states-graph}
\end{figure}

\clearpage
\newpage
\section{Detailed characteristics of test instance sets}
\label{section:detailed_characteristics}

The following tables detail the characteristics of the instance sets employed in Section 6 to assess the different methods. Table~\ref{tab:summary_stats} presents the characteristics of the standard instances set by set, while Table~\ref{tab:summary_stats_dense_sparse} give summary statistics of dense and sparse instance sets.
Figure~\ref{fig:precedence_example} draws a typical precedence structure obtained by merging two RCPSP instances from PSPLIB\footnote{Available at: \href{https://www.om-db.wi.tum.de/psplib/}{https://www.om-db.wi.tum.de/psplib/}} together.

\begin{table}[ht]
\centering
\begin{small}
\begin{tabular}{|l||r|rrrr|r|rrrr|}
\toprule
 & \multicolumn{1}{c|}{$|\mathcal{T}|$} & \multicolumn{4}{c|}{$|\mathcal{J}|$} & $|\mathcal{J}|/|\mathcal{T}|$ & \multicolumn{4}{c|}{Horizon $h$} \\
Set & & mean & min & max & std & mean & mean & min & max & std \\
\midrule
1 & 32 & 6.55 & 3 & 10 & 1.67 & 20.47 & 836.30 & 486 & 1117 & 150.19 \\
2 & 64 & 11.60 & 8 & 15 & 1.64 & 18.12 & 963.95 & 740 & 1288 & 153.68 \\
3 & 96 & 17.70 & 12 & 22 & 2.58 & 18.44 & 1297.70 & 928 & 1543 & 166.88 \\
4 & 128 & 23.05 & 17 & 31 & 3.05 & 18.01 & 1672.35 & 1272 & 2055 & 225.03 \\
5 & 160 & 29.70 & 23 & 35 & 3.66 & 18.56 & 1953.05 & 1601 & 2291 & 187.93 \\
6 & 192 & 35.10 & 26 & 53 & 5.42 & 18.28 & 2306.45 & 1624 & 3228 & 397.50 \\
7 & 224 & 40.65 & 35 & 47 & 3.54 & 18.15 & 2723.45 & 1999 & 3539 & 349.78 \\
8 & 256 & 44.80 & 38 & 50 & 3.24 & 17.50 & 3039.45 & 2110 & 3757 & 411.91 \\
9 & 288 & 51.40 & 42 & 60 & 5.48 & 17.85 & 3423.10 & 2670 & 3997 & 401.12 \\
10 & 320 & 56.35 & 46 & 68 & 6.43 & 17.61 & 3816.95 & 2924 & 4517 & 509.27 \\
11 & 352 & 59.25 & 52 & 69 & 3.70 & 16.83 & 4056.00 & 3445 & 4607 & 329.44 \\
12 & 384 & 66.60 & 58 & 79 & 6.17 & 17.34 & 4406.70 & 3469 & 5710 & 531.31 \\
13 & 416 & 71.95 & 61 & 85 & 5.95 & 17.30 & 4763.65 & 4180 & 6307 & 493.58 \\
14 & 448 & 79.25 & 73 & 87 & 4.56 & 17.69 & 5202.85 & 4488 & 6066 & 389.12 \\
15 & 480 & 85.75 & 74 & 97 & 7.37 & 17.86 & 5527.15 & 4513 & 6207 & 519.27 \\
\bottomrule
\end{tabular}
\end{small}
\caption{Characteristics of standard instance sets. For each set, we present the number of tasks ($|\mathcal{T}|$), the number of energy-intensive tasks ($|\mathcal{J}|$), the ratio of energy-intensive tasks over the overall number of tasks (in \%), and the time horizon. All instances of a set share the same number of tasks (std = standard deviation).}
\label{tab:summary_stats}
\end{table}

\begin{table}[t]
\centering
\label{tab:large-instances}
\begin{subtable}[t]{0.48\linewidth}
\centering
\begin{small}
\begin{tabular}{|l||rrrr|}
\toprule
 & $|\mathcal{T}|$ & $|\mathcal{J}|$ & $|\mathcal{J}|/|\mathcal{T}|$ & Horizon $h$ \\
\midrule
mean & 128.00 & 65.00 & 51.16 & 4315.06 \\
std  &  69.12 & 34.56 &  0.95 & 2313.90 \\
min  &  32.00 & 17.00 & 50.45 & 1098.25 \\
25\% &  80.00 & 41.00 & 50.57 & 2705.88 \\
50\% & 128.00 & 65.00 & 50.78 & 4277.05 \\
75\% & 176.00 & 89.00 & 51.30 & 5978.52 \\
max  & 224.00 &113.00 & 53.12 & 7461.30 \\
\bottomrule
\end{tabular}
\end{small}
\caption{Summary statistics of dense instances.}
\label{tab:densely-large-instances}
\end{subtable}
\hfill
\begin{subtable}[t]{0.48\linewidth}
\centering
\begin{small}
\begin{tabular}{|l||rrrr|}
\toprule
 & $|\mathcal{T}|$ & $|\mathcal{J}|$ & $|\mathcal{J}|/|\mathcal{T}|$ & Horizon $h$ \\
\midrule
mean & 336.00 & 11.50 & 3.69 & 2407.87 \\
std  & 189.31 &  5.92 & 0.70 & 1114.68 \\
min  &  32.00 &  2.00 & 3.28 &  764.20 \\
25\% & 184.00 &  6.75 & 3.33 & 1549.23 \\
50\% & 336.00 & 11.50 & 3.42 & 2349.10 \\
75\% & 488.00 & 16.25 & 3.67 & 3360.75 \\
max  & 640.00 & 21.00 & 6.25 & 4065.30 \\
\bottomrule
\end{tabular}
\end{small}
\caption{Summary statistics of sparse instances.}
\label{tab:sparsely-large-instances}
\end{subtable}
\caption{Summary statistics of sparse and dense instances. $|\mathcal{J}|$ and $|\mathcal{T}|$ denote the number of energy-intensive tasks and overall tasks, respectively. The ratio $|\mathcal{J}|/|\mathcal{T}|$ is reported as a percentage. Horizon length and $|\mathcal{J}|$ are averaged across instances within each set.}
\label{tab:summary_stats_dense_sparse}
\end{table}

\begin{figure}[ht]
    \centering
\begin{tikzpicture}[
    scale=1.0,
    every node/.style={inner sep=0pt, outer sep=0pt},
    job/.style={circle, draw=black!55, line width=0.2pt, minimum size=2mm},
    normal/.style={job, fill=blue!28},
    ee/.style={job, fill=orange!70},
    dummy/.style={job, fill=black!18},
    boundary/.style={circle, draw=black!70, line width=0.4pt,
                     minimum size=2.6mm, fill=black!15},
    arc/.style={-{Latex[length=1.1mm,width=0.8mm]},
                draw=black!50, line width=0.2pt, opacity=0.65},
  ]
  \node[dummy] (n0) at (0.700,-0.525) {};
  \node[normal] (n1) at (1.400,-0.875) {};
  \node[normal] (n2) at (1.400,-1.225) {};
  \node[normal] (n3) at (1.400,-0.525) {};
  \node[normal] (n4) at (2.100,-2.275) {};
  \node[normal] (n5) at (2.100,-1.225) {};
  \node[normal] (n6) at (2.100,-0.875) {};
  \node[normal] (n7) at (2.100,-1.925) {};
  \node[normal] (n8) at (2.100,-0.525) {};
  \node[normal] (n9) at (2.800,-2.100) {};
  \node[normal] (n10) at (2.800,-1.050) {};
  \node[ee] (n11) at (3.500,-1.750) {};
  \node[ee] (n12) at (3.500,-1.400) {};
  \node[normal] (n13) at (2.100,-1.575) {};
  \node[ee] (n14) at (4.200,-1.925) {};
  \node[ee] (n15) at (2.800,-1.750) {};
  \node[ee] (n16) at (4.200,-1.575) {};
  \node[normal] (n17) at (4.200,-1.225) {};
  \node[normal] (n18) at (4.900,-1.225) {};
  \node[normal] (n19) at (2.100,-0.175) {};
  \node[normal] (n20) at (4.900,-0.875) {};
  \node[normal] (n21) at (2.800,-1.400) {};
  \node[normal] (n22) at (4.200,-0.875) {};
  \node[normal] (n23) at (5.600,-0.700) {};
  \node[normal] (n24) at (6.300,-0.525) {};
  \node[normal] (n25) at (3.500,-1.050) {};
  \node[ee] (n26) at (3.500,-0.700) {};
  \node[normal] (n27) at (4.200,-0.525) {};
  \node[normal] (n28) at (4.200,-0.175) {};
  \node[normal] (n29) at (4.900,-0.525) {};
  \node[normal] (n30) at (7.000,-0.525) {};
  \node[dummy] (n31) at (7.700,0.000) {};
  \node[dummy] (n32) at (0.700,0.525) {};
  \node[normal] (n33) at (1.400,1.225) {};
  \node[normal] (n34) at (1.400,0.875) {};
  \node[normal] (n35) at (1.400,0.525) {};
  \node[normal] (n36) at (2.100,2.275) {};
  \node[normal] (n37) at (2.800,0.350) {};
  \node[normal] (n38) at (2.800,2.100) {};
  \node[normal] (n39) at (2.100,1.225) {};
  \node[normal] (n40) at (2.100,0.875) {};
  \node[normal] (n41) at (2.100,1.575) {};
  \node[normal] (n42) at (2.100,1.925) {};
  \node[normal] (n43) at (2.800,1.400) {};
  \node[normal] (n44) at (3.500,1.400) {};
  \node[normal] (n45) at (3.500,1.750) {};
  \node[normal] (n46) at (2.800,0.000) {};
  \node[normal] (n47) at (2.800,1.750) {};
  \node[normal] (n48) at (4.200,1.925) {};
  \node[ee] (n49) at (4.200,0.875) {};
  \node[normal] (n50) at (2.800,0.700) {};
  \node[ee] (n51) at (3.500,1.050) {};
  \node[ee] (n52) at (4.200,1.225) {};
  \node[ee] (n53) at (3.500,0.700) {};
  \node[normal] (n54) at (4.900,1.225) {};
  \node[normal] (n55) at (2.800,1.050) {};
  \node[normal] (n56) at (3.500,0.350) {};
  \node[normal] (n57) at (4.900,0.525) {};
  \node[normal] (n58) at (5.600,0.350) {};
  \node[normal] (n59) at (4.200,1.575) {};
  \node[normal] (n60) at (5.600,0.700) {};
  \node[normal] (n61) at (6.300,0.525) {};
  \node[normal] (n62) at (4.900,0.875) {};
  \node[dummy] (n63) at (7.000,0.525) {};
  \node[boundary, label={[font=\scriptsize, label distance=1.5mm]below:$0$}] (S) at (0.000,0.000) {};
  \node[boundary, label={[font=\scriptsize, label distance=1.5mm]below:$n+1$}] (T) at (8.400,0.000) {};
  \begin{scope}[on background layer]
    \draw[arc] (n0) -- (n1);
    \draw[arc] (n0) -- (n2);
    \draw[arc] (n0) -- (n3);
    \draw[arc] (n1) -- (n5);
    \draw[arc] (n1) -- (n8);
    \draw[arc] (n1) -- (n13);
    \draw[arc] (n2) -- (n6);
    \draw[arc] (n2) -- (n7);
    \draw[arc] (n2) -- (n9);
    \draw[arc] (n3) -- (n4);
    \draw[arc] (n3) -- (n19);
    \draw[arc] (n3) -- (n22);
    \draw[arc] (n4) -- (n9);
    \draw[arc] (n4) -- (n17);
    \draw[arc] (n5) -- (n15);
    \draw[arc] (n5) -- (n16);
    \draw[arc] (n5) -- (n23);
    \draw[arc] (n6) -- (n10);
    \draw[arc] (n7) -- (n12);
    \draw[arc] (n7) -- (n14);
    \draw[arc] (n8) -- (n25);
    \draw[arc] (n8) -- (n29);
    \draw[arc] (n9) -- (n11);
    \draw[arc] (n9) -- (n26);
    \draw[arc] (n10) -- (n12);
    \draw[arc] (n10) -- (n27);
    \draw[arc] (n10) -- (n28);
    \draw[arc] (n11) -- (n14);
    \draw[arc] (n11) -- (n16);
    \draw[arc] (n11) -- (n27);
    \draw[arc] (n12) -- (n17);
    \draw[arc] (n12) -- (n22);
    \draw[arc] (n12) -- (n24);
    \draw[arc] (n13) -- (n14);
    \draw[arc] (n13) -- (n21);
    \draw[arc] (n14) -- (n18);
    \draw[arc] (n14) -- (n20);
    \draw[arc] (n15) -- (n17);
    \draw[arc] (n15) -- (n20);
    \draw[arc] (n16) -- (n20);
    \draw[arc] (n17) -- (n29);
    \draw[arc] (n18) -- (n23);
    \draw[arc] (n19) -- (n23);
    \draw[arc] (n19) -- (n25);
    \draw[arc] (n19) -- (n27);
    \draw[arc] (n20) -- (n30);
    \draw[arc] (n21) -- (n24);
    \draw[arc] (n21) -- (n25);
    \draw[arc] (n21) -- (n26);
    \draw[arc] (n22) -- (n29);
    \draw[arc] (n23) -- (n24);
    \draw[arc] (n24) -- (n30);
    \draw[arc] (n25) -- (n28);
    \draw[arc] (n26) -- (n28);
    \draw[arc] (n27) -- (n30);
    \draw[arc] (n28) -- (n31);
    \draw[arc] (n29) -- (n31);
    \draw[arc] (n30) -- (n31);
    \draw[arc] (n31) -- (T);
    \draw[arc] (n32) -- (n33);
    \draw[arc] (n32) -- (n34);
    \draw[arc] (n32) -- (n35);
    \draw[arc] (n33) -- (n36);
    \draw[arc] (n33) -- (n41);
    \draw[arc] (n33) -- (n42);
    \draw[arc] (n34) -- (n38);
    \draw[arc] (n34) -- (n40);
    \draw[arc] (n34) -- (n42);
    \draw[arc] (n35) -- (n37);
    \draw[arc] (n35) -- (n39);
    \draw[arc] (n35) -- (n41);
    \draw[arc] (n36) -- (n37);
    \draw[arc] (n36) -- (n38);
    \draw[arc] (n36) -- (n47);
    \draw[arc] (n37) -- (n60);
    \draw[arc] (n38) -- (n44);
    \draw[arc] (n38) -- (n45);
    \draw[arc] (n39) -- (n43);
    \draw[arc] (n39) -- (n46);
    \draw[arc] (n39) -- (n50);
    \draw[arc] (n40) -- (n46);
    \draw[arc] (n40) -- (n51);
    \draw[arc] (n40) -- (n56);
    \draw[arc] (n41) -- (n44);
    \draw[arc] (n41) -- (n55);
    \draw[arc] (n41) -- (n61);
    \draw[arc] (n42) -- (n51);
    \draw[arc] (n43) -- (n44);
    \draw[arc] (n43) -- (n53);
    \draw[arc] (n43) -- (n56);
    \draw[arc] (n44) -- (n48);
    \draw[arc] (n44) -- (n62);
    \draw[arc] (n45) -- (n48);
    \draw[arc] (n45) -- (n49);
    \draw[arc] (n46) -- (n61);
    \draw[arc] (n47) -- (n51);
    \draw[arc] (n47) -- (n53);
    \draw[arc] (n48) -- (n54);
    \draw[arc] (n49) -- (n58);
    \draw[arc] (n50) -- (n53);
    \draw[arc] (n51) -- (n52);
    \draw[arc] (n51) -- (n54);
    \draw[arc] (n52) -- (n57);
    \draw[arc] (n53) -- (n54);
    \draw[arc] (n53) -- (n57);
    \draw[arc] (n53) -- (n59);
    \draw[arc] (n54) -- (n60);
    \draw[arc] (n55) -- (n60);
    \draw[arc] (n55) -- (n62);
    \draw[arc] (n56) -- (n58);
    \draw[arc] (n56) -- (n59);
    \draw[arc] (n57) -- (n58);
    \draw[arc] (n58) -- (n61);
    \draw[arc] (n59) -- (n62);
    \draw[arc] (n60) -- (n63);
    \draw[arc] (n61) -- (n63);
    \draw[arc] (n62) -- (n63);
    \draw[arc] (n63) -- (T);
    \draw[arc] (S) -- (n0);
    \draw[arc] (S) -- (n32);
  \end{scope}
\end{tikzpicture}
    \caption{Precedence structure in the example instance created from two RCPSP instances. In grey: original (from the RCPSP instances) and new global dummy vertices; in orange: energy-intensive tasks.}
    \label{fig:precedence_example}
\end{figure}

\clearpage
\newpage
\section{Evolution of Upper \& Lower bounds between approaches}
\label{section:ub-lb-evolution}

In Figure~\ref{fig:evolution_of_UB_gap_5_5_1.0}, the evolution of the value of the upper bound (UB) and of the gap to the best proven lower bound (LB) of methods \LBBDG (without \Freecp warmstart), \ILPF, and \Freecp, for instance standard\_5\_5, is presented. This particular instance was selected due to its representativeness of the algorithms' behavior. We can see that the \Freecp is the method having the steepest UB curve (see Figure~\ref{fig:evolution_of_UB_5_5_1.0}). The \LBBDG method follows closely behind, exhibiting a similarly sharp incline. This means that these approaches quickly converge to a good solution. In contrast, \ILPF takes more time to improve its UB. However, it ends up by finding and proving the optimal solution, as the \LBBDG does, whereas the \Freecp gets stuck with a good but non-optimal solution, which seems unable to improve this solution. Indeed, its LB is very weak compared to the \texttt{ILP} approaches (Figure~\ref{fig:evolution_of_gap_5_5_1.0}). For information, in this instance, the \LBBDG generates 18 feasibility cuts and 15 optimality cuts, with an average size of $\mathcal{I}nf$ of 1.333 variables. Moreover, for an overall computation time of 35 seconds, the solver passes around 0.6 seconds calculating the subproblems (considering that 33 subproblems were computed in total, this represents approximately 18 milliseconds per subproblem).

\begin{figure}[h!]
    \centering
    \begin{subfigure}[t]{0.48\linewidth}
        \centering
        \begin{tikzpicture}
\begin{axis}[
    width=0.75\linewidth,
    height=0.4\linewidth,
    scale only axis,
    grid=both,
    xlabel={time [s]},
    ylabel={objective value [-]},
    xmin=0, xmax=180, ymax=63000,
     scaled y ticks=false,
    y tick label style={
        /pgf/number format/fixed,
        /pgf/number format/precision=0 
    },
    legend style={
        legend cell align=left,   
    },
]

\pgfplotsset{
  mybox A/.style={solid,  draw=blue!60,  fill opacity=0.45},
  mybox B/.style={solid,   draw=red!60,   fill opacity=0.45},
  mybox C/.style={solid,  draw=green!60, fill opacity=0.45},
  mybox D/.style={solid,draw=orange!60,fill opacity=0.45},
  mybox E/.style={solid,draw=gray!60,fill opacity=0.45},
}

    \addplot[mybox A, very thick
    ] table[
        x=Time,
        y=ObjLBBD,
        col sep=comma
    ] {fig/tikz_data/RCPSP_standar_instances_5_5_tikz.csv};
    \addlegendentry{LBBD-opt}


    \addplot[mybox D, very thick
    ] table[
        x=Time,
        y=ObjCP,
        col sep=comma
    ] {fig/tikz_data/RCPSP_standar_instances_5_5_tikz.csv};
    \addlegendentry{freecp}

    \addplot[mybox E, very thick
    ] table[
        x=Time,
        y=ObjILP,
        col sep=comma
    ] {fig/tikz_data/RCPSP_standar_instances_5_5_tikz.csv};
    \addlegendentry{ILP-fsws}



    \addplot[
        black,
        very thick,
        dotted
    ] table[
        x=Time,
        y=ObjOptimal,
        col sep=comma
    ] {fig/tikz_data/RCPSP_standar_instances_5_5_tikz.csv};
    \addlegendentry{Optimum}
\end{axis}
\end{tikzpicture}
        \caption{Evolution of the upper bound (UB).}
        \label{fig:evolution_of_UB_5_5_1.0}
    \end{subfigure}~~~
    \begin{subfigure}[t]{0.48\linewidth}
        \centering
        \begin{tikzpicture}
\begin{axis}[
    width=0.75\linewidth,
    height=0.4\linewidth,
    scale only axis,
    grid=both,
    xlabel={time [s]},
    ylabel={optimality gap [\%]},
    xmin=0, xmax=180,
     scaled y ticks=false,
    y tick label style={
        /pgf/number format/fixed,
        /pgf/number format/precision=0 
    },
    legend pos=north east,        
    legend style={
        legend cell align=left,   
        yshift=-15pt             
    },
]

\pgfplotsset{
  mybox A/.style={solid,  draw=blue!60,  fill opacity=0.45},
  mybox B/.style={solid,   draw=red!60,   fill opacity=0.45},
  mybox C/.style={solid,  draw=green!60, fill opacity=0.45},
  mybox D/.style={solid,draw=orange!60,fill opacity=0.45},
  mybox E/.style={solid,draw=gray!60,fill opacity=0.45},
}


    \addplot[mybox A, very thick,
        dashed
    ] table[
        x=Time,
        y=GapLBBD,
        col sep=comma
    ] {fig/tikz_data/RCPSP_standar_instances_5_5_tikz.csv};
    \addlegendentry{LBBD-opt}


    \addplot[mybox D, very thick,
        dashed
    ] table[
        x=Time,
        y=GapCP,
        col sep=comma
    ] {fig/tikz_data/RCPSP_standar_instances_5_5_tikz.csv};
    \addlegendentry{freecp}

    \addplot[mybox E, very thick,
        dashed
    ] table[
        x=Time,
        y=GapILP,
        col sep=comma
    ] {fig/tikz_data/RCPSP_standar_instances_5_5_tikz.csv};
    \addlegendentry{ILP-fsws}



linewidth\end{axis}
\end{tikzpicture}
        \caption{Evolution of the gap of the best proven lower bound (LB).}
        \label{fig:evolution_of_gap_5_5_1.0}
    \end{subfigure}
    \caption{Evolution of the UB and the optimality gap during the solution in standard instance 5\_5 and $\alpha = 1$. }
    \label{fig:evolution_of_UB_gap_5_5_1.0}
\end{figure}

\clearpage
\newpage
\section{Boxplots of real improvements over \ILPF}
\label{section:boxplots}

This section displays the real improvements over \ILPF in the form of boxplots per set and per $\alpha$ values. Three methods are considered, \LBBDG (\LBBD in sparse instances), \ILPC and \Freecp. The real improvements are computed for any method $\mu$ as:
\begin{small}
\begin{align}
\label{eq:delta}
    \delta_{\mu} = \frac{obj_{\ILPF}-obj_\mu}{|obj_{\ILPF}|}\cdot 100
\end{align}
\end{small}

\begin{figure}[h!]
    \centering
    \resizebox{\linewidth}{!}{\begin{tikzpicture}

\pgfplotsset{
  mybox A/.style={solid, fill=blue!35,   draw=blue!60,   fill opacity=0.45}, 
  mybox B/.style={solid, fill=green!35,  draw=green!60,  fill opacity=0.45}, 
  mybox C/.style={solid, fill=orange!50, draw=orange!60, fill opacity=0.45}, 
}

\def\offA{-0.25}
\def\offB{ 0.00}
\def\offC{ 0.25}

\pgfplotstableread[col sep=comma]{fig/tikz_data/gap_standard_alpha25.csv}\datAlphaA
\pgfplotstableread[col sep=comma]{fig/tikz_data/gap_standard_alpha50.csv}\datAlphaB
\pgfplotstableread[col sep=comma]{fig/tikz_data/gap_standard_alpha75.csv}\datAlphaC
\pgfplotstableread[col sep=comma]{fig/tikz_data/gap_standard_alpha100.csv}\datAlphaD

\newcommand{\drawbox}[5]{%
  \pgfplotstablegetelem{#2}{median}\of#1\edef\bxmed{\pgfplotsretval}%
  \pgfplotstablegetelem{#2}{q1}\of#1\edef\bxqone{\pgfplotsretval}%
  \pgfplotstablegetelem{#2}{q3}\of#1\edef\bxqthree{\pgfplotsretval}%
  \pgfplotstablegetelem{#2}{lower}\of#1\edef\bxlo{\pgfplotsretval}%
  \pgfplotstablegetelem{#2}{upper}\of#1\edef\bxhi{\pgfplotsretval}%
  \edef\tmpaddplot{%
    \noexpand\addplot+[#3,%
      boxplot prepared={%
        median=\bxmed,%
        lower quartile=\bxqone,%
        upper quartile=\bxqthree,%
        lower whisker=\bxlo,%
        upper whisker=\bxhi%
      },%
      boxplot/draw position={#5+#4}%
    ] coordinates {}%
  }%
  \tmpaddplot;%
}

\newcommand{\drawset}[2]{%
  \pgfmathtruncatemacro{\rowF}{(#2-1)*3+0}%
  \pgfmathtruncatemacro{\rowI}{(#2-1)*3+1}%
  \pgfmathtruncatemacro{\rowL}{(#2-1)*3+2}%
  \drawbox{#1}{\rowF}{mybox B}{\offB}{#2}
  \drawbox{#1}{\rowI}{mybox C}{\offC}{#2}
  \drawbox{#1}{\rowL}{mybox A}{\offA}{#2}
}

\begin{groupplot}[
  group style={
    group size=2 by 2,
    horizontal sep=1.2cm,
    vertical sep=1.6cm,
  },
  width=0.5\textwidth, height=0.30\textwidth,
  ymajorgrids,
  enlarge x limits=0.03,
  boxplot/box extend=0.18,
  boxplot/draw direction=y,
  xtick={1,...,15},
  xticklabels={1,2,3,4,5,6,7,8,9,10,11,12,13,14,15},
  xlabel={Instance set [-]},
  ylabel={rel. impr. over ILP-fsws [\%]},
  every axis plot/.append style={solid},
  every boxplot/.append style={forget plot},
  legend image post style={fill opacity=1, draw opacity=1},
  legend style={
    at={(0.02,0.98)}, anchor=north west,
    font=\scriptsize, fill=white, fill opacity=0.9, text opacity=1,
    draw=black!30,
  },
  title style={font=\small},
]

\nextgroupplot[title={$\alpha = 0.25$}, ymin=-5, ymax=35, xlabel=\empty]
\foreach \g in {1,...,15}{\drawset{\datAlphaA}{\g}}
\addlegendimage{area legend, mybox A}\addlegendentry{\LBBDG}
\addlegendimage{area legend, mybox B}\addlegendentry{\Freecp}
\addlegendimage{area legend, mybox C}\addlegendentry{\ILPC}

\nextgroupplot[
  title={$\alpha = 0.5$},
  ymin=-5,
  ymax=25,
  ylabel=\empty,
  xlabel=\empty
]
\foreach \g in {1,...,15}{\drawset{\datAlphaB}{\g}}

\nextgroupplot[title={$\alpha = 0.75$}, ymin=-6, ymax=30]
\foreach \g in {1,...,15}{\drawset{\datAlphaC}{\g}}

\nextgroupplot[
  title={$\alpha = 1.0$},
  ymin=-20,
  ymax=40,
  ylabel=\empty
]
\foreach \g in {1,...,15}{\drawset{\datAlphaD}{\g}}

\end{groupplot}
\end{tikzpicture}}
    \caption{Real improvements of different methods relative to \ILPF in standard sets.}
    \label{fig:standard_gap_to_ilp}
\end{figure}

\begin{figure}[h!]
    \centering
    \resizebox{\linewidth}{!}{\begin{tikzpicture}

\pgfplotsset{
  mybox A/.style={solid, fill=blue!35,   draw=blue!60,   fill opacity=0.45}, 
  mybox B/.style={solid, fill=green!35,  draw=green!60,  fill opacity=0.45}, 
  mybox C/.style={solid, fill=orange!50, draw=orange!60, fill opacity=0.45}, 
}

\def\offA{-0.25}
\def\offB{ 0.00}
\def\offC{ 0.25}

\pgfplotstableread[col sep=comma]{fig/tikz_data/gap_dense_alpha25.csv}\datAlphaA
\pgfplotstableread[col sep=comma]{fig/tikz_data/gap_dense_alpha50.csv}\datAlphaB
\pgfplotstableread[col sep=comma]{fig/tikz_data/gap_dense_alpha75.csv}\datAlphaC
\pgfplotstableread[col sep=comma]{fig/tikz_data/gap_dense_alpha100.csv}\datAlphaD

\newcommand{\drawbox}[5]{%
  \pgfplotstablegetelem{#2}{median}\of#1\edef\bxmed{\pgfplotsretval}%
  \pgfplotstablegetelem{#2}{q1}\of#1\edef\bxqone{\pgfplotsretval}%
  \pgfplotstablegetelem{#2}{q3}\of#1\edef\bxqthree{\pgfplotsretval}%
  \pgfplotstablegetelem{#2}{lower}\of#1\edef\bxlo{\pgfplotsretval}%
  \pgfplotstablegetelem{#2}{upper}\of#1\edef\bxhi{\pgfplotsretval}%
  \edef\tmpaddplot{%
    \noexpand\addplot+[#3,%
      boxplot prepared={%
        median=\bxmed,%
        lower quartile=\bxqone,%
        upper quartile=\bxqthree,%
        lower whisker=\bxlo,%
        upper whisker=\bxhi%
      },%
      boxplot/draw position={#5+#4}%
    ] coordinates {}%
  }%
  \tmpaddplot;%
}

\newcommand{\drawset}[2]{%
  \pgfmathtruncatemacro{\rowF}{(#2-1)*3+0}%
  \pgfmathtruncatemacro{\rowI}{(#2-1)*3+1}%
  \pgfmathtruncatemacro{\rowL}{(#2-1)*3+2}%
  \drawbox{#1}{\rowF}{mybox B}{\offB}{#2}
  \drawbox{#1}{\rowI}{mybox C}{\offC}{#2}
  \drawbox{#1}{\rowL}{mybox A}{\offA}{#2}
}

\begin{groupplot}[
  group style={
    group size=2 by 2,
    horizontal sep=1.2cm,
    vertical sep=1.6cm,
  },
  width=0.5\textwidth, height=0.30\textwidth,
  ymajorgrids,
  enlarge x limits=0.03,
  boxplot/box extend=0.18,
  boxplot/draw direction=y,
  xtick={1,...,7},
  xticklabels={1,2,3,4,5,6,7},
  xlabel={Instance set [-]},
  ylabel={rel. impr. over ILP-fsws [\%]},
  every axis plot/.append style={solid},
  every boxplot/.append style={forget plot},
  legend image post style={fill opacity=1, draw opacity=1},
  legend style={
    at={(0.02,0.98)}, anchor=north west,
    font=\scriptsize, fill=white, fill opacity=0.9, text opacity=1,
    draw=black!30,
  },
  title style={font=\small},
]

\nextgroupplot[title={$\alpha = 0.25$}, ymin=-6, ymax=30, xlabel=\empty]
\foreach \g in {1,...,7}{\drawset{\datAlphaA}{\g}}
\addlegendimage{area legend, mybox A}\addlegendentry{\LBBDG}
\addlegendimage{area legend, mybox B}\addlegendentry{\Freecp}
\addlegendimage{area legend, mybox C}\addlegendentry{\ILPC}

\nextgroupplot[
  title={$\alpha = 0.5$},
  ymin=-6,
  ymax=35,
  ylabel=\empty,
  xlabel=\empty
]
\foreach \g in {1,...,7}{\drawset{\datAlphaB}{\g}}

\nextgroupplot[title={$\alpha = 0.75$}, ymin=-10, ymax=30]
\foreach \g in {1,...,7}{\drawset{\datAlphaC}{\g}}

\nextgroupplot[
  title={$\alpha = 1.0$},
  ymin=-25,
  ymax=40,
  ylabel=\empty
]
\foreach \g in {1,...,7}{\drawset{\datAlphaD}{\g}}

\end{groupplot}
\end{tikzpicture}}
    \caption{Real improvements of different methods relative to \ILPF in dense sets.}
    \label{fig:dense_gap_to_ilp}
\end{figure}

\begin{figure}[h!]
    \centering
    \resizebox{\linewidth}{!}{\begin{tikzpicture}

\pgfplotsset{
  mybox A/.style={solid, fill=blue!35,   draw=blue!60,   fill opacity=0.45}, 
  mybox B/.style={solid, fill=green!35,  draw=green!60,  fill opacity=0.45}, 
  mybox C/.style={solid, fill=orange!50, draw=orange!60, fill opacity=0.45}, 
}

\def\offA{-0.25}
\def\offB{ 0.00}
\def\offC{ 0.25}

\pgfplotstableread[col sep=comma]{fig/tikz_data/gap_sparse_alpha25.csv}\datAlphaA
\pgfplotstableread[col sep=comma]{fig/tikz_data/gap_sparse_alpha50.csv}\datAlphaB
\pgfplotstableread[col sep=comma]{fig/tikz_data/gap_sparse_alpha75.csv}\datAlphaC
\pgfplotstableread[col sep=comma]{fig/tikz_data/gap_sparse_alpha100.csv}\datAlphaD

\newcommand{\drawbox}[5]{%
  \pgfplotstablegetelem{#2}{median}\of#1\edef\bxmed{\pgfplotsretval}%
  \pgfplotstablegetelem{#2}{q1}\of#1\edef\bxqone{\pgfplotsretval}%
  \pgfplotstablegetelem{#2}{q3}\of#1\edef\bxqthree{\pgfplotsretval}%
  \pgfplotstablegetelem{#2}{lower}\of#1\edef\bxlo{\pgfplotsretval}%
  \pgfplotstablegetelem{#2}{upper}\of#1\edef\bxhi{\pgfplotsretval}%
  \edef\tmpaddplot{%
    \noexpand\addplot+[#3,%
      boxplot prepared={%
        median=\bxmed,%
        lower quartile=\bxqone,%
        upper quartile=\bxqthree,%
        lower whisker=\bxlo,%
        upper whisker=\bxhi%
      },%
      boxplot/draw position={#5+#4}%
    ] coordinates {}%
  }%
  \tmpaddplot;%
}

\newcommand{\drawset}[2]{%
  \pgfmathtruncatemacro{\rowF}{(#2-1)*3+0}%
  \pgfmathtruncatemacro{\rowI}{(#2-1)*3+1}%
  \pgfmathtruncatemacro{\rowL}{(#2-1)*3+2}%
  \drawbox{#1}{\rowF}{mybox B}{\offB}{#2}
  \drawbox{#1}{\rowI}{mybox C}{\offC}{#2}
  \drawbox{#1}{\rowL}{mybox A}{\offA}{#2}
}

\begin{groupplot}[
  group style={
    group size=2 by 2,
    horizontal sep=1.2cm,
    vertical sep=1.6cm,
  },
  width=0.5\textwidth, height=0.30\textwidth,
  ymajorgrids,
  enlarge x limits=0.03,
  boxplot/box extend=0.18,
  boxplot/draw direction=y,
  xtick={1,...,20},
  xticklabels={1,2,3,4,5,6,7,8,9,10,11,12,13,14,15,16,17,18,19,20},
  xticklabel style={font=\tiny},
  xlabel={Instance set [-]},
  ylabel={rel. impr. over ILP-fsws [\%]},
  every axis plot/.append style={solid},
  every boxplot/.append style={forget plot},
  legend image post style={fill opacity=1, draw opacity=1},
  legend style={
    at={(0.02,0.98)}, anchor=north west,
    font=\scriptsize, fill=white, fill opacity=0.9, text opacity=1,
    draw=black!30,
  },
  title style={font=\small},
]

\nextgroupplot[title={$\alpha = 0.25$}, ymin=-55, ymax=75, xlabel=\empty]
\foreach \g in {1,...,20}{\drawset{\datAlphaA}{\g}}
\addlegendimage{area legend, mybox A}\addlegendentry{\LBBD}
\addlegendimage{area legend, mybox B}\addlegendentry{\Freecp}
\addlegendimage{area legend, mybox C}\addlegendentry{\ILPC}

\nextgroupplot[
  title={$\alpha = 0.5$},
  ymin=-70,
  ymax=110,
  ylabel=\empty,
  xlabel=\empty
]
\foreach \g in {1,...,20}{\drawset{\datAlphaB}{\g}}

\nextgroupplot[title={$\alpha = 0.75$}, ymin=-70, ymax=110]
\foreach \g in {1,...,20}{\drawset{\datAlphaC}{\g}}

\nextgroupplot[
  title={$\alpha = 1.0$},
  ymin=-70,
  ymax=95,
  ylabel=\empty
]
\foreach \g in {1,...,20}{\drawset{\datAlphaD}{\g}}

\end{groupplot}
\end{tikzpicture}}
    \caption{Real improvements of different methods relative to \ILPF in sparse sets.}
    \label{fig:sparse_gap_to_ilp}
\end{figure}

\clearpage
\newpage
\section{Detailed results for different density sets}
\label{section:detailed_results}

This section presents for each density the average of results aggregated per set and per $\alpha$ value. The best results are reported in bold. The computation times are given in seconds, and the gaps in percentages.

\begin{table}[h]
\begin{tiny}
\centering
\setlength{\tabcolsep}{2pt}
\makebox[\textwidth][c]{
\begin{tabular}{|cr|rrrr|rrrr|rrrr|}
\toprule
 &  & \multicolumn{4}{c|}{Objective value [-]} & \multicolumn{4}{c|}{Computation Time [s]} & \multicolumn{4}{c|}{Gap to LB [\%]} \\
 & & \Freecp & \ILPC & \ILPF & \LBBDG & \Freecp & \ILPC & \ILPF & \LBBDG & \Freecp & \ILPC & \ILPF & \LBBDG \\
$\alpha$ & Set &  &  &  &  &  &  &  &  &  &  &  &  \\
\midrule
\multirow[c]{7}{*}{0.25} & 1 & \textbf{1.05860} & \textbf{1.05860} & \textbf{1.05860} & \textbf{1.05860} & 1321.14 & 60.03 & \textbf{11.31} & 57.41 & 10.93 & \textbf{0.00} & \textbf{0.00} & \textbf{0.00} \\
 & 2 & \textbf{1.05505} & \textbf{1.05505} & \textbf{1.05505} & \textbf{1.05505} & TLR & 225.50 & 232.72 & \textbf{186.48} & 26.75 & \textbf{0.25} & 0.26 & 0.26 \\
 & 3 & \textbf{1.04645} & \textbf{1.04645} & \textbf{1.04645} & \textbf{1.04645} & TLR & \textbf{548.16} & 609.66 & 630.62 & 38.21 & \textbf{0.07} & 0.10 & 0.28 \\
 & 4 & \textbf{7.32443} & 7.41852 & 7.80674 & 7.68027 & TLR & 1142.44 & 1358.85 & \textbf{1070.25} & 66.95 & \textbf{1.13} & 2.40 & 1.70 \\
 & 5 & 80.93551 & 82.40018 & 86.34838 & \textbf{77.83109} & \textbf{TLR} & TLR & TLR & TLR & 89.64 & 24.93 & 29.50 & \textbf{3.89} \\
 & 6 & 97.97512 & 100.03595 & 105.51415 & \textbf{97.43803} & \textbf{TLR} & TLR & TLR & TLR & 172.56 & 5544.87 & 5007.17 & \textbf{41.96} \\
 & 7 & \textbf{82.73975} & 84.94907 & 88.58123 & 84.16912 & TLR & 1415.57 & 1352.68 & \textbf{1132.79} & 72.30 & 160.65 & 148.05 & \textbf{44.96} \\
\cline{1-14}
\multirow[c]{7}{*}{0.5} & 1 & \textbf{1.11721} & \textbf{1.11721} & \textbf{1.11721} & \textbf{1.11721} & 1538.67 & 136.00 & \textbf{73.27} & 185.64 & 33.28 & \textbf{0.00} & \textbf{0.00} & 0.17 \\
 & 2 & \textbf{1.08590} & \textbf{1.08590} & \textbf{1.08590} & \textbf{1.08590} & TLR & \textbf{505.29} & 644.50 & 800.25 & 39.04 & \textbf{0.46} & 0.61 & 1.13 \\
 & 3 & \textbf{1.08238} & \textbf{1.08238} & \textbf{1.08238} & \textbf{1.08238} & TLR & 1308.08 & 1268.98 & \textbf{1201.41} & 71.10 & 2.19 & \textbf{1.88} & 2.11 \\
 & 4 & 16.59966 & \textbf{16.13410} & 18.41283 & 16.91629 & TLR & \textbf{1642.80} & 1742.26 & 1661.19 & 123.92 & \textbf{2.71} & 6.07 & 3.61 \\
 & 5 & 152.09141 & 161.23074 & 170.34044 & \textbf{150.40514} & TLR & TLR & TLR & \textbf{1780.22} & 87.73 & 20.75 & 26.10 & \textbf{2.21} \\
 & 6 & 194.42347 & 200.26270 & 210.00732 & \textbf{191.02657} & \textbf{TLR} & TLR & TLR & TLR & 174.98 & 8111.50 & 7567.59 & \textbf{75.73} \\
 & 7 & \textbf{163.94154} & 168.47701 & 176.14825 & 166.26513 & TLR & 1413.53 & 1354.23 & \textbf{1132.26} & 72.69 & 153.16 & 148.50 & \textbf{44.37} \\
\cline{1-14}
\multirow[c]{7}{*}{0.75} & 1 & \textbf{1.10964} & \textbf{1.10964} & \textbf{1.10964} & \textbf{1.10964} & TLR & 382.29 & 317.62 & \textbf{264.16} & 40.76 & 0.19 & 0.12 & \textbf{0.00} \\
 & 2 & \textbf{1.09154} & 1.09175 & 1.09179 & 1.09191 & TLR & \textbf{1474.39} & 1561.03 & 1540.80 & 65.68 & \textbf{2.07} & 2.73 & 2.86 \\
 & 3 & 10.39699 & \textbf{10.17224} & 10.18724 & 10.17340 & TLR & TLR & TLR & \textbf{1799.39} & 86.40 & 6.93 & 8.03 & \textbf{6.67} \\
 & 4 & \textbf{28.32943} & 29.35689 & 31.93691 & 28.91132 & \textbf{TLR} & TLR & TLR & TLR & 165.17 & 6.73 & 9.52 & \textbf{6.53} \\
 & 5 & 234.38472 & 239.65550 & 255.20947 & \textbf{225.22221} & TLR & TLR & TLR & \textbf{1757.85} & 88.36 & 50.12 & 30.37 & \textbf{2.67} \\
 & 6 & 290.29475 & 300.76084 & 314.50048 & \textbf{281.58823} & \textbf{TLR} & TLR & TLR & TLR & 178.78 & 5634.02 & 2792.82 & \textbf{41.29} \\
 & 7 & \textbf{242.25874} & 252.53526 & 263.71497 & 248.55957 & TLR & 1503.01 & 1350.43 & \textbf{1132.21} & 72.70 & 166.58 & 141.40 & \textbf{47.32} \\
\cline{1-14}
\multirow[c]{7}{*}{1.0} & 1 & \textbf{49408.81750} & \textbf{49408.81750} & \textbf{49408.81750} & \textbf{49408.81750} & TLR & 292.15 & \textbf{225.00} & 228.09 & 64.93 & 0.07 & \textbf{0.03} & 0.16 \\
 & 2 & 108772.76421 & \textbf{107160.87684} & 107161.49211 & 107161.61316 & TLR & 354.45 & 282.75 & \textbf{156.03} & 118.42 & \textbf{0.00} & \textbf{0.00} & 0.01 \\
 & 3 & 144168.30895 & 139158.31105 & 139449.32421 & \textbf{139158.01211} & TLR & 755.11 & 910.36 & \textbf{436.88} & 124.05 & 0.02 & 0.27 & \textbf{0.01} \\
 & 4 & 173414.08050 & 166458.56800 & 180866.29849 & \textbf{163275.10700} & TLR & 1616.52 & 1648.86 & \textbf{1202.86} & 196.45 & 1.98 & 10.19 & \textbf{0.06} \\
 & 5 & 265937.62500 & 262395.06300 & 283524.59650 & \textbf{250761.59450} & TLR & TLR & TLR & \textbf{1442.99} & 88.88 & 5.54 & 12.05 & \textbf{0.32} \\
 & 6 & 393629.11450 & 407528.09250 & 426675.22200 & \textbf{368332.18250} & TLR & TLR & TLR & \textbf{1608.22} & 174.00 & 49.73 & 15.43 & \textbf{0.13} \\
 & 7 & \textbf{377964.79000} & 391617.81950 & 411972.01100 & 379780.83050 & TLR & 1509.29 & 1355.02 & \textbf{1181.05} & 72.86 & 159.47 & 130.97 & \textbf{45.95} \\
\bottomrule
\end{tabular}
}
\end{tiny}
\caption{Detailed results of the dense density sets. Values are averaged over instances within each set. ``TLR'' = Time Limit Reached ($\geq 1800$\,s). Best values are reported in bold.}
\label{tab:agg-results}
\end{table}

\begin{table}
\begin{tiny}
\centering
\setlength{\tabcolsep}{2pt}
\makebox[\textwidth][c]{
\begin{tabular}{|cr|rrrr|rrrr|rrrr|}
\toprule
 &  & \multicolumn{4}{c|}{Objective value [-]} & \multicolumn{4}{c|}{Computation Time [s]} & \multicolumn{4}{c|}{Final gap to LB [\%]} \\
 $\alpha$ & Set & \texttt{freecp} & \texttt{ILP-cpws} & \texttt{ILP-fsws} & \LBBDG & \texttt{freecp} & \texttt{ILP-cpws} & \texttt{ILP-fsws} & \LBBDG & \texttt{freecp} & \texttt{ILP-cpws} & \texttt{ILP-fsws} & \LBBDG \\
\midrule
\multirow[c]{15}{*}{0.25} & 1 & \textbf{1.08459} & \textbf{1.08459} & \textbf{1.08459} & \textbf{1.08459} & 349.96 & 48.08 & \textbf{20.12} & 36.81 & 23.88 & \textbf{0.00} & \textbf{0.00} & \textbf{0.00} \\
 & 2 & \textbf{1.05195} & \textbf{1.05195} & \textbf{1.05195} & \textbf{1.05195} & 871.22 & 83.85 & 252.36 & \textbf{57.51} & 3.14 & \textbf{0.00} & \textbf{0.00} & \textbf{0.00} \\
 & 3 & 1.07074 & 1.07055 & 1.10652 & \textbf{1.07026} & 1764.58 & \textbf{340.66} & 1564.50 & 426.34 & 23.86 & \textbf{0.25} & 12.47 & 0.26 \\
 & 4 & 1.05507 & 1.06199 & 1.13507 & \textbf{1.05270} & TLR & 854.79 & 1796.06 & \textbf{278.73} & 10.84 & 0.99 & 23.99 & \textbf{0.16} \\
 & 5 & 1.04661 & 1.04776 & 1.12266 & \textbf{1.04534} & TLR & 816.89 & TLR & \textbf{446.09} & 16.29 & 0.48 & 28.09 & \textbf{0.11} \\
 & 6 & 1.04790 & 1.05408 & 1.14501 & \textbf{1.04505} & TLR & 1218.21 & TLR & \textbf{514.52} & 32.56 & 1.31 & 31.34 & \textbf{0.15} \\
 & 7 & 1.05982 & 1.06266 & 1.13424 & \textbf{1.05974} & TLR & 990.98 & TLR & \textbf{504.06} & 41.58 & 0.59 & 32.16 & \textbf{0.58} \\
 & 8 & \textbf{1.07104} & 1.07438 & 1.15224 & 1.07317 & TLR & 1219.45 & TLR & \textbf{1201.63} & 36.90 & \textbf{0.54} & 33.31 & 1.46 \\
 & 9 & 1.04976 & 1.03908 & 1.10634 & \textbf{1.03680} & TLR & 949.61 & TLR & \textbf{850.48} & 18.79 & \textbf{0.38} & 33.59 & 0.69 \\
 & 10 & \textbf{1.01389} & 1.02791 & 1.06565 & 1.01404 & TLR & TLR & TLR & \textbf{1296.81} & 33.71 & 1.93 & 37.79 & \textbf{0.55} \\
 & 11 & \textbf{1.04572} & 1.05748 & 1.09257 & 1.04614 & TLR & 1630.08 & 1744.89 & \textbf{1500.07} & 25.49 & 1.77 & 5.59 & \textbf{0.67} \\
 & 12 & \textbf{1.03324} & 1.04158 & 1.08769 & 1.03338 & TLR & 1740.37 & 1752.09 & \textbf{981.44} & 27.06 & 6.73 & 11.32 & \textbf{0.51} \\
 & 13 & 3.40062 & 3.41969 & 3.60836 & \textbf{3.31770} & TLR & 1783.06 & 1759.36 & \textbf{1202.39} & 30.88 & 7.52 & 16.54 & \textbf{0.75} \\
 & 14 & 11.19105 & 11.24702 & 11.96225 & \textbf{10.66233} & TLR & TLR & 1764.39 & \textbf{1422.41} & 27.40 & 26.46 & 30.31 & \textbf{0.45} \\
 & 15 & 5.21018 & 5.21865 & 5.63304 & \textbf{4.85424} & TLR & TLR & 1771.08 & \textbf{787.89} & 17.63 & 105.91 & 105.90 & \textbf{0.06} \\
\cline{1-14}
\multirow[c]{15}{*}{0.5} & 1 & \textbf{1.14210} & \textbf{1.14210} & \textbf{1.14210} & \textbf{1.14210} & 681.41 & 69.47 & \textbf{61.51} & 72.08 & 45.16 & \textbf{0.00} & \textbf{0.00} & \textbf{0.00} \\
 & 2 & \textbf{1.10034} & \textbf{1.10034} & \textbf{1.10034} & \textbf{1.10034} & 1432.25 & 84.68 & 294.92 & \textbf{70.68} & 10.12 & \textbf{0.00} & \textbf{0.00} & \textbf{0.00} \\
 & 3 & \textbf{1.13696} & 1.13726 & 1.14840 & 1.13768 & TLR & \textbf{716.94} & 1665.01 & 824.88 & 53.42 & \textbf{0.64} & 10.40 & 1.78 \\
 & 4 & \textbf{1.10261} & 1.10597 & 1.16521 & 1.10494 & TLR & \textbf{1223.63} & TLR & 1391.25 & 22.87 & 2.60 & 18.42 & \textbf{2.45} \\
 & 5 & \textbf{1.09023} & 1.09062 & 1.14469 & 1.09044 & TLR & 1482.18 & TLR & \textbf{1131.51} & 36.78 & 2.02 & 19.94 & \textbf{1.59} \\
 & 6 & \textbf{1.08917} & 1.09325 & 1.16505 & 1.09037 & TLR & 1600.48 & TLR & \textbf{1324.10} & 68.95 & 2.21 & 22.31 & \textbf{1.28} \\
 & 7 & \textbf{1.11429} & 1.11906 & 1.17585 & 1.11633 & TLR & \textbf{979.56} & TLR & 1208.35 & 78.68 & \textbf{1.34} & 23.72 & 3.46 \\
 & 8 & \textbf{1.14059} & 1.15011 & 1.20299 & 1.14200 & TLR & \textbf{1543.41} & TLR & 1598.82 & 71.60 & \textbf{4.90} & 25.19 & 5.28 \\
 & 9 & \textbf{1.09764} & 1.10026 & 1.14592 & 1.09779 & TLR & \textbf{1332.94} & TLR & 1786.07 & 35.62 & \textbf{2.30} & 23.68 & 3.06 \\
 & 10 & \textbf{1.02097} & 1.03175 & 1.06224 & 1.02704 & TLR & 1761.59 & TLR & \textbf{1606.72} & 68.47 & \textbf{1.83} & 25.99 & 2.96 \\
 & 11 & \textbf{1.08772} & 1.10206 & 1.12132 & 1.08818 & TLR & \textbf{1666.37} & 1719.87 & 1794.28 & 48.82 & \textbf{2.13} & 4.94 & 2.62 \\
 & 12 & 1.06384 & 1.07465 & 1.10225 & \textbf{1.06347} & TLR & \textbf{1696.71} & 1749.30 & 1709.05 & 50.53 & 8.65 & 11.94 & \textbf{2.59} \\
 & 13 & 5.55047 & 5.49835 & 5.85881 & \textbf{5.32971} & TLR & TLR & 1758.87 & \textbf{1584.83} & 52.64 & 9.29 & 12.46 & \textbf{2.78} \\
 & 14 & 20.03995 & 20.39942 & 21.62985 & \textbf{19.06893} & TLR & TLR & 1763.11 & \textbf{1605.99} & 40.33 & 10.37 & 19.50 & \textbf{1.77} \\
 & 15 & 10.10940 & 10.35888 & 11.25391 & \textbf{9.53414} & TLR & TLR & 1769.46 & \textbf{1418.19} & 38.32 & 167.70 & 156.43 & \textbf{0.51} \\
\cline{1-14}
\multirow[c]{15}{*}{0.75} & 1 & \textbf{1.13832} & \textbf{1.13832} & \textbf{1.13832} & \textbf{1.13832} & 901.69 & 87.81 & 89.64 & \textbf{60.90} & 75.51 & \textbf{0.00} & \textbf{0.00} & \textbf{0.00} \\
 & 2 & \textbf{1.11362} & \textbf{1.11362} & \textbf{1.11362} & \textbf{1.11362} & 1746.07 & 161.22 & 417.25 & \textbf{108.25} & 21.68 & \textbf{0.00} & 0.16 & \textbf{0.00} \\
 & 3 & 1.15451 & \textbf{1.15079} & 1.16005 & 1.15177 & TLR & 1437.21 & 1771.71 & \textbf{1326.31} & 88.32 & \textbf{3.20} & 9.50 & 3.49 \\
 & 4 & 1.11382 & 1.11410 & 1.14391 & \textbf{1.11301} & TLR & TLR & TLR & \textbf{1722.50} & 38.47 & 4.95 & 11.00 & \textbf{3.18} \\
 & 5 & 1.11919 & 1.11853 & 1.14569 & \textbf{1.11369} & TLR & TLR & TLR & \textbf{1762.97} & 60.59 & 5.20 & 12.47 & \textbf{4.16} \\
 & 6 & 1.12251 & 1.12756 & 1.16635 & \textbf{1.12227} & TLR & TLR & TLR & \textbf{1560.15} & 107.06 & 6.43 & 14.53 & \textbf{5.06} \\
 & 7 & \textbf{1.12241} & 1.13846 & 1.17222 & 1.12290 & TLR & \textbf{1601.52} & TLR & TLR & 100.20 & 6.36 & 16.27 & \textbf{5.51} \\
 & 8 & \textbf{1.03726} & 1.14741 & 1.23506 & 1.10494 & TLR & \textbf{1716.94} & TLR & 1783.66 & 88.27 & 9.23 & 19.98 & \textbf{7.28} \\
 & 9 & 1.08480 & 1.08635 & 1.12071 & \textbf{1.07815} & TLR & \textbf{1677.04} & TLR & TLR & 50.83 & \textbf{5.25} & 17.60 & 6.17 \\
 & 10 & \textbf{1.01578} & 1.02162 & 1.05158 & 1.01884 & TLR & TLR & TLR & \textbf{1684.74} & 98.55 & \textbf{4.81} & 18.69 & 5.43 \\
 & 11 & 1.12130 & 1.12681 & 1.14778 & \textbf{1.11638} & TLR & TLR & \textbf{1750.61} & TLR & 70.30 & 6.20 & 8.44 & \textbf{5.97} \\
 & 12 & 1.06878 & 1.07141 & 1.09077 & \textbf{1.06523} & TLR & 1774.10 & \textbf{1755.41} & TLR & 43.93 & 6.03 & 8.51 & \textbf{5.03} \\
 & 13 & 8.06592 & 8.05083 & 8.66011 & \textbf{7.78350} & TLR & TLR & 1758.96 & \textbf{1666.11} & 76.15 & 7.47 & 10.99 & \textbf{5.73} \\
 & 14 & 29.31666 & 30.21150 & 31.93878 & \textbf{28.13772} & TLR & TLR & 1763.84 & \textbf{1759.66} & 61.01 & 18.93 & 27.20 & \textbf{4.20} \\
 & 15 & 13.88790 & 14.20358 & 15.55150 & \textbf{13.10518} & TLR & TLR & 1769.70 & \textbf{1742.26} & 71.66 & 26.92 & 814.91 & \textbf{1.83} \\
\cline{1-14}
\multirow[c]{15}{*}{1.0} & 1 & \textbf{18836.80250} & \textbf{18836.80250} & \textbf{18836.80250} & \textbf{18836.80250} & 891.27 & 56.27 & \textbf{5.98} & 52.22 & 107.84 & \textbf{0.00} & \textbf{0.00} & \textbf{0.00} \\
 & 2 & \textbf{36342.00267} & \textbf{36342.00267} & \textbf{36342.00267} & \textbf{36342.00267} & TLR & 66.55 & \textbf{13.09} & 62.09 & 35.22 & \textbf{0.00} & \textbf{0.00} & \textbf{0.00} \\
 & 3 & \textbf{45363.12875} & \textbf{45363.12875} & \textbf{45363.12875} & \textbf{45363.12875} & TLR & 89.34 & 72.59 & \textbf{67.20} & 192.71 & \textbf{0.00} & \textbf{0.00} & \textbf{0.00} \\
 & 4 & 62794.54909 & \textbf{62751.56455} & \textbf{62751.56455} & \textbf{62751.56455} & TLR & 145.17 & 129.38 & \textbf{80.06} & 49.63 & \textbf{0.00}& \textbf{0.00} & \textbf{0.00} \\
 & 5 & 76688.37182 & \textbf{76266.03000} & \textbf{76266.03000} & \textbf{76266.03000} & TLR & 231.32 & 239.58 & \textbf{98.22} & 96.05 & \textbf{0.00} & \textbf{0.00} & \textbf{0.00} \\
 & 6 & 112013.69842 & \textbf{109582.86315} & \textbf{109582.86315} & \textbf{109582.86315} & TLR & 487.40 & 520.72 & \textbf{127.83} & 143.99 & \textbf{0.00} & \textbf{0.00} & \textbf{0.00} \\
 & 7 & 118167.45105 & \textbf{114656.45232} & 114672.92411 & \textbf{114656.45232} & TLR & 842.11 & 906.26 & \textbf{165.40} & 120.52 & \textbf{0.00} & 0.03 & \textbf{0.00} \\
 & 8 & 129669.98737 & 124555.63158 & 124856.84789 & \textbf{124534.03315} & TLR & 1199.41 & 1219.21 & \textbf{224.64} & 137.12 & 0.03 & 0.19 & \textbf{0.00} \\
 & 9 & 183904.69526 & 177028.07632 & 179707.40632 & \textbf{177028.01519} & TLR & 1299.15 & 1382.82 & \textbf{333.15} & 61.48 & 0.02 & 0.81 & \textbf{0.00} \\
 & 10 & 189091.72050 & 186672.55000 & 194835.94949 & \textbf{181718.76700} & TLR & 1688.72 & 1678.52 & \textbf{288.47} & 120.53 & 2.17 & 5.83 & \textbf{0.00} \\
 & 11 & 184661.06550 & 181347.22400 & 188558.34750 & \textbf{177171.48050} & TLR & 1597.09 & 1646.59 & \textbf{243.89} & 90.58 & 1.50 & 5.34 & \textbf{0.01} \\
 & 12 & 221676.40250 & 215938.58150 & 228079.62244 & \textbf{211589.56950} & TLR & 1765.82 & 1749.87 & \textbf{282.67} & 101.02 & 1.58 & 6.59 & \textbf{0.00} \\
 & 13 & 204602.23550 & 196880.24150 & 216430.08050 & \textbf{192963.33359} & TLR & TLR & 1758.98 & \textbf{712.21} & 98.62 & 2.08 & 11.76 & \textbf{0.01} \\
 & 14 & 251520.22900 & 249023.57500 & 281722.66250 & \textbf{236034.34960} & TLR & TLR & 1764.36 & \textbf{751.54} & 157.92 & 5.69 & 17.90 & \textbf{0.01} \\
 & 15 & 256199.55950 & 249610.25150 & 289114.31100 & \textbf{242278.01933} & TLR & TLR & 1770.96 & \textbf{878.84} & 238.85 & 2.93 & 17.35 & \textbf{0.01} \\
\bottomrule
\end{tabular}
}
\end{tiny}
\caption{Detailed results for standard density sets. Values are averaged over instances within each set. ``TLR'' = Time Limit Reached ($\geq 1800$\,s). Best values are reported in bold. Computation times are given in seconds, and gaps are reported as percentages.}
\label{tab:agg-results}
\end{table}

\begin{table}
\begin{tiny}
\centering
\setlength{\tabcolsep}{2pt}
\makebox[\textwidth][c]{
\begin{tabular}{|cr|rrrr|rrrr|rrrr|}
\toprule
 &  & \multicolumn{4}{c|}{Objective value [-]} & \multicolumn{4}{c|}{Computation Time [s]} & \multicolumn{4}{c|}{Gap to LB [\%]} \\
 & & \Freecp & \ILPC & \ILPF & \LBBD & \Freecp & \ILPC & \ILPF & \LBBD & \Freecp & \ILPC & \ILPF & \LBBD \\
$\alpha$ & Set &  &  &  &  &  &  &  &  &  &  &  &  \\
\midrule
\multirow[c]{20}{*}{0.25} & 1 & \textbf{1.09023} & \textbf{1.09023} & \textbf{1.09023} & \textbf{1.09023} & 6.81 & 9.69 & \textbf{4.08} & 7.50 & 0.01 & \textbf{0.00} & \textbf{0.00} & \textbf{0.00} \\
 & 2 & \textbf{1.07280} & 1.07512 & 1.07656 & \textbf{1.07280} & 873.01 & 531.25 & 787.55 & \textbf{70.82} & 5.52 & 0.62 & 0.84 & \textbf{0.00} \\
 & 3 & 1.08411 & 1.08868 & 1.09532 & \textbf{1.08197} & 1312.74 & 1204.90 & 1255.79 & \textbf{238.56} & 6.20 & 1.06 & 1.61 & \textbf{0.00} \\
 & 4 & 1.15686 & 1.16207 & 1.22160 & \textbf{1.15340} & 1649.63 & 1656.86 & 1504.31 & \textbf{572.74} & 12.04 & 1.42 & 5.47 & \textbf{0.00} \\
 & 5 & \textbf{1.16648} & 1.17881 & 1.21112 & 1.16688 & 1632.16 & 1577.41 & 1644.91 & \textbf{1269.00} & 221.80 & 2.77 & 5.60 & \textbf{1.05} \\
 & 6 & \textbf{1.09679} & 1.13246 & 1.20050 & 1.09843 & TLR & TLR & TLR & \textbf{1724.52} & 11.40 & 7.27 & 12.45 & \textbf{2.98} \\
 & 7 & 1.16676 & 1.17364 & 1.31607 & \textbf{1.15752} & TLR & \textbf{1575.86} & 1579.20 & 1739.10 & 219.30 & 3.13 & 13.85 & \textbf{1.64} \\
 & 8 & 1.31094 & 1.30386 & 1.63713 & \textbf{1.27765} & TLR & 1684.67 & 1661.99 & \textbf{1240.86} & 1126.71 & 8.02 & 26.52 & \textbf{3.57} \\
 & 9 & 1.04436 & 1.04715 & 1.45101 & \textbf{1.02183} & \textbf{TLR} & TLR & TLR & TLR & 14.07 & 6.11 & 28.25 & \textbf{3.49} \\
 & 10 & \textbf{1.06016} & 1.08312 & 1.38801 & 1.06168 & TLR & TLR & TLR & \textbf{1567.72} & 14.77 & 5.44 & 26.07 & \textbf{2.44} \\
 & 11 & \textbf{1.08755} & 1.11401 & 1.49809 & 1.09618 & \textbf{TLR} & TLR & TLR & TLR & 119.08 & 4.89 & 28.55 & \textbf{3.11} \\
 & 12 & \textbf{1.08884} & 1.19911 & 1.48523 & 1.14369 & TLR & TLR & TLR & \textbf{1576.00} & 123.47 & 8.37 & 26.12 & \textbf{2.84} \\
 & 13 & \textbf{0.88733} & 1.00737 & 1.41042 & 0.96489 & \textbf{TLR} & TLR & TLR & TLR & 8.14 & 7.45 & 26.72 & \textbf{3.82} \\
 & 14 & \textbf{0.92505} & 1.01825 & 1.70346 & 1.00792 & TLR & TLR & TLR & \textbf{1739.54} & 113.36 & 7.42 & 32.61 & \textbf{3.51} \\
 & 15 & 2.20333 & 2.29312 & 3.52375 & \textbf{0.97253} & \textbf{TLR} & TLR & TLR & TLR & 15.39 & 7.73 & 28.23 & \textbf{5.47} \\
 & 16 & \textbf{0.89375} & 1.06637 & 1.36038 & 0.97866 & \textbf{TLR} & TLR & TLR & TLR & 181.90 & 7.23 & 23.91 & \textbf{4.63} \\
 & 17 & \textbf{0.80364} & 0.86932 & 1.44146 & 0.86289 & \textbf{TLR} & TLR & TLR & TLR & 41.83 & 3.29 & 29.10 & \textbf{2.68} \\
 & 18 & \textbf{0.77309} & 0.93138 & 1.13899 & 0.91351 & TLR & TLR & TLR & \textbf{1752.57} & 5.57 & 5.87 & 20.36 & \textbf{3.42} \\
 & 19 & 2.79964 & 4.80577 & 7.19811 & \textbf{2.63461} & TLR & TLR & TLR & \textbf{1757.45} & 11.68 & 11.78 & 28.85 & \textbf{4.04} \\
 & 20 & 3.51991 & 3.74243 & 5.45435 & \textbf{2.91696} & TLR & 1720.00 & \textbf{1631.06} & 1718.24 & 80.07 & 5.79 & 32.91 & \textbf{2.61} \\
\cline{1-14}
\multirow[c]{20}{*}{0.5} & 1 & \textbf{1.08398} & \textbf{1.08398} & \textbf{1.08398} & \textbf{1.08398} & 7.36 & 7.93 & \textbf{3.30} & 5.65 & 0.01 & \textbf{0.00} & \textbf{0.00} & \textbf{0.00} \\
 & 2 & \textbf{1.10739} & \textbf{1.10739} & \textbf{1.10739} & \textbf{1.10739} & 687.02 & 53.92 & 65.37 & \textbf{45.40} & 5.93 & \textbf{0.00} & \textbf{0.00} & \textbf{0.00} \\
 & 3 & \textbf{1.09877} & 1.10299 & 1.10299 & 1.09924 & 1231.91 & 835.48 & 810.02 & \textbf{151.30} & 8.06 & 0.78 & 0.78 & \textbf{0.00} \\
 & 4 & \textbf{1.17812} & 1.18087 & 1.18684 & 1.17823 & 1045.44 & 673.46 & 848.34 & \textbf{169.69} & 5.24 & 0.42 & 0.99 & \textbf{0.00} \\
 & 5 & 1.15669 & 1.16182 & 1.19310 & \textbf{1.15150} & 1631.32 & 1535.99 & 1545.09 & \textbf{833.38} & 434.39 & 1.73 & 4.31 & \textbf{0.33} \\
 & 6 & 1.12451 & 1.14394 & 1.18943 & \textbf{1.11696} & TLR & TLR & TLR & \textbf{1371.63} & 21.04 & 5.72 & 9.75 & \textbf{1.99} \\
 & 7 & 1.16885 & 1.18468 & 1.28593 & \textbf{1.15879} & 1800.00 & 1367.26 & 1521.87 & \textbf{1037.38} & 485.69 & 3.97 & 10.16 & \textbf{0.56} \\
 & 8 & 1.31984 & 1.23629 & 1.61115 & \textbf{1.21985} & 1757.01 & 1312.71 & 1451.92 & \textbf{987.61} & 2120.18 & 3.48 & 19.61 & \textbf{1.50} \\
 & 9 & 1.05489 & \textbf{1.01059} & 1.49786 & 1.04348 & TLR & \textbf{1692.87} & TLR & 1703.18 & 25.63 & 5.23 & 29.67 & \textbf{2.15} \\
 & 10 & \textbf{1.06482} & 1.07143 & 1.50049 & 1.13253 & TLR & TLR & TLR & \textbf{1458.43} & 25.86 & 5.03 & 29.46 & \textbf{2.55} \\
 & 11 & 1.14481 & 1.16425 & 1.62094 & \textbf{1.13989} & TLR & TLR & TLR & \textbf{1474.52} & 215.05 & 7.34 & 31.89 & \textbf{2.35} \\
 & 12 & \textbf{1.13701} & 1.14284 & 2.01879 & 1.22036 & TLR & 1761.86 & 1779.82 & \textbf{1657.10} & 216.25 & 6.09 & 28.39 & \textbf{3.33} \\
 & 13 & \textbf{0.75005} & 0.96838 & 1.51021 & 0.93357 & TLR & TLR & \textbf{TLR} & TLR & 14.18 & 7.63 & 29.20 & \textbf{4.59} \\
 & 14 & \textbf{0.77957} & 1.03886 & 2.27411 & 0.97402 & TLR & TLR & TLR & \textbf{1726.02} & 198.45 & 12.30 & 41.80 & \textbf{5.40} \\
 & 15 & 2.94050 & 3.41132 & 5.23333 & \textbf{2.66294} & TLR & TLR & TLR & \textbf{1673.65} & 21.27 & 10.19 & 32.20 & \textbf{4.02} \\
 & 16 & \textbf{0.70172} & 0.98988 & 1.41011 & 0.91431 & TLR & TLR & TLR & \textbf{1456.95} & 358.91 & 9.29 & 27.89 & \textbf{3.46} \\
 & 17 & \textbf{0.58883} & 0.71334 & 1.47722 & 0.75661 & TLR & TLR & TLR & \textbf{1787.09} & 84.75 & 5.19 & 32.00 & \textbf{3.80} \\
 & 18 & \textbf{0.52446} & 0.75656 & 1.17983 & 0.79561 & TLR & TLR & TLR & \textbf{1744.16} & 12.31 & 8.99 & 27.09 & \textbf{6.51} \\
 & 19 & 7.64672 & 7.41469 & 12.32683 & \textbf{4.11329} & TLR & TLR & TLR & \textbf{1741.21} & 15.61 & 14.61 & 32.93 & \textbf{4.84} \\
 & 20 & 5.91176 & 6.33539 & 9.24639 & \textbf{0.81636} & TLR & 1695.92 & \textbf{1692.25} & TLR & 148.49 & 7.39 & 29.08 & \textbf{3.65} \\
\cline{1-14}
\multirow[c]{20}{*}{0.75} & 1 & \textbf{1.06916} & \textbf{1.06916} & \textbf{1.06916} & \textbf{1.06916} & 5.64 & 8.59 & \textbf{3.17} & 6.35 & 0.01 & \textbf{0.00} & \textbf{0.00} & \textbf{0.00} \\
 & 2 & \textbf{1.09688} & \textbf{1.09688} & \textbf{1.09688} & \textbf{1.09688} & 150.14 & 64.14 & 67.11 & \textbf{48.15} & 0.01 & \textbf{0.00} & \textbf{0.00} & \textbf{0.00} \\
 & 3 & \textbf{1.13331} & \textbf{1.13331} & \textbf{1.13331} & \textbf{1.13331} & 420.79 & 61.12 & \textbf{21.59} & 63.57 & 4.09 & \textbf{0.00} & \textbf{0.00} & \textbf{0.00} \\
 & 4 & \textbf{1.11885} & 1.12014 & 1.12545 & 1.11938 & 924.66 & 710.18 & 664.26 & \textbf{156.52} & 7.65 & 0.16 & 0.62 & \textbf{0.00} \\
 & 5 & 1.13391 & 1.13058 & 1.13055 & \textbf{1.12281} & 1632.00 & 1220.33 & 1178.31 & \textbf{472.97} & 664.45 & 1.04 & 1.02 & \textbf{0.23} \\
 & 6 & \textbf{1.08118} & 1.09257 & 1.13770 & 1.08174 & TLR & 1674.21 & 1552.25 & \textbf{1292.29} & 34.18 & 2.40 & 5.37 & \textbf{1.00} \\
 & 7 & 1.12975 & 1.12389 & 1.13115 & \textbf{1.11830} & 1800.00 & 1010.98 & 970.81 & \textbf{360.32} & 881.07 & 1.48 & 2.12 & \textbf{0.16} \\
 & 8 & 1.13502 & \textbf{1.12639} & 1.13347 & 1.13284 & 1725.88 & 922.05 & 1004.40 & \textbf{620.79} & 616.88 & \textbf{1.11} & 1.82 & 1.47 \\
 & 9 & 1.02646 & \textbf{1.00199} & 1.10871 & 1.01501 & TLR & TLR & 1767.38 & \textbf{755.29} & 36.69 & 8.85 & 23.06 & \textbf{1.06} \\
 & 10 & \textbf{1.03711} & 1.06779 & 1.42239 & 1.11413 & TLR & TLR & TLR & \textbf{1431.29} & 36.39 & 7.19 & 25.18 & \textbf{0.62} \\
 & 11 & 1.12342 & 1.14591 & 1.62332 & \textbf{1.11716} & TLR & TLR & TLR & \textbf{1106.89} & 293.88 & 7.55 & 31.58 & \textbf{1.45} \\
 & 12 & \textbf{1.07760} & 1.13723 & 1.36671 & 1.15214 & TLR & 1721.66 & 1645.78 & \textbf{723.05} & 377.97 & 11.50 & 21.98 & \textbf{0.26} \\
 & 13 & 0.94024 & \textbf{0.59146} & 1.24210 & 0.76857 & \textbf{TLR} & TLR & TLR & TLR & 35.66 & 3.79 & 27.60 & \textbf{3.06} \\
 & 14 & \textbf{0.75900} & 0.84168 & 1.27255 & 0.76369 & TLR & TLR & TLR & \textbf{1186.79} & 650.40 & 10.84 & 32.00 & \textbf{1.59} \\
 & 15 & 16.18649 & 16.69934 & 27.73834 & \textbf{3.00898} & TLR & TLR & TLR & \textbf{1365.55} & 28.23 & 7.26 & 28.84 & \textbf{2.67} \\
 & 16 & 0.58644 & 0.73353 & 0.95812 & \textbf{0.58406} & TLR & TLR & TLR & \textbf{1301.02} & 1038.00 & 13.45 & 24.76 & \textbf{2.73} \\
 & 17 & \textbf{0.42141} & 0.55061 & 0.72388 & 0.54802 & TLR & TLR & TLR & \textbf{1601.44} & 225.64 & 5.22 & 23.43 & \textbf{3.56} \\
 & 18 & \textbf{0.27476} & 0.28152 & 0.32659 & 0.27863 & \textbf{TLR} & TLR & TLR & TLR & 37.51 & 5.51 & 18.54 & \textbf{4.20} \\
 & 19 & 11.83949 & 19.55914 & 30.76335 & \textbf{8.54565} & TLR & TLR & TLR & \textbf{1381.50} & 25.20 & 17.93 & 35.64 & \textbf{4.11} \\
 & 20 & 9.08337 & 10.43514 & 15.64535 & \textbf{7.84095} & TLR & 1621.19 & 1565.19 & \textbf{1357.25} & 243.00 & 8.07 & 25.48 & \textbf{2.71} \\
\cline{1-14}
\multirow[c]{20}{*}{1.0} & 1 & \textbf{3098.32700} & \textbf{3098.32700} & \textbf{3098.32700} & \textbf{3098.32700} & 181.35 & 8.19 & \textbf{1.17} & 7.30 & 1576.27 & \textbf{0.00} & \textbf{0.00} & \textbf{0.00} \\
 & 2 & \textbf{6865.85800} & \textbf{6865.85800} & \textbf{6865.85800} & \textbf{6865.85800} & 211.71 & 23.41 & \textbf{2.83} & 20.50 & 8952.38 & \textbf{0.00} & \textbf{0.00} & \textbf{0.00} \\
 & 3 & \textbf{13816.60667} & \textbf{13816.60667} & \textbf{13816.60667} & \textbf{13816.60667} & 295.50 & 33.55 & \textbf{2.95} & 29.96 & 3.10 & \textbf{0.00} & \textbf{0.00} & \textbf{0.00} \\
 & 4 & \textbf{9934.80000} & \textbf{9934.80000} & \textbf{9934.80000} & \textbf{9934.80000} & 702.16 & 53.06 & \textbf{7.09} & 51.11 & 10.77 & \textbf{0.00} & \textbf{0.00} & \textbf{0.00} \\
 & 5 & \textbf{13695.72000} & \textbf{13695.72000} & \textbf{13695.72000} & \textbf{13695.72000} & 1433.40 & 61.35 & \textbf{19.29} & 56.42 & 28.88 & \textbf{0.00} & \textbf{0.00} & \textbf{0.00} \\
 & 6 & \textbf{14319.59000} & \textbf{14319.59000} & \textbf{14319.59000} & \textbf{14319.59000} & TLR & 75.46 & \textbf{30.99} & 64.80 & 43.56 & \textbf{0.00} & \textbf{0.00} & \textbf{0.00} \\
 & 7 & \textbf{22062.46000} & \textbf{22062.46000} & \textbf{22062.46000} & \textbf{22062.46000} & TLR & 77.38 & \textbf{46.03} & 62.71 & 51.29 & \textbf{0.00} & \textbf{0.00} & \textbf{0.00} \\
 & 8 & \textbf{8876.83400} & \textbf{8876.83400} & \textbf{8876.83400} & \textbf{8876.83400} & 1624.82 & 77.92 & \textbf{44.12} & 64.31 & 1289.57 & \textbf{0.00} & \textbf{0.00} & \textbf{0.00} \\
 & 9 & 28228.93167 & \textbf{27995.23500} & \textbf{27995.23500} & \textbf{27995.23500} & TLR & 155.24 & 275.05 & \textbf{70.50} & 35.64 & \textbf{0.00} & \textbf{0.00} & \textbf{0.00} \\
 & 10 & 20909.16000 & \textbf{20370.97444} & \textbf{20370.97444} & \textbf{20370.97444} & TLR & 140.12 & 227.69 & \textbf{67.03} & 49.85 & \textbf{0.00} & \textbf{0.00} & \textbf{0.00} \\
 & 11 & 21746.32750 & \textbf{21695.36000} & \textbf{21695.36000} & \textbf{21695.36000} & TLR & 366.63 & 937.43 & \textbf{71.25} & 1105.64 & \textbf{0.00} & 0.06 & 0.01 \\
 & 12 & 33952.92500 & \textbf{33949.77000} & \textbf{33949.77000} & \textbf{33949.77000} & 1799.98 & 123.14 & 323.16 & \textbf{66.11} & 1450.54 & \textbf{0.00} & \textbf{0.00} & \textbf{0.00} \\
 & 13 & 44945.87100 & 43098.60500 & 48287.74300 & \textbf{43098.60435} & TLR & 589.82 & 1476.29 & \textbf{86.89} & 52.11 & \textbf{0.00} & 13.54 & \textbf{0.00} \\
 & 14 & 31407.00200 & 32898.55400 & 40393.43799 & \textbf{29029.53100} & TLR & 1169.30 & 1332.30 & \textbf{93.02} & 1511.01 & 15.24 & 19.30 & \textbf{0.00} \\
 & 15 & 42279.75200 & 45475.15200 & 64708.47900 & \textbf{40527.99300} & TLR & 1247.36 & 1766.16 & \textbf{99.62} & 311.96 & 14.26 & 37.11 & \textbf{0.00} \\
 & 16 & 47490.91900 & 49334.45800 & 54347.68000 & \textbf{44503.74900} & TLR & 1222.40 & 1325.24 & \textbf{92.63} & 1805.29 & 14.29 & 15.78 & \textbf{0.00} \\
 & 17 & 36084.89222 & 36466.57333 & 54024.20667 & \textbf{33581.20889} & TLR & 1111.14 & 1727.34 & \textbf{92.72} & 1565.81 & 8.76 & 34.57 & \textbf{0.00} \\
 & 18 & 47058.80700 & 47496.28100 & 74651.44400 & \textbf{40888.69999} & TLR & 1452.08 & 1676.58 & \textbf{111.32} & 1045.68 & 13.46 & 44.17 & \textbf{0.00} \\
 & 19 & 33712.84100 & 36922.42400 & 66602.83000 & \textbf{29024.95700} & TLR & 1314.25 & TLR & \textbf{102.88} & 66.78 & 18.97 & 53.55 & \textbf{0.00} \\
 & 20 & 52712.91300 & 53078.60900 & 76349.92300 & \textbf{48569.68783} & TLR & 1520.83 & 1548.72 & \textbf{131.62} & 327.90 & 7.06 & 32.22 & \textbf{0.00} \\
\bottomrule
\end{tabular}
}
\end{tiny}
\caption{Detailed results of the sparse density sets. Values are averaged over instances within each set. ``TLR'' = Time Limit Reached ($\geq 1800$\,s). Best values are reported in bold.}
\label{tab:agg-results}
\end{table}

\clearpage
\newpage
\section{Monolithic MILP formulation for Flexible Job Shop with Machine States and Time-Of-Use}
\label{app:FJSP_monolithic_ilp}

The following monolithic MILP models the Flexible Job Shop with Machine states and TOU problem. The binary decision variable $x_{jt}^k = 1$ only if an operation $j \in \mathcal{T}$ starts on machine $k \in \mathcal{M}_j$ at time $t \in \mathcal{I}$.
Binary decision variables $z_{lm} = 1$ only if the machine is performing an optimal SPACES transition between the intervals $l$ and $m$. Let $C_{\max}$ be a continuous decision variable. We denote as $p_j^k$ the processing time of task $j \in \mathcal{T}$ on machine $k$, and  $k_0$ the energy-intensive machine.

\begin{small}
\begin{align}
&\min\  C_{\max} && \label{eq:fjsp_obj}\\
&s.t.\ & \sum_{t \in \mathcal{I}} \sum_{k \in \mathcal{M}_j} x^k_{jt} &= 1 &\quad \forall j \in \mathcal{T} \label{eq:fjsp_all_operations_scheduled} \\
&& \sum_{t \in \mathcal{I}} \sum_{k_j \in \mathcal{M}_j} x_{jt}^{k_j} - \sum_{t \in \mathcal{I}} \sum_{k_i \in \mathcal{M}_i} (t + p_i^{k_i}) x_{it}^{k_i} &\geq 0 &\quad \forall (i,j) \in \mathcal{T}\times\mathcal{T}, i \prec j \label{eq:fjsp_precedences}\\
&& \sum_{\substack{j \in \mathcal{T} | k \in \mathcal{M}_j}} \sum_{l = \max\{1, t - p_j^k+1\}}^t x^k_{jl} &\leq 1 &\quad \forall k \in \mathcal{M}, t \in \mathcal{I} \label{eq:fjsp_one_operation_per_machine}\\
&& \sum_{j \in \mathcal{J}} \sum_{t \in \text{StartTime} \cup \text{EndTime}_j} x_{jt}^{k_0} &= 0 &\label{eq:fjsp_eliminate_impossible_times} \\
&& \sum_{j \in \mathcal{J}} \sum_{t' = \max\{1, t - p_j^{k_0} + 1\}}^{t + 1} x_{jt'}^{k_0} + \sum_{l=1}^{t + 1} \sum_{m=t + 1}^{h + 1} z_{lm} &= 1 &\quad \forall t \in \mathcal{I} \label{eq:fjsp_machine_states}\\
&& \sum_{t \in \mathcal{I}} \sum_{k \in \mathcal{M}_j} (t + p_{j}^k) x_{jt}^k &\leq C_{\max} &\quad \forall j \in \mathcal{T} \label{eq:fjsp_cmax}
\end{align}
\end{small}
Constraints~\eqref{eq:fjsp_all_operations_scheduled} ensure all operations are scheduled. Constraints~\eqref{eq:fjsp_precedences} enforce the precedence constraints, while constraints~\eqref{eq:fjsp_one_operation_per_machine} prevent operations from overlapping on machines. Constraints~\eqref{eq:fjsp_eliminate_impossible_times} and \eqref{eq:fjsp_machine_states} model the machine state part of the problem. Finally, constraints~\eqref{eq:fjsp_cmax} monitor the makespan \eqref{eq:fjsp_obj} of the problem.


\end{document}